\documentclass[a4paper]{amsart}

\usepackage{amsmath}
\usepackage{amsfonts}
\usepackage{amssymb}
\usepackage{amsthm}
\usepackage{amstext}
\usepackage{tikz}
\usepackage{tikz-cd}
\usepackage{mathrsfs}
\usepackage{indentfirst}
\usepackage{hyperref}
\usepackage{theoremref}
\usepackage[inline]{enumitem}
\usepackage[latin1]{inputenc}
\usepackage{cite}
\usepackage{tabularx}
\usetikzlibrary{patterns}
\usetikzlibrary{backgrounds}
\usepackage{cleveref}
\usepackage{capt-of}

\newenvironment{psmallmatrix}{\big(\begin{smallmatrix}} {\end{smallmatrix}\big)}
\DeclareMathOperator{\vol}{vol}

\DeclareMathOperator{\relint}{relint}
\DeclareMathOperator{\Gr}{Gr}
\DeclareMathOperator{\aut}{Aut}
\DeclareMathOperator{\bl}{Bl}
\DeclareMathOperator{\ord}{ord}
\DeclareMathOperator{\rk}{rk}

\DeclareMathOperator{\codim}{codim}

\DeclareMathOperator{\Spec}{Spec}
\DeclareMathOperator{\Proj}{Proj}

\DeclareMathOperator{\Div}{div}
\renewcommand{\div}{\Div}

\newcommand{\bb}{\mathbb}
\newcommand{\union}{\cup}
\newcommand{\inter}{\cap}
\renewcommand{\frak}{\mathfrak}
\newcommand{\cal}{\mathcal}
\newcommand{\scr}{\mathscr}

\renewcommand{\phi}{\varphi}
\renewcommand{\epsilon}{\varepsilon}

\newcommand{\sub}{\subseteq}

\newcommand{\aff}[1]{\mathbb{A}^{#1}}
\newcommand{\proj}[1]{\mathbb{P}^{#1}}
\newcommand{\sheafhom}{\scr{H}\kern -1pt om}

\DeclareMathOperator{\diag}{diag}

\renewcommand{\sl}[1]{\text{SL}_{#1}}
\newcommand{\pgl}[1]{\text{PGL}_{#1}}

\DeclareMathOperator{\Hom}{Hom}

\newcommand{\bbk}{\Bbbk}

\DeclareMathOperator{\DF}{DF}

\newtheorem{theorem}{Theorem}[section]
\newtheorem{corollary}{Corollary}[section]
\newtheorem{lemma}{Lemma}[section]
\newtheorem{proposition}{Proposition}[section]
\theoremstyle{definition}
\newtheorem{definition}{Definition}[section]
\newtheorem*{example}{Example}
\theoremstyle{remark}
\newtheorem*{remark}{Remark}

\makeatletter
\pgfdeclarepatternformonly[\LineSpace]{my north east lines}{\pgfqpoint{-1pt}{-1pt}}{\pgfqpoint{\LineSpace}{\LineSpace}}{\pgfqpoint{\LineSpace}{\LineSpace}}%
{
	\pgfsetcolor{\tikz@pattern@color}
	\pgfsetlinewidth{0.4pt}
	\pgfpathmoveto{\pgfqpoint{0pt}{0pt}}
	\pgfpathlineto{\pgfqpoint{\LineSpace + 0.1pt}{\LineSpace + 0.1pt}}
	\pgfusepath{stroke}
}

\pgfdeclarepatternformonly[\LineSpace]{my north west lines}{\pgfqpoint{-1pt}{-1pt}}{\pgfqpoint{\LineSpace}{\LineSpace}}{\pgfqpoint{\LineSpace}{\LineSpace}}%
{
	\pgfsetcolor{\tikz@pattern@color}
	\pgfsetlinewidth{0.4pt}
	\pgfpathmoveto{\pgfqpoint{0pt}{\LineSpace}}
	\pgfpathlineto{\pgfqpoint{\LineSpace + 0.1pt}{-0.1pt}}
	\pgfusepath{stroke}
}
\makeatother

\newdimen\LineSpace
\tikzset{
	line space/.code={\LineSpace=#1},
	line space=15pt
}

\title{$K$-polystability of smooth Fano $\sl{2}$-threefolds}
\author{Jack Rogers}

\begin{document}
\maketitle
\begin{abstract}
	We prove the $K$-polystability of all smooth complex Fano threefolds admitting an effective action of $\sl{2}$ but not of a 2-torus or 3-torus. In particular, the existence of K\"{a}hler-Einstein metrics on varieties in the families (1.10), (1.15), (1.16), (1.17), (2.21), (2.27), (2.32), (3.13), (3.17), (3.25) and (4.6) of the Mori-Mukai classification of smooth Fano threefolds is proved.
\end{abstract}

\tableofcontents

\section{Introduction}

Since the solution of the Yau-Tian Donaldson conjecture by Chen-Donaldson-Sun \cite{cds}, showing that the existence of a K\"{a}hler-Einstein metric on a Fano manifold is equivalent to its $K$-polystability, there has been much progress in the theoretical study of $K$-stability in its various forms. However, there is still much work to be done when it comes to practical methods to actually check the $K$-(semi/poly)stability of a given variety.

One of the more fruitful approaches in this direction is to consider varieties equipped with group actions, so that we can use the equivariant $K$-stability introduced by Datar-Sz\'{e}kelyhidhi \cite{ds}. Exploiting these symmetries can make verifying $K$-stability far more straightforward than it is in general via the analysis of associated combinatorial data. For example, Wang-Zhu proved that a smooth Fano toric variety is $K$-polystable if and only if the barycentre of its associated dual polytope is the origin.

One of the key invariants determining the viability of this method is called the \emph{complexity} of the group action. Suppose a reductive algebraic group $G$ acts on a normal variety $X$. Let $B \sub G$ be a Borel subgroup of $G$. The \emph{complexity} $c_G(X)$ of the action of $G$ on $X$ is the minimal codimension in $X$ of the orbits of $B$, or equivalently the transcendence degree of the field of $B$-invariant rational functions on $X$ over the base field.

Criteria for $K$-stability using combinatorial methods have been found in the toric case by Wang-Zhu \cite{wz}, in the case of complexity one $T$-varieties by Ilten-Suess \cite{is}, and for complexity zero varieties under general reductive groups (\emph{spherical varieties}) by Delcroix \cite{del}.
 
In this paper we provide the first results in a project to find a combinatorial criterion for the $K$-stability of complexity-one $G$-varieties for general reductive groups $G$. Specifically, we consider the simplest nontrivial examples of such varieties, the smooth Fano threefolds admitting effective actions of $\sl{2}$. 

Using the classification of Cheltsov-Przyjalkowski-Shramov \cite{cps} of smooth Fano threefolds with infinite automorphism groups, it is possible to identify all such threefolds admitting effective $\sl{2}$-actions. After ruling out those whose automorphism groups are non-reductive (since these are known not to be $K$-polystable by Matsushima's criterion \cite{matsu}) and those also known to be toric or to admit the effective action of a 2-torus (since the $K$-polystability of these varieties is already checkable by the previously mentioned results of Wang-Zhu and Ilten-S\"{u}{\ss}), there are seven smooth Fano $\sl{2}$-threefolds remaining. They are listed below by the numbers and descriptions given in the paper of Cheltsov-Przyjalkowski-Shramov. The cases denoted with daggers consist of families of varieties, only some of which admit effective $\sl{2}$-actions.\\

\renewcommand{\arraystretch}{1.5}
\begin{center}
	\begin{tabularx}{\textwidth}{ |l|X|} 
		\hline
		1.10$^\dagger$ & $V_{22}$, a zero locus of three sections of the rank 3 vector bundle $\bigwedge^2 \cal{Q}$, where $\cal{Q}$ is the universal quotient bundle on $\Gr{(3,7)}$  \\ 
		\hline
		1.15 & $V_5$, a section of $\Gr{(2,5)} \sub \proj{9}$ by a linear subspace of codimension 3  \\ 
		\hline
		2.21$^\dagger$ & The blow up of a quadric threefold  $Q \sub \proj{4}$ along a twisted quartic curve  \\
		\hline
		2.27 & The blow up of $\proj{3}$ along a twisted cubic curve  \\
		\hline
		3.13$^\dagger$ & The blow up of a divisor $W \sub \proj{2} \times \proj{2}$ of bidegree $(1,1)$ along a curve of bidegree $(2,2)$ which is mapped to irreducible conics by the natural projections to $\proj{2}$  \\
		\hline
		3.17 & A divisor on $\proj{1} \times \proj{1} \times \proj{2}$ of tridegree $(1,1,1)$, or a blow-up of $\proj{1} \times \proj{2}$ along a curve of bidegree $(1,1)$  \\
		\hline
		4.6 & The blow up of $\proj{3}$ along a disjoint union of three lines  \\
		\hline
	\end{tabularx}
\end{center}

\vspace{5pt}

Our main result is:

\begin{theorem}
	The smooth Fano threefolds, (1.10), (1.15), (1.16), (1.17), (2.27), (2.32), (3.17), (3.25) and (4.6) in the Mori-Mukai classification are $K$-polystable and hence admit K\"{a}hler-Einstein metrics. The families (2.21) and (3.13) each contain a $K$-polystable variety admitting a K\"{a}hler-Einstein metric.
\end{theorem}

\begin{remark}
	The same result were recently obtained independently by other authors using different methods, see \cite{sc,cal}. The $K$-polystability of the Mukai-Umemura threefold in the family (1.10) was already known by Donaldson \cite{don1}, and the $K$-polystability of $V_5$ (1.15) was known by Cheltsov-Shramov \cite{cs}. 
\end{remark}

The rest of this paper is organised as follows: in Section 2 we recall the definition of $K$-polystability and its equivariant version. In Section 3 we discuss the theory of varieties with actions of complexity one and the combinatorial description of these varieties. In Section 4, we state the technical result (\thref{Kstablecentral}) which allows us to prove the main theorem stated above. The remaining sections provide a proof of \thref{Kstablecentral}, calculations of the combinatorial data of the varieties in question, and demonstrations that \thref{Kstablecentral} does indeed imply their $K$-polystability.

\section{Equivariant $K$-stability}

\subsection{Test configurations}

We begin by recalling the definitions of $K$-(semi/poly)stability and setting conventions, then describe the equivariant version due to Datar and Sz\'{e}kelyhidi.

\begin{definition}
	Let $(X,L)$ be a complex polarised variety and let $m > 0$. A \emph{test configuration} for $(X,L)$ of \emph{exponent} $m$ consists of: 
	\begin{itemize} 
		\item A flat morphism of schemes $\pi \colon \cal{X} \to \bb{A}_{\bb{C}}^1$; 
		\item A $\pi$-relatively ample line bundle $\cal{L} \to \cal{X}$; 
		\item A $\bb{C}^{\times}$ action on $\cal{X}, \cal{L}$; 
	\end{itemize}
	such that $\pi$ and the bundle map $\cal{L} \to \cal{X}$ are $\bb{C}^{\times}$-equivariant for the standard action of $\bb{C}^{\times}$ on $\bb{A}_{\bb{C}}^1$ by multiplication, and for some (and hence all by equivariance) $t \neq 0$ in $\bb{A}_{\bb{C}}^1$, the pair $(X_t,L_t):= (\pi^{-1}(t),\cal{L}\vert_{\pi^{-1}(t)})$ is isomorphic to $(X,L^{\otimes m})$. We also require that the \emph{central fibre} $X_0$ is irreducible.
	
	We call $(\cal{X},\cal{L})$ a \emph{product configuration} if $\cal{X} \cong X \times \aff{1}$ and a \emph{trivial configuration} if it is a product configuration and the $\bb{C}^{\times}$-action is trivial on $X$.
\end{definition}

Note that since $0 \in \aff{1}_{\bb{C}}$ is fixed by the standard $\bb{C}^\times$ action, the morphism $\pi$ induces a $\bb{C}^\times$ action on the central fibre $X_0$ and the line bundle $L_0$. Also note that Fano varieties are polarised by their anticanonical bundle, so the notion of test configuration makes sense in this case.

In the general case of a projective scheme $Z$ with an ample line bundle $\Lambda$, we can consider the vector spaces $H^k = H^0(Z,\Lambda^{\otimes k})$ of global sections of the tensor powers of $\Lambda$. Let $d_k :=\dim{H^k}$. For $k$ large enough that $\Lambda^{\otimes k}$ is very ample, the $d_k$ are known to be given by a Hilbert polynomial of degree $n = \dim{Z}$. Now suppose there is a $\bb{C}^{\times}$-action on the pair $(Z,\Lambda)$. This induces a $\bb{C}^{\times}$ action on each $H^k$. Let $w_k$ be the sum of the weights of this action, or equivalently the weight on the top exterior power. Then for $k$ large enough, $w_k$ is also given by a polynomial, this being of degree $n+1$ \cite{don}. Now set $F(k) = w_k/kd_k$, so that there is an expansion for large $k$ given by: \[F(k) = F_0 + F_1k^{-1} + F_2k^{-2} + \ldots.\]

\begin{definition}
	The \emph{Donaldson-Futaki invariant} of $(Z,\Lambda)$ is the coefficient $F_1$ in the above expansion. For a test configuration $(\cal{X},\cal{L})$ of a polarised variety $(X,L)$, we define $\DF{(\cal{X},\cal{L})}$ to be the Donaldson-Futaki invariant of the central fibre $(X_0,L_0)$.
\end{definition}

\begin{definition}
	A polarised variety $(X,L)$ is:\begin{itemize} \item \emph{$K$-semistable} if $\DF{(\cal{X},\cal{L})} \geq 0$ for every test configuration $(\cal{X},\cal{L})$ on $(X,L)$; \item \emph{$K$-polystable} if it is $K$-semistable and $\DF{(\cal{X},\cal{L})} = 0$ only for product configurations; \item \emph{$K$-stable} if it is $K$-semistable and $\DF{(\cal{X},\cal{L})} = 0$ only for trivial configurations; \item \emph{$K$-unstable} if it is not $K$-semistable.\end{itemize}
\end{definition}

An important result of Li-Xu immediately allows us to restrict the set of test configurations we need to check in order to verify $K$-(semi/poly)stability:

\begin{definition}
	A test configuration $(\cal{X},\cal{L})$ for a polarised variety $(X,L)$ is called \emph{special} if the central fibre $X_0$ is normal.
\end{definition}

\begin{theorem}\cite{lx}
	For a Fano variety $(X,-K_X)$, $K$-(poly/semi)stability can be verified by checking the Donaldson-Futaki invariant of only the special test configurations.
\end{theorem}

Interest in $K$-stability mostly stems from the solution of the Yau-Tian-Donaldson conjecture by Chen-Donaldson-Sun, i.e.:

\begin{theorem}\cite{cds}
	A smooth complex Fano variety $X$ admits a K\"{a}hler-Einstein metric if and only if $(X,K_X^{-1})$ is $K$-polystable.
\end{theorem}

Unfortunately, since there are generally infinitely many test configurations for a given polarised variety, it is very difficult to check $K$-(poly/semi)stability in the general case, even accounting for the Li-Xu theorem. 

The $\alpha$-invariant of Tian \cite{tian2} was for a long time one of the only practical methods to check $K$-(poly)stability. Recently, work of Abban-Zhuang \cite{az1,az2} has provided other more powerful methods of verification. Otherwise, most progress on this front has come from the equivariant perspective due to the work of Datar-Sz\'{e}kelyhidi \cite{ds}, which we summarise now.

\begin{definition}
	Let $X$ be a $G$-variety, and let $\pi \colon L \to X$ be a line bundle on $X$. We say that $L$ is \emph{$G$-linearised} if there is a $G$-action on $L$ such that $\pi$ is $G$-equivariant and the map $\pi^{-1}(x) \to \pi^{-1}(g\cdot x)$ induced on the fibres is linear for all $g \in G$ and all $x \in X$.
\end{definition}

\begin{definition}
	Let $G$ be a reductive algebraic group and let $(X,L)$ be a polarised variety with a $G$-action on $X$ such that $L$ is $G$-linearised. A test configuration $(\cal{X},\cal{L})$ of exponent $m$ is \emph{G-equivariant} if there is a $G$-action on $(\cal{X},\cal{L})$ which commutes with the $\bb{C}^{\times}$ action and such that the isomorphisms between $(X,L^{\otimes m})$ and $(X_t,L_t)$ for $t \neq 0$ are $G$-equivariant. Then $(X,L)$ is \emph{equivariantly $K$-(poly/semi)stable} if it is $K$-(poly/semi)stable with respect to $G$-equivariant special test configurations.
\end{definition}	

The main result of Datar-Sz\'{e}kelyhidi is the following:

\begin{theorem}\cite{ds}\thlabel{datszek}
	Let $G$ be a reductive algebraic group and let $X$ be a smooth complex Fano $G$-variety. Then $(X,K_X^{-1})$ is equivariantly $K$-polystable if and only if $X$ admits a K\"{a}hler-Einstein metric.
\end{theorem}

\begin{remark}
	We should mention that the result of Datar-Sz\'{e}kelyhidi has been generalised to the singular case when $G$ is finite by Liu-Zhu \cite{lz} and for general reductive groups by Zhuang \cite{zhuang}. Specifically, Zhuang uses a purely algebraic argument showing (among other things) that $K$-polystability of a log Fano pair $(X,\Delta)$ is equivalent to $G$-equivariant $K$-polystability when $G$ is reductive. 
\end{remark}

\subsection{$\beta$ invariant}

Here we discuss an invariant introduced by Fujita \cite{fuj} and Li \cite{li} which they have shown to have an intimate connection to $K$-stability. We must first include some preliminary definitions.

\begin{definition}
	Let $\delta$ be a Cartier divisor on a smooth projective variety $X$ of dimension $d$. The \emph{volume} of $\delta$ is \[\vol{\delta} = \limsup_{n \to \infty}{\frac{\dim{H^0(X,\cal{O}(\delta)^{\otimes n})}}{n^d/d!}}.\]
\end{definition}

In fact, by \cite[Ex. 11.4.7]{lazII} the lim sup is actually a limit. One can also check that if $\delta$ is ample, then $\vol{\delta} = \delta^d$. \begin{definition}
	Let $X$ be a Fano variety. If $\sigma \colon Y \to X$ is any projective birational morphism with $Y$ normal, we call a prime divisor $F \sub Y$ a \emph{prime divisor over $X$}.
\end{definition}

\begin{proposition}\thlabel{primetc}
	Let $X$ be a smooth complex Fano projective variety. There is a bijective correspondence between prime divisors over $X$ and test configurations on $X$ (excluding the trivial test configuration).
\end{proposition}

\begin{proof}\cite[Lemma 3.7]{xu}
	Let $F \sub Y \to X$ be a prime divisor over $X$ and let $-K_X$ be the anticanonical divisor of $X$. Consider the section ring $R = \bigoplus_{k \in \bb{Z}}{H^0(X,-kK_X)}$ of $-K_X$. The prime divisor $F$ induces a filtration \[\cal{F}^rR := \bigoplus_{k \in \bb{Z}}{\{f \in H^0(X,-kK_X) \mid \nu_F(f) \geq r\}}\] of $R$. Consider the Rees algebra \[\cal{A} = \bigoplus_{r \in \bb{Z}}{\cal{F}^rR\cdot z^{-r}}\] of this filtration. Setting $r = -1$, we see that the $k = 0$ piece of $\cal{F}^rR$ is $\bb{C}$. It follows that $z \in \cal{A}_{(0,-1)}$, so there is an embedding $\bb{C}[z] \hookrightarrow \cal{A}$. This induces a morphism $ \cal{X} = \Proj{\cal{A}} \to \aff{1}$, which is automatically compatible with the standard $\bb{C}^\times$ action on $\aff{1}$. Furthermore, since $R = \cal{A}/(z-1)$ we see that $X = \Proj{R}$ embeds into $\cal{X} = \Proj{\cal{A}}$ as a closed subscheme and is the preimage of $1 \in \aff{1}$ under the given morphism. Hence $(\cal{X},\cal{O}(1)_{\cal{X}})$ is indeed a test configuration on $(X,-K_X)$.
	
	Conversely, given a test configuration $(\cal{X},\cal{L})$ on $(X,-K_X)$, the special fibre $X_0 \sub \cal{X}$ is a prime divisor, having a corresponding valuation $\nu_0$ on $\bbk(\cal{X}) = \bbk(X\times \aff{1})$. We can thus restrict $\nu_0$ to $\bbk(X)$. Since the restriction of a geometric valuation is itself geometric, there exists some model $Y$ of $\bbk(X)$ with a prime divisor $F \sub Y$ such that $\nu_F = \nu_0\vert_{\bbk(X)}$, i.e. a prime divisor over $X$.
\end{proof}

\begin{definition}\thlabel{assgr}
	Let $\cal{F}^r$ be a filtration of an algebra $R$. The \emph{associated graded ring} to the filtration $\cal{F}^r$ is the ring \[\cal{B} = \bigoplus_{r \in \bb{Z}}{\cal{F}^r/\cal{F}^{r+1}}.\]
\end{definition}

\begin{remark}
	In the notation of \thref{primetc}, the associated graded ring to the filtration $\cal{F}^r$ on $R$ is $\cal{A}/(z)$. The ideal $(z) \sub \bb{C}[z]$ corresponds to $0 \in \aff{1}$, so $\cal{A}/(z)$ corresponds to the fibre over $0$ of the morphism $\Proj{\cal{A}} \to \aff{1}$, i.e. the central fibre of the corresponding test configuration.
\end{remark}

\begin{definition}
	Let $X$ be a Fano variety and let $F \sub Y \xrightarrow{\sigma} X$ be a prime divisor over $X$. The \emph{log discrepancy of $F$ over $X$} is $A_X(F) = \ord_F(K_{Y/X}) + 1$, where $K_{Y/X} = K_Y - \sigma^*(K_X)$ is the \emph{relative canonical divisor}. 
\end{definition}

The usual definition of log discrepancy is more general, but we will only need the one above. For full details see e.g. \cite[\S 2]{km}. A useful consequence of this definition is that it is not hard to see that if $\sigma$ is a sequence of $n$ nested blow-ups, of which $F$ is the final exceptional divisor, then $A_X(F) = n + 1$.

\begin{definition}
	Let $X$ be a smooth complex Fano variety of dimension $n$. Let $F \sub Y$ be a prime divisor over $X$. The \emph{$\beta$-invariant} of $F$ over $X$ is \[\beta_X(F) = A_X(F)(-K_X)^n - \int_{0}^{\infty}{\vol(-K_X-xF)dx},\] where $\vol(-K_X-xF)$ is shorthand for $\vol(\sigma^*(-K_X)-xF)$, where $\sigma \colon Y \to X$.
\end{definition}

\begin{theorem}\cite{li,fuj}
	A smooth complex Fano variety $X$ is $K$-semistable if and only if $\beta_X(F) \geq 0$ for any prime divisor $F$ over $X$, and $K$-stable if and only if the inequality is always strict.
\end{theorem}

The proof of the above essentially amounts to the fact that, under the correspondence described in \thref{primetc}, the $\beta$-invariant of a prime divisor $F$ over $X$ is a positive multiple of the Donaldson-Futaki invariant of the corresponding test configuration. This gives a $K$-\emph{polystability} criterion as well:

\begin{corollary}
	A smooth complex Fano variety $X$ is $K$-polystable if and only if $\beta_F(X) \geq 0$ for any prime divisor $F$ over $X$, and $\beta_F(X) = 0$ only when $F$ corresponds to a product test configuration.
\end{corollary}

In our perspective when $X$ is equipped with a group action, we have:

\begin{corollary}
	If $X$ comes equipped with the action of a reductive group $G$, then we can verify the $K$-polystability of $X$ by checking the $\beta$-invariant for prime divisors $F \sub Y \to X$ in the case where $Y$ also has a $G$-action, the morphism $Y \to X$ is $G$-equivariant, and $F$ is $G$-invariant in $Y$.
\end{corollary}

\begin{proof}
	This follows by combining the theorem of Fujita-Li with the theorem of Datar-Sz\'{e}kelyhidi.
\end{proof}

It is this final corollary which we will ultimately use to prove our main theorem.

\section{Actions of Complexity One}

\subsection{Notation and Conventions}

In this section we will summarise the theory of group actions of complexity one on normal varieties. This is work mainly completed by Timashev \cite{tim1,tim2} and proofs of all results in this section can be found in \cite{tim}. This theory is complicated and for reasons of space our coverage of it is quite brief.

Fix an algebraically closed field $\bbk$ of characteristic 0, a finitely generated extension $K$ of $\bbk$, a connected and reductive affine algebraic group $G$ and a Borel subgroup $B$ of $G$. Let $G$ act on $K$ such that $K$ is the function field common to a $G$-birational class of normal $G$-varieties. We call such a $G$-variety a \emph{$G$-model} of $K$ and use this theory to classify the $G$-models of $K$ up to $G$-isomorphism.

Denote by $K^B$ the set of $B$-invariants of $K$ under the $G$-action, and by $K^{(B)}$ the set of $B$-\emph{semi-invariants}, that is elements of $K$ on which $B$ acts via a character $\lambda \in \frak{X}(B)$, the character group of $B$. For fixed $\lambda$ let $K^{(B)}_\lambda$ denote the semi-invariants of weight $\lambda$. Let $\Lambda \sub \frak{X}(B)$ denote the (free abelian) subgroup consisting of $\lambda$ such that $K^{(B)}_\lambda$ is nonzero. We call $\Lambda$ the \emph{weight lattice}.

Let $\cal{D}$ denote the set of non-$G$-invariant divisors on all $G$-models of $K$ and let $\cal{D}^B$ denote the subset of $\cal{D}$ consisting of $B$-stable divisors. Elements of $\cal{D}^B$ are called \emph{colours}. Let $K_B \sub K$ denote the subalgebra of elements of $K$ with $B$-stable divisor of poles.

We call a discrete valuation $\nu$ of $K$ \emph{geometric} if it is a positive rational multiple of the valuation corresponding to a prime divisor $D$ on some model of $K$. A discrete valuation of $K$ is \emph{$G$-invariant} if $\nu(g\cdot f) = \nu(f)$ for all $g \in G$ and $f \in K$. A \emph{$G$-valuation} of $K$ is a $G$-invariant geometric valuation of $K$. Every $G$-valuation of $K$ is a positive rational multiple of a valuation corresponding to a $G$-invariant prime divisor on some $G$-model of $K$. Denote by $\cal{V}$ the set of $G$-valuations of $K$.

\subsection{Luna-Vust theory}

The Luna-Vust theory \cite{lv}, further developed by Knop \cite{knop,knop1} and Timashev \cite{tim1}, is a deep theory which allows us to classify up to isomorphism homogeneous spaces of algebraic groups and embeddings thereof in algebraic varieties. It contains within it the classification of toric varieties by fans and of spherical varieties by coloured fans as special cases. We give a very broad overview of its working here, in the notation of the previous section.

The Luna-Vust theory classifies varieties in terms of \emph{$G$-germs}, which are essentially $G$-stable subvarieties, and \emph{$B$-charts}, which are affine open $B$-stable subsets. Associated to these are sets of valuations called \emph{coloured data} corresponding to the $G$- and $B$-stable divisors containing a given $G$-germ or intersecting a given $B$-chart. The coloured data satisfy certain admissibility conditions, and ultimately determine the $G$-germs and hence the $G$-models up to isomorphism. 
 
\begin{definition}
	A \emph{$G$-germ} of $K$ is a local ring $\cal{O}_{X,Y} \sub K$ of some $G$-stable subvariety $Y$ on some $G$-model $X$ of $K$. A \emph{geometric realisation} of a $G$-germ is a $G$-model $X$ containing a $G$-subvariety $Y$ with the corresponding local ring. The \emph{support} of a $G$-germ $\cal{O}_{X,Y}$ is the set of $G$-valuations of $K$ having centre on $Y$ in any geometric realisation of $\cal{O}_{X,Y}$, i.e. $G$-valuations $\nu$ such that $\cal{O}_{\nu}$ dominates $\cal{O}_{X,Y}$.
\end{definition}

We will often conflate a $G$-germ $\cal{O}_{X,Y}$ with a $G$-subvariety $Y$ in a geometric realisation.

\begin{proposition}\cite[\S 12]{tim}
	The support of a $G$-germ is non-empty.	
\end{proposition}

The following is a version of the valuative criterion of separation:

\begin{proposition}\cite[Thm 12.2]{tim}
	The supports of all $G$-germs realised by a fixed $G$-model $X$ are pairwise disjoint.
\end{proposition}

The Luna-Vust theory is capable of classifying $G$-models up to isomorphism by the following:

\begin{theorem}\cite[\S 1.1]{tim1}
	A $G$-model $X$ of $K$ is uniquely determined among $G$-models of $K$ by its set of $G$-germs.
\end{theorem}

We keep track of $G$-germs using \emph{$B$-charts}:

\begin{definition}
	A \emph{$B$-chart} of $K$ is a $B$-stable affine open subset on some $G$-model of $K$.
\end{definition}

\begin{proposition}\cite[Lemma 1.1]{tim1}
	Any $G$-germ $\cal{O}_{X,Y}$ admits a geometric realisation $X$ such that $Y$ intersects some $B$-chart $X_0 \sub X$. We can therefore cover any $G$-model by the $G$-translates of finitely many $B$-charts.
\end{proposition}

$B$-charts are associated with the valuations of the $G$- and $B$-invariant prime divisors they intersect:

\begin{definition}
	Let $X_0$ be a $B$-chart. Let $\cal{W} \sub \cal{V}$ be the set of $G$-valuations corresponding to $G$-invariant prime divisors which intersect $X_0$ and let $\cal{R} \sub \cal{D}^B$ be the set of valuations corresponding to colours which intersect $X_0$. We call the pair $(\cal{W},\cal{R})$ the \emph{coloured data} of $X_0$.
\end{definition}

\begin{proposition}\cite[\S 13]{tim}
	$B$-charts are determined by their coloured data.
\end{proposition}

$G$-germs also have coloured data:

\begin{definition}
	Let $Y$ be a $G$-germ. Let $\cal{V}_Y \sub \cal{V}$ be the set of $G$-valuations corresponding to $G$-invariant prime divisors containing $Y$, and let $\cal{D}^B_Y \sub \cal{D}^B$ be the set of valuations corresponding to colours containing $Y$. We call the pair $(\cal{V}_Y,\cal{D}^B_Y)$ the \emph{coloured data} of $Y$.
\end{definition}

\begin{proposition}\cite[Prop 14.1]{tim}
	$G$-germs are determined by their coloured data.
\end{proposition}

In short, the Luna-Vust theory describes properties which must be satisfied by subsets of $\cal{V}$ and $\cal{D}^B$ in order for them to correspond to the coloured data of some $B$-chart or $G$-germ, allows us to compute the supports of $G$-germs from their coloured data, and provides a method to construct varieties from sets of coloured data. We can then distinguish between nonisomorphic varieties by differences in the coloured data of their $G$-germs. We can also verify properties of varieties from their coloured data, for example, the valuative criterion of completeness becomes:

\begin{proposition}
	A $G$-model $X$ of $K$ is complete if and only if the supports of its $G$-germs cover $\cal{V}$.
\end{proposition}

We will show in the remainder of this section how the Luna-Vust theory is applied to give a combinatorial description of varieties of complexity one.

\subsection{Hyperspace}

Here we introduce the hyperspace, which will be home of certain collections of cones which will classify complexity one varieties. These cones will be generated from the coloured data of each variety, i.e. by $G$-valuations and valuations corresponding to colours. Hence we address some properties of these first:

\begin{proposition}\cite[Cor 19.13]{tim}
	A $G$-valuation of $K$ is uniquely determined by its restriction to $K^{(B)}$.	
\end{proposition}

\begin{proposition}\cite[\S 13.1]{tim}
	There is a split exact sequence of abelian groups \[0 \to (K^B)^\times \to K^{(B)} \to \Lambda \to 0.\]
\end{proposition}

\begin{corollary}
	A $G$-valuation of $K$ is uniquely determined by its restriction to $K^B$ and a functional on $\Lambda$.
\end{corollary}

\begin{proof}
	Fix a map $e \colon \Lambda \to K^{(B)}$ splitting the exact sequence above by sending a weight $\lambda$ to a semi-invariant $e_\lambda \in K^{(B)}_\lambda$. Then the exact sequence tells us that $K^{(B)} = (K^B)^\times \oplus e(\Lambda)$. We know that a $G$-valuation $\nu$ is determined by its restriction to $K^{(B)}$, and this in turn is given by its restriction to $K^B$ and a functional $\ell \colon \Lambda \to \bb{Q}$ given by $\ell(\lambda) = \nu(e_\lambda)$.
\end{proof}

In complexity one, $K^B$ has transcendence degree 1 over $\bbk$ and hence is the function field common to a birational class of curves containing a unique smooth projective curve $C$. Then any geometric valuation of $K^B$ is of the form $\nu = h\nu_x$ where $h$ is a nonnegative rational number and $x \in C$.  Hence $G$-valuations of $K$ correspond to triples $(x,h,\ell)$ where $x \in C$, $h \in \bb{Q}_{\geq 0}$ and $\ell \in \Hom(\Lambda,\bb{Q}) = \Lambda^*$, and are uniquely determined by this restriction up to the equivalence relation whereby $(x,h,\ell) \sim (x^\prime,h^\prime,\ell^\prime)$ if and only if we have equality or $h = h^\prime = 0$ and $\ell = \ell^\prime$. Let \[\cal{H} = \bigcup_{x \in C}{(\{x\} \times \bb{Q}_{\geq 0} \times \Lambda^*}/\sim\] and call $\cal{H}$ the \emph{hyperspace} of $K$. 

We have an injective map $\kappa\colon \cal{V} \to \cal{H}$ by the above, and can also map $\cal{D}^B$ to $\cal{H}$ using $\kappa$, although this is not necessarily injective. We can thus view the coloured data of a $G$-model of $K$ as sitting within $\cal{H}$, and the classification of complexity one varieties comes down to classifying collections of certain types of cones within $\cal{H}$, described in a later subsection.

Denote by $\cal{H}_x$ the subset of $\cal{H}$ consisting of points with first co-ordinate $x$. We call this the \emph{slice} of hyperspace corresponding to $x$. Likewise set $\cal{V}_x = \cal{V} \inter \cal{H}_x$ etc. The subset of $\cal{H}$ consisting of points with $h = 0$ is called the \emph{central hyperplane} and denoted by $\cal{Z}$.

\subsubsection{Splitting Maps}

The choice of the map $e \colon \Lambda \to K^{(B)}$ is arbitrary, so we must keep track of what happens if a different map is chosen. If a valuation has co-ordinates $(x, h, \ell)$ under $e$, and co-ordinates $(x^\prime,h^\prime,\ell^\prime)$ under another splitting $e^\prime$, then we have $x = x^\prime$, $h = h^\prime$ since these depend only on the restriction to $K^B$, and we have \[\ell^\prime(\lambda) = \ell(\lambda) + h\nu_x(e^\prime_\lambda/e_\lambda)\] for each $\lambda \in \Lambda$. Thus a different choice of splitting introduces an `integral shift' to the $\ell$-co-ordinates of the hyperspace.

Since $e^\prime_\lambda/e_\lambda$ is a rational function on the smooth projective curve $C$, the corresponding principal divisor has degree zero, so we have $0 = \sum_{x \in C}{\nu_x(e^\prime_\lambda/e_\lambda)}$ for each $\lambda$. Therefore when we introduce these integral shifts, they must balance each other out. When $C = \proj{1}$, which we will see later is always the case for us, any such collection of balanced integral shifts corresponds to a different choice of splitting.

\subsubsection{Quasihomogeneous and One-Parameter Cases}

There are two types of complexity one varieties depending on the relationship between $K^B$ and $K^G$. In the \emph{quasihomogeneous case} when $K^G = \bbk$, any $G$-model $X$ of $K$ has an open $G$-orbit containing a one-parameter family of codimension 1 $B$-orbits. This open orbit is an embedding of some homogeneous space $G/H$, which provides a minimal model for the $G$-birational class of varieties determined by $K$. In the \emph{one-parameter case}, where $K^G = K^B$, $G$-models contain a one-parameter family of codimension 1 $G$-orbits, each of which contains a $B$-orbit as an open subset. 

In all examples of interest to us, we are in the quasihomogeneous case, so we assume from now on that $K^G = \bbk$. 

\begin{proposition}
	In the quasihomogeneous case, the smooth projective curve $C$ is in fact $\proj{1}$.
\end{proposition}

\begin{proof}\cite[\S 16.2]{tim}
		Note that $G$ is a rational variety. Now since $G/H \sub X$ is an open orbit, we have $K = \bbk(G/H)$. Since $K^B = \bbk(C) \sub K$, we have a dominant rational map $G/H \dashrightarrow C$ and hence a dominant rational map $G \dashrightarrow C$ via the quotient. It follows that $C$ is unirational and hence equal to $\proj{1}$ by L\"{u}roth's theorem.
\end{proof}

\subsubsection{Regular, Subregular and Central Divisors}

A complexity one $G$-model $X$ comes with a dominant rational $B$-quotient map $\pi \colon X \dashrightarrow C$ arising from the inclusion $K^B = \bbk(C) \sub K$, which separates general $B$-orbits. There is a one-parameter family of $B$-stable prime divisors (i.e. colours) $D_x \sub \pi^*(x)$ in $X$ parameterised by points $x \in \pi(X)$, an open subset of $C$.

The behaviour of prime divisors on $X$ under $\pi$ determines to some extent the position in hyperspace of their corresponding valuations, and we introduce a typology of divisors based on this.

First, the choice of splitting $e$ marks out certain colours as being \emph{distinguished}, namely those colours lying in $\div(e_\lambda)$ for some $\lambda \in \Lambda$. They have $\ell(\lambda) = \nu_D(e_\lambda) \neq 0$ for some $\lambda$. Colours lying outside all $\div(e_\lambda)$ have $\ell = 0$, and these constitute all but finitely many colours.

Next, it will be the case that certain colours $D_x$ are exactly equal to $\pi^*(x)$, while at certain points $x \in C$ we will see that $\pi^*(x)$ is non-reduced and contains the colour $D_x$ with multiplicity greater than one. The former type of colour is called \emph{regular} and has $h$-co-ordinate in hyperspace equal to 1, while the latter colours are called \emph{subregular} and have $h$-co-ordinate greater than 1. Again, all but finitely many colours are regular, and hence have co-ordinates $(x, h,\ell) = (x, 1, 0)$ in hyperspace. We denote these points by $\epsilon_x$.

Finally, colours (or $G$-divisors) which do not appear in any $\pi^*(x)$ for $x \in C$ have $h$-co-ordinate 0 and are called \emph{central}. There are only finitely many central divisors on any $G$-model.

The existence and properties of subregular and central divisors will be of central importance to later results in this paper.

\subsection{Coloured Hyperfans}

We now describe how the Luna-Vust theory translates into a combinatorial description of varieties once the coloured data is inserted into the hyperspace. The key is that we can interpret semi-invariant functions in $K$ as `linear functionals' on $\cal{H}$. Indeed, given $f \in K^{(B)}_\lambda$ and a valuation $\nu$ corresponding to a point $(x, \ell, h) \in \cal{H}$, we can define $f(x,\ell,h) = \nu(f) = h\nu_x(f) + \ell(e_\lambda)$. It turns out that, once a sensible definition of a linear functional on $\cal{H}$ is given (see \cite[\S 2.1]{tim1}), all functionals on $\cal{H}$ are determined by semi-invariant functions.

Since the coloured data of a $G$-model $X$ consists of subsets of $\cal{V}$ and $\cal{D}^B$ corresponding to the $G$- and $B$-divisors containing its $G$-germs, and we have mapped this data into $\cal{H}$, we can use the functionals in $K^{(B)}$ to generate cones from this coloured data.

\subsubsection{$B$-charts}

Let $X_0$ be a $B$-chart with coloured data $\cal{W} \sub \cal{V}$, $\cal{R} \sub \cal{D}^B$. Let $\cal{C} = \cal{C}(\cal{W},\cal{R}) \sub \cal{H}$ be the set consisting of points $q \in \cal{H}$ such that any functional $\phi$ on $\cal{H}$ non-negative on $\cal{W}$ and $\cal{R}$ is non-negative on $q$. This is in some sense the `double dual' of the set $\cal{W} \union \cal{R}$ in $\cal{H}$. It is not a cone in the conventional sense because $\cal{H}$ is not a vector space or half-space. However, $\cal{C}_x = \cal{C} \inter \cal{H}_x$ is a cone for each $x \in \proj{1}$ and so is $\cal{K} = \cal{C} \inter \cal{Z}$. We call these objects $\cal{C}$ \emph{hypercones}, and their properties are determined by conditions on coloured data contained in the Luna-Vust theory.

First, we note that $B$-charts split into two types. Those we call \emph{type I} have nonconstant $B$-invariants in their algebra of functions. In this case, there is some $x \in \proj{1}$ such that $\cal{C}_x \sub \cal{K}$, i.e. there are no noncentral $G$- or $B$-divisors corresponding to some $x \in \proj{1}$. $B$-charts of \emph{type II} have no nonconstant $B$-invariant functions, and in this case every $x \in \proj{1}$ corresponds to at least one noncentral divisor.

The double dual set $\cal{C} \sub \cal{H}$ defined by the $B$-chart $X_0$ above determines an object called a \emph{hypercone}. Recall that $\epsilon_x$ denotes the point $(x, 1, 0) \in \cal{H}$.

\begin{definition}
	A \emph{hypercone} in $\cal{H}$ is a union $\cal{C} = \bigcup_{x \in \proj{1}}{\cal{C}_x}$ of finitely generated convex cones $\cal{C}_x = \cal{C} \inter \cal{H}_x$ such that: \begin{enumerate} \item $\cal{C}_x = \cal{K} + \bb{Q}_{\geq 0}\epsilon_x$ for all but finitely many $x$, where $\cal{K} = \cal{C} \inter \cal{Z}$; \item Either: \begin{enumerate}[label=(\Roman*)] \item there exists $x \in \proj{1}$ with $\cal{C}_x = \cal{K}$, or; \item the polytope $\cal{P} = \sum_{x \in \proj{1}}{\cal{P}_x}$ is non-empty, where the $\cal{P}_x$ are defined by $\epsilon_x + \cal{P}_x = \cal{C}_x \inter (\epsilon_x + \cal{Z})$.\end{enumerate}\end{enumerate} The hypercone $\cal{C}$ is called \emph{strictly convex} if every $\cal{C}_x$ is strictly convex and $0 \notin \cal{P}$.
\end{definition}

\begin{definition}
	A \emph{coloured hypercone} in $\cal{H}$ is a pair $(\cal{C},\cal{R})$ such that $\cal{R} \sub \cal{D}^B$, $0 \notin \cal{R}$, and $\cal{C}$ is a strictly convex hypercone in $\cal{H}$ generated by $\kappa(\cal{R})$, a finite subset $\cal{W} \sub \cal{V}$, and (if $\cal{C}$ is of type II) the polytope $\cal{P}$.
\end{definition}

\begin{theorem}\thlabel{hypercone}\cite[Thm 3.1]{tim1}
	$B$-charts of $K$ correspond bijectively to coloured hypercones in $\cal{H}$ of the corresponding type.
\end{theorem}

\subsubsection{$G$-germs}

Now let $Y$ be a $G$-germ of $K$. We say that $Y$ is of \emph{type I} if it admits a $B$-chart of type I (i.e. there exists a $B$-chart of type I intersecting $Y$) and $Y$ is of \emph{type II} if not. A $G$-germ is of type I if and only if $\cal{V}_Y \union \cal{D}^B_Y$ is finite and of type II if and only if it admits a minimal $B$-chart.

\begin{definition}
	The \emph{relative interior} of a (coloured) hypercone $\cal{C}$ of type II is the set $\bigcup_{x \in \proj{1}}{\relint{\cal{C}_x}} \union \relint\cal{K}$. We call $\cal{C}$ \emph{supported} if $\relint\cal{C} \inter \cal{V}$ is nonempty.
\end{definition}

\begin{definition}
	Let $\cal{C}$ be a hypercone in $\cal{H}$. A \emph{face} of $\cal{C}$ is a face of some $\cal{C}_p$ not intersecting $\cal{P}$. A \emph{hyperface} of $\cal{C}$ is a hypercone $\cal{C}^\prime = \cal{C} \inter \ker{\phi}$ for some functional $\phi$ on $\cal{H}$ nonnegative on $\cal{C}$. We call $\phi$ a \emph{supporting functional} for the face $\cal{C}^\prime$. 
	
	A \emph{(hyper)face} of a coloured hypercone $(\cal{C},\cal{R})$ is a coloured (hyper)cone $(\cal{C}^\prime,\cal{R}^\prime)$ where $\cal{C}^\prime$ is a (hyper)face of $\cal{C}$ and $\cal{R}^\prime = \cal{R}\inter \kappa^{-1}(\cal{C}^\prime)$.
\end{definition}

\begin{theorem}\thlabel{hypergerms}\cite[Thm 3.2]{tim1}
	$G$-germs of type I are in bijection with supported coloured cones in $\cal{H}$, and $G$-germs of type II are in bijection with supported coloured hypercones of type II in $\cal{H}$. Inclusion of $G$-germs in each other corresponds to opposite inclusions of the respective (hyper)cones as (hyper)faces of each other.
\end{theorem}

\subsubsection{$G$-models}

We know that $G$-models are determined by their $G$-germs, which lie in a finite collection of $B$-charts. The supports of the $G$-germs must be disjoint, and inclusions of $G$-germs must be kept track of. In this spirit we have the following definition:

\begin{definition}
	A \emph{coloured hyperfan} in $\cal{H}$ is a collection of supported coloured cones and supported coloured hypercones of type II in $\cal{H}$, obtained as the set of all supported (hyper)faces of a finite collection of coloured hypercones, subject to the condition that the relative interiors of these (hyper)faces are disjoint inside $\cal{V}$.
\end{definition}

The coloured hyperfan $\cal{F}_X$ of a $G$-model $X$ is the collection $\{\cal{C}_Y \mid Y \sub X\}$ where $Y$ runs over all $G$-germs of $X$. Then by \thref{hypercone} and \thref{hypergerms}, we have \cite[Thm 3.3]{tim1}:

\begin{theorem}
	$G$-models of $K$ are in bijection up to isomorphism with coloured hyperfans in $\cal{H}$.
\end{theorem}

\begin{corollary}
	A $G$-model is complete if and only if its coloured hyperfan covers $\cal{V}$.
\end{corollary}

\begin{proposition}\thlabel{typeIbirational}
	We say that a $G$-model $X$ is of \emph{type I} if all of its $G$-germs are of type I, and of \emph{type II} if it contains any $G$-germ of type II. For any $G$-model $X$, there exists a $G$-model $\check{X}$ of type I and a proper birational morphism $\phi \colon \check{X} \to X$.
\end{proposition}

\begin{proof}
	See \cite[\S 16.6]{tim}
\end{proof}

Section 7 of this paper contains many example calculations with figures of the coloured hyperfans of  complexity one varieties, which should illuminate this complicated theory.

\subsection{Divisors on Complexity One $G$-Varieties}\label{div}

We now begin to study the properties of divisors on complexity-one $G$-varieties, following \cite{tim2}. Throughout this section $X$ is a normal but possibly singular variety unless otherwise specified. Helpfully, we can reduce everything to $B$-stable divisors:

\begin{proposition}\thlabel{Bequiv}
	Let a connected solvable algebraic group $B$ act on a normal variety $X$. Then any Weil divisor on $X$ is linearly equivalent to a $B$-stable one.
\end{proposition}

\begin{proof}
	See \cite[Prop 17.1]{tim}.
\end{proof}

\subsubsection{Cartier Divisors}

Next we want to investigate conditions which guarantee that a divisor is Cartier. We will assume that the associated line bundle to any Cartier divisor is $G$-linearised (this is fine by \cite[Prop 2.4]{kklv} since $G$ is factorial).

\begin{lemma}\thlabel{globgencartier}
	Any prime divisor $D \sub X$ which does not contain a $G$-orbit is Cartier and generated by global sections.
\end{lemma}

\begin{proof}
	See \cite[Lemma 17.3]{tim}.
\end{proof}

\begin{theorem}
	Let $\delta$ be a divisor on $X$ and assume by \thref{Bequiv} that $\delta$ is $B$-stable. Then $\delta$ is Cartier if and only if for any $G$-germ $Y$ of $X$, there exists $f_Y \in K^{(B)}$ such that each prime divisor $D$ containing $Y$ occurs in $\delta$ with multiplicity $\nu_D(f_Y)$. 
\end{theorem}

\begin{proof}
	See \cite[Thm 17.4]{tim}
\end{proof}

\begin{corollary}
	A Cartier divisor $\delta$ on a $G$-model $X$ is determined by the following data: \begin{enumerate}
		\item a collection $\{f_Y\}$ of $B$-eigenfunctions for each $G$-germ $Y \sub X$ such that $\nu(f_{Y_1}) = \nu(f_{Y_2})$ and $\nu_D(f_{Y_1}) = \nu_D(f_{Y_2})$ for all $\nu \in \cal{V}_{Y_1} \inter \cal{V}_{Y_2}$ and all $D \in \cal{D}^B_{Y_1} \inter \cal{D}^B_{Y_2}$;
		\item a collection of integers $m_D$ for each $D \in \cal{D}^B\setminus{\bigcup_{Y \sub X}{\cal{D}^B_Y}}$ ($m_D$ being the multiplicity of $D$ in the divisor), only finitely many of which are nonzero.
	\end{enumerate}
\end{corollary}

If $X$ is quasihomogeneous of complexity one, each $f_Y$ determines up to scalar multiples (and is up to powers determined by) a linear functional $\phi_Y$ on the coloured cone or hypercone $\cal{C}_Y$ such that $\phi_{Y_1}\vert_{\cal{C}_{Y_2}} = \phi_{Y_2}$ whenever $\cal{C}_{Y_2}$ is a face of $\cal{C}_{Y_1}$, that is, whenever $Y_2$ contains $Y_1$. Then the functionals $\phi_Y$ paste together to a piecewise linear function on $\bigcup_{Y \sub X}{\cal{C}_Y \inter \cal{V}}$, which we call a \emph{piecewise linear function} on the coloured hyperfan $\cal{F}_X = \{\cal{C}_Y \mid Y \sub X\}$ of $X$. Then Cartier divisors on $X$ correspond to these piecewise linear functions.

\subsubsection{Globally Generated and Ample Divisors}

\begin{proposition}
	Let $\delta$ be a Cartier divisor on $X$ given by $\{f_Y\}$, $\{m_D\}$ as above. Then: \begin{enumerate}
		\item $\delta$ is globally generated if and only if $f_Y$ can be chosen such that for any $G$-germ $Y \sub X$, we have: \begin{enumerate}
			\item for any other $G$-germ $Y^\prime \sub X$ and every $B$-stable divisor $D$ containing $Y^\prime$, $\nu_D(f_Y) \leq \nu_D(f_{Y^\prime})$;
			\item for any $D \in \cal{D}^B\setminus{\bigcup_{Y^\prime \sub X}{\cal{D}^B_Y}}$, $\nu_D(f_Y) \leq m_D$.
		\end{enumerate}
		\item $\delta$ is ample if and only if, after replacing $\delta$ with some positive multiple, $f_Y$ can be chosen such that for any $G$-germ $Y \sub X$, there exists a $B$-chart $X_0 \sub X$ intersecting $Y$ such that (a) and (b) hold, and the inequalities therein are strict if and only if $D \inter X_0 = \emptyset$.
	\end{enumerate}
\end{proposition}

\begin{proof}
	See \cite[Thm 17.18]{tim}
\end{proof}

\subsubsection{Global Sections}\label{globsec}

Let $\cal{B}(X)$ be the set of all $B$-stable prime divisors on $X$, including the $G$-stable ones. Let $\delta = \sum_{D\in \cal{B}(X)}{m_D D}$ be a $B$-stable Cartier divisor, and let $\eta_\delta \in H^0(X,\cal{O}(\delta))^{(B)}$ be the respective rational $B$-eigensection (i.e. $\Div{\eta_\delta} = \delta$). We have \[H^0(X,\cal{O}(\delta))^{(B)} = \{f\eta_\delta \mid f \in K^{(B)}, \Div{f} + \delta \geq 0\}.\] The $B$-weight of an arbitrary $B$-eigensection $\sigma = f\eta_\delta$ is $\lambda + \lambda_\delta$, where $\lambda$ is the weight of $f$ and $\lambda_\delta$ is the weight of $\eta_\delta$. The latter is determined up to a character of $G$ and can be calculated as follows: let $Y$ be a $G$-orbit intersecting $\delta$ and pull $Y \inter \delta$ back to $G$ under the orbit map, giving a divisor $\tilde{\delta}$ on $G$. Since we assume $G$ to be factorial, $\tilde{\delta}$ is principal, defined by a rational function $F \in \bbk(G)^{(B)}$. Then $\lambda_\delta$ is the $B$-weight of $F$. 

It follows that \[H^0(X,\cal{O}(\delta))^{(B)}_{\lambda + \lambda_\delta} \cong \{f \in K^{(B)}_\lambda \mid \Div{f} + \delta \geq 0\} \cong \{f \in K^B \mid \Div{f} + \Div{e_\lambda} + \delta \geq 0\}.\]

We want to calculate the calculate the dimension of the space $H^0(X,\cal{O}(\delta))$ of global sections of $\delta$. Note that, setting $m_\lambda(\delta) =\dim{H^0(X,\cal{O}(\delta))^{(B)}_{\lambda+\lambda_\delta}}$ for brevity, using the Weyl dimension formula \cite[\S 24.3]{hum1} for modules and counting multiplicities of simple submodules, we have \[\dim{H^0(X,\cal{O}(\delta))} = \sum_{\lambda \in \Lambda}{m_\lambda(\delta)\prod_{\alpha^\vee \in \Delta^\vee_+}{\left(1 + \frac{(\lambda,\alpha^\vee)}{(\rho,\alpha^\vee)}\right)}}.\]

We can calculate $m_\lambda(\delta)$ using the notion of a pseudodivisor:

\begin{definition}
	Let $C$ be a smooth projective curve. A \emph{pseudodivisor} $\mu$ on $C$ is a formal linear combination $\mu = \sum_{p \in C}{m_p\cdot p}$ where $m_p \in \bb{R} \union \{\pm \infty\}$ and all but finitely many $m_p$ are 0. Let $H^0(C,\mu) = \{f \in \bbk(C) \mid \Div{f} + \mu \geq 0\}$ where for all $x \in \bb{R}$, we set $x + (\pm \infty) = \pm\infty$.
	
	If there is $p \in C$ with $m_p = -\infty$, then $H^0(C,\mu) = 0$. Otherwise, $H^0(C,\mu)$ is the space of global sections of the divisor $\lfloor\mu\rfloor = \sum_{p}{\lfloor m_p\rfloor\cdot p}$ on $C\setminus\{p \in C\mid m_p = +\infty\}$, where $\lfloor m_p\rfloor$ represents the floor of $m_p$.
\end{definition}

Now let $\delta $ be as above. Note that $H^0(X,\cal{O}(\delta))^{(B)}$ is isomorphic to \[\left\{f_0e_\lambda \mid f_0 \in K^B, \lambda \in \Lambda, \sum_{D \in \cal{B}(X)}{[h_D\nu_{p_D}(f_0)+\langle \lambda, \ell_D\rangle + m_D]D} \geq 0\right\}.\] Hence fix $\lambda \in \Lambda$ and consider the pseudodivisor \[H_\lambda = H_\lambda(\delta) = \sum_{p \in \proj{1}}{\left(\min_{p_D = p}{\frac{\langle \lambda,\ell_D\rangle + m_D}{h_D}}\right)p},\] where we assume $\frac{x}{0} = +\infty$ for $x \geq 0$ and $\frac{x}{0} = -\infty$ for $x < 0$. It is clear from the above description of $H^0(X,\cal{O}(\delta))^{(B)}$ that $m_\lambda(\delta) = \dim{H^0(\proj{1},H_\lambda(\delta))} := h^0(\delta,\lambda)$.

We know that $h^0(\proj{1},H_\lambda) = 0$ if any of its coefficients are $-\infty$. This is the case exactly when there is $p \in \proj{1}$ and $D \in \cal{B}(X)$ with $p_D = p$ satisfying $h_D = 0$ and $\langle \lambda, \ell_D\rangle < -m_D$. Hence we define a polyhedral domain \[\cal{P}(\delta) = \{\lambda \in \Lambda \otimes \bb{R} \mid \langle \lambda,\ell_D\rangle \geq -m_D \ \text{for all}\ D \ \text{with} \ h_D = 0\}.\] Then $H^0(\proj{1},H_\lambda) = 0$ for all $\lambda \notin \cal{P}(\delta)$. Conversely, a coefficient of $H_\lambda$ is $+\infty$ if and only if there is $p \in \proj{1}$ such that no divisor $D \in \cal{B}(X)$ with $p_D = p$ satisfies $h_D > 0$. This is the case e.g. if $X$ is a $B$-chart of type I. Then $H^0(\proj{1},H_\lambda)$ is the space of global sections of $\lfloor H_\lambda \rfloor$ on the affine curve $\proj{1}\setminus{\{p \mid m_p = +\infty\}}$ and hence $h^0(\delta,\lambda) = \infty$ for all $\lambda \in \cal{P}(\delta)$.

Otherwise, $H_\lambda$ is a `standard' Weil divisor on $\proj{1}$, so by Riemann-Roch we have $h^0(\delta,\lambda) = \deg{\lfloor H_\lambda \rfloor} + 1 + h^1(\delta,\lambda)$, where $h^1(\delta,\lambda) := \dim{H^1(\proj{1},\lfloor H_\lambda\rfloor)}$. If we define \[A(\delta,\lambda) = \sum_{p \in \proj{1}}{\min_{p_D=p}{\frac{\langle \lambda,\ell_D\rangle + m_D}{h_D}}},\] i.e. $A(\delta,\lambda) = \deg{H_\lambda}$, then $\deg{\lfloor H_\lambda(\delta) \rfloor}$ differs from $A(\delta,\lambda)$ by some bounded non-negative function $\sigma(\delta,\lambda)$ for all $\delta,\lambda$. We then have $h^0(\delta,\lambda) = A(\delta,\lambda) - \sigma(\delta,\lambda) + h^1(\delta,\lambda) + 1$. 

\begin{proposition}\thlabel{largen}
	If $A(\delta,\lambda) < 0$, then $h^0(\delta,\lambda) = 0$. Otherwise, for large $n$, $h^0(n\delta,n\lambda) \sim nA(\delta,\lambda)$.
\end{proposition}

\begin{proof}
	See \cite[\S 17.4]{tim}
\end{proof}

Inspired by the first part of the above Proposition, define the polyhedral domain \[\cal{P}_+(\delta) = \{\lambda \in \cal{P}(\delta) \mid A(\delta,\lambda) \geq 0\}.\] Then by the above and the definition of $\cal{P}(\delta)$, we have $h^0(\delta,\lambda) = 0$ for all $\lambda \notin \cal{P}_+(\delta)$.

\subsubsection{Volume of Divisors}

Now assume that $X$ is a smooth projective $G$-model. Let $\Delta$ be the root system of $G$, $\Delta_+$ the positive roots determined by $B$, $\Delta_+^\vee$ the corresponding set of positive coroots and $\rho = \frac{1}{2}\sum_{\alpha \in \Delta_+}{\alpha}$. The following formula of Timashev \cite{tim2} allows us to compute the volume of a Cartier divisor on a complexity one variety.

\begin{theorem}
	Let $\delta$ be a $B$-stable Cartier divisor of weight $\lambda_\delta$ on a normal projective quasihomogeneous $G$-variety $X$ of dimension $d$, complexity $c = 1$ and rank $r$. Then, in the notation of the previous subsection: \[d = c + r + |\Delta^\vee_+\setminus{(\Lambda + \bb{Z}\lambda_\delta)^\perp}|,\] and \[\vol{\delta} = d!\int\limits_{\lambda_\delta + \cal{P}_+(\delta)}{A(\delta,\lambda-\lambda_\delta)\prod_{\alpha^\vee \in \Delta^\vee_+\setminus{(\Lambda + \bb{Z}\lambda_\delta)^\perp}}{\frac{\langle \lambda,\alpha^\vee\rangle}{\langle\rho,\alpha^\vee\rangle}}\, \mathrm{d}\lambda}\] where the Lebesgue measure on $\Lambda \otimes \bb{R}$ is normalised such that a fundamental parallelepiped of $\Lambda$ has volume 1.
\end{theorem}

\begin{proof}
	See \cite[Thm 18.8]{tim}
\end{proof}

\begin{example}
	Let $G = \sl{2}$ so that $\Delta = \{\pm\alpha\}$, $\Delta_+ = \{\alpha\}$, $\Delta^\vee_+ = \left\{\alpha^\vee = \frac{2\alpha}{(\alpha,\alpha)}\right\}$ and $\rho = \frac{1}{2}\alpha$. Any quasihomogeneous complexity one $\sl{2}$-threefold $X$ has $d = 3$, $c = 1$, $r = 1$, so $|\Delta^\vee_+\setminus{(\Lambda + \bb{Z}\lambda_\delta)^\perp}| = 1$, i.e. this set is just $\Delta^\vee_+ = \{\alpha^\vee\}$. If we identify $\Lambda = \bb{Z}\alpha$ with $\bb{Z}$, and hence $\alpha$ with $1$, we have \[\prod_{\alpha^\vee \in \Delta^\vee_+\setminus{(\Lambda + \bb{Z}\lambda_\delta)^\perp}}{\frac{\langle \lambda,\alpha^\vee\rangle}{\langle\rho,\alpha^\vee\rangle}} = \frac{\langle \lambda,\alpha^\vee\rangle}{\langle \frac{\alpha}{2},\alpha^\vee\rangle} = 2\lambda,\] and the volume of a Cartier divisor $\delta$ on $X$ is given by \[\vol{\delta} = 6\int\limits_{\lambda_\delta + \cal{P}_+(\delta)}{2\lambda A(\delta,\lambda-\lambda_\delta)\, \mathrm{d}\lambda}.\]
\end{example}

We will put this formula to great use later on to calculate $\beta$-invariants of prime divisors over $\sl{2}$-threefolds. 

\section{Main Results}

The main results of this paper are the following:

\begin{theorem}\thlabel{Kstablecentral}
	Let $X$ be a smooth Fano $\sl{2}$-threefold. If any of the three conditions below holds, then $X$ is $K$-polystable if $\beta_F(X) > 0$ for all \emph{central} $G$-stable prime divisors over $X$: \begin{enumerate}[label=(\roman*)] \item A finite subgroup $A \sub \aut{X}$ acts on $\proj{1}$ with no fixed points, such that the rational $B$-quotient $X \dashrightarrow \proj{1}$ is $A$-equivariant \item A finite subgroup $A \sub \aut{X}$ acts on $\proj{1}$, interchanging two points in $\proj{1}$ corresponding to subregular colours of $X$, and the rational $B$-quotient $X \dashrightarrow \proj{1}$ is $A$-equivariant \item $X$ has subregular colours lying over three or more distinct points of $\proj{1}$.\end{enumerate}
\end{theorem}

\begin{remark}
	We expect but have not proved that \thref{Kstablecentral} applies with only minor alterations to smooth Fano $G$-varieties of complexity one in general, rather than just to $\sl{2}$-threefolds.
\end{remark}

\begin{remark}
	We note the similarity of this result to \cite[Thm 1.1]{suess1}
\end{remark}

This result is essential since there are in general infinitely many prime divisors over $X$, even $G$-invariant ones, but there are only finitely many central ones.

\begin{theorem}\thlabel{threefoldsKstable}
	The smooth Fano threefolds, (1.16), (1.17), (2.27), (2.32), (3.17), (3.25) and (4.6) in the Mori-Mukai classification are $K$-polystable. The families (2.21) and (3.13) each contain a $K$-polystable variety.
\end{theorem}

\begin{remark}
	The same result were recently obtained independently by other authors using different methods, see \cite{sc,cal}. The $K$-polystability of the Mukai-Umemura threefold in the family (1.10) was already known by Donaldson \cite{don1}, and the $K$-polystability of $V_5$ (1.15) was known by Cheltsov-Shramov \cite{cs}. 
\end{remark}

We will prove \thref{Kstablecentral} in the next section for each of the three cases. In the following two sections, we present the coloured data for each complexity one homogeneous space of $\sl{2}$ and calculate the coloured hyperfans of the varieties listed in \thref{threefoldsKstable}. We then use this data to demonstrate that for each of these, one of the three conditions of \thref{Kstablecentral} holds. In the final section, we show that there can only be one central prime divisor over any of these varieties. We then calculate the $\beta$-invariant of this divisor explicitly in each case, completing the proof of \thref{threefoldsKstable}.

\section{Proof of \thref{Kstablecentral}}

\subsection{Action Without Fixed Points}

\begin{proof}[Proof of \thref{Kstablecentral}(i)]
	If $A$ acts on $\proj{1}$ and the $B$-quotient $X \dashrightarrow \proj{1}$ is $A$-equivariant, we have an overall action on $X$ of an extension $G^\prime$ of $G$ by $A$ which descends to the quotient. Suppose $F$ is a non-central $G$-invariant prime divisor over $X$. Since $F$ is non-central, it must lie over some point $P_F \in \proj{1}$. If $A$ acts with no fixed points on $\proj{1}$, then in particular $P_F$ is not fixed, so $F$ cannot be $G^\prime$-invariant. It follows that only central divisors can be $G^\prime$-invariant. Since $A$ is finite, $G^\prime$ is reductive, so by the theorems of Fujita-Li and Datar-Sz\'{e}kelyhidi, $X$ is $K$-polystable if $\beta(F) \geq 0$ for all $G^\prime$-invariant prime divisors over $X$, i.e. for all central $G$-divisors over $X$, and is only 0 for divisors corresponding to product configurations.
\end{proof}

\subsection{Non-Normality}

The proof of \thref{Kstablecentral} parts (ii) and (iii) will be by showing that, under these conditions, the test configurations corresponding to non-central prime divisors over $X$ have non-normal central fibre, and hence are not \emph{special} test configurations, and therefore we do not need to calculate their Donaldson-Futaki invariant (or, equivalently, the $\beta$ invariant of $F$). We will use the correspondence described in \Cref{div} between $B$-semi-invariant sections of prime divisors on over $X$ and sections of divisors on the $B$-quotient. We show that the filtrations defined by the resulting divisors on $\proj{1}$ give non-integrally closed rings which correspond to the central fibres of the given test configuration.

\subsubsection{Divisors on $\proj{1}$}

\begin{theorem}\thlabel{nonnormal}
	Let $H = \sum_{i=1}^{m}{a_iQ_i} = \sum_{i=1}^{m}{\frac{b_i}{c_i}Q_i}$ be a $\bb{Q}$-divisor of positive degree on $\proj{1}$, and let $P \in \proj{1}$. Let \[\cal{A} = \bigoplus_{k \in \bb{Z}}{H^0(kH)}\] be the section ring of $H$. Fix $q \in \bb{Z}$ and consider the filtration on $\cal{A}$ over $r \in \bb{Z}$ given by \[\cal{F}^q_r = \bigoplus_{k \in \bb{Z}}{\left\{f \in H^0(kH) \mid \ord_P(f) \geq \frac{r}{q}\right\}}.\] Then take the associated graded ring \[\cal{B}^q = \bigoplus_{r \in \bb{Z}}{\cal{F}^q_r/\cal{F}^q_{r+1}}.\]
	
	If at least two $Q_i$, both distinct from $P$, have non-integral coefficients $a_i \notin \bb{Z}$ in $H$, then for each $q \in \bb{Z}_{\geq 1}$, the ring $\cal{B}^q$ is not integrally closed.
\end{theorem}

We will prove this theorem in a number of steps, beginning with:

\begin{proposition}\thlabel{nkr}
	With all notation as in \thref{nonnormal}, for any $q \geq 1$ there exist integers $k, r$ and $n$, with $k$ and $n$ positive, such that $\cal{B}^q_{(k,r)} = 0$ and $\cal{B}^q_{(nk,nr)} \neq 0$.	
\end{proposition}

\begin{lemma}
	In the above proposition and theorem, we can assume without loss of generality that $q = 1$.
\end{lemma}

\begin{proof}
	Let $q \geq 1$. We have $\cal{B}^1_{(k,r)} = \cal{B}^q_{(k,qr)}$, so if we find $k, r$ and $n$ with $\cal{B}^1_{(k,r)} = 0$ and $\cal{B}^1_{(nk,nr)} \neq 0$, then $k, qr$ and $n$ give the required result for $q > 1$. Hence we set $q = 1$ going forward, and we drop the corresponding superscript.
\end{proof}

\begin{proof}[Proof of \thref{nkr}]	
	By the lemma, we assume that $q = 1$, and we want to find $n, k$ and $r$ with $\cal{B}_{(k,r)} = 0$ and $\cal{B}_{(nk,nr)} \neq 0$. We will rewrite $H = \sum_{i=1}^{m}{a_iQ_i} + a_PP$, with $a_P = \frac{b_P}{c_P}$, and assuming all $Q_i \neq P$. Sums and products indexed over $i$ should be understood as running from $i=1$ to $i=m$ and \emph{excluding} $P, a_P$ etc. unless otherwise specified.\\
	
	Before we choose particular values of $k$, $r$ and $n$, we will demonstrate an alternative description of $\cal{B}_{(k,r)}$. Denote by $\cal{F}_{(k,r)}$ the degree-$k$ part of $\cal{F}_r$, i.e. $\cal{F}_{(k,r)} = \{f \in H^0(kH) \mid \ord_P(f) \geq r\}$. Then $\cal{B}_{(k,r)} = \cal{F}_{(k,r)}/\cal{F}_{(k,r+1)}$. Define \[H_{(k,r)} = \begin{cases} kH & -r \geq \lfloor ka_P \rfloor \\ kH-ka_PP-rP & -r \leq \lfloor ka_P\rfloor.\end{cases}\] We will show that $\cal{F}_{(k,r)} = H^0(H_{(k,r)})$, so $\cal{B}_{(k,r)} = H^0(H_{(k,r)})/H^0(H_{(k,r+1)})$. 
	
	Indeed, when $-r \geq \lfloor ka_P\rfloor$, it suffices to show that any $f \in H^0(kH)$ automatically has order at least $r$ at $P$. Any such $f$ must satisfy $\ord_P(f) + \lfloor ka_P\rfloor \geq 0$, and if $0 \geq r + \lfloor ka_P\rfloor$ then the result follows.
	
	On the other hand, it is clear that $\cal{F}_{(k,r)} \sub H^0(kH-ka_PP-rP)$. Now suppose $f \in H^0(kH-ka_PP-rP)$ and $-r < \lfloor ka_P\rfloor$. We have $\ord_P(f) \geq r$ since the coefficient at $P$ of $(f)+\lfloor kH-ka_PP-rP\rfloor$ is $\ord_P(f) - r$, so it remains to show that $f \in H^0(kH)$. Since $kH$ only differs from $kH-ka_PP-rP$ at $P$, it is sufficient to note that $\ord_P(f)+\lfloor ka_P\rfloor \geq r + \lfloor ka_P\rfloor > 0$.
	
	Hence to show $\cal{B}_{(k,r)} = 0$ it is sufficient either that $\deg{\lfloor H_{(k,r)}\rfloor} < 0$ or $H^0(H_{(k,r)}) = H^0(H_{(k,r+1)})$. 
	
	Likewise, for $\cal{B}_{(k,r)} \neq 0$ we must show that $\deg{\lfloor H_{(k,r)}\rfloor} \geq 0$ and $H^0(H_{(k,r)}) \neq H^0(H_{(k,r+1)})$. Note that when the first of these conditions holds and $-r \leq \lfloor ka_P\rfloor$, the second one also holds by definition of $H_{(k,r)}$. If $-r \geq \lfloor ka_P\rfloor$ and $\deg{\lfloor kH\rfloor} \geq 0$ we have \[0 \leq \deg{\lfloor kH\rfloor} = \sum_{i}{\lfloor ka_i\rfloor} + \lfloor ka_i\rfloor \leq r + \lfloor ka_i\rfloor \leq 0,\] so $-(r+1) < \lfloor ka_i\rfloor$ and $H^0(H_{(k,r)}) \neq H^0(H_{(k,r+1)})$ as well. Thus $\cal{B}_{(k,r)} \neq 0$ if and only if $\deg{\lfloor H_{(k,r)}\rfloor} \geq 0$.\\
	
	\noindent\underline{Choice of $k$}:\\
	
	Our choice of $k$ is motivated by two requirements, the reasons for which will be seen later, these being: \[(1) \ \sum_{i}{\{ka_i\}} \geq 1\] where $\{x\} = x - \lfloor x\rfloor$ is the fractional part of a real number $x$, and \[(2) \ \deg{\lfloor kH\rfloor} \geq 0.\] With that in mind, consider
	\[k = \begin{cases} \prod_{i}{c_i}+1 & \sum_{i}{\{a_i\}} \geq 1 \\ \prod_{i}{c_i}-1 & \sum_{i}{\{a_i\}} < 1.\end{cases}\] This choice satisfies requirement (1): in the first case we have $\{ka_i\} = \{a_i\}$ for each $i$ (by the fact that $\{x + n\} = \{x\}$ for all integers $n$ and real $x$), so $\sum_{i}{\{ka_i\}} = \sum_{i}{\{a_i\}}\geq 1$ by assumption. If $\sum_{i}{\{a_i\}} < 1$ we have \[\begin{split} \sum_{i}{\{ka_i\}} & = \sum_{i}{\{-a_i\}} \\ & = |\{i \mid a_i \notin \bb{Z}\}| - \sum_{i}{\{a_i\}} \\ & \geq 2 - \sum_{i}{\{a_i\}} > 1\end{split}\] since $\{-x\}$ is $1-\{x\}$ whenever $x \notin \bb{Z}$ and 0 when $x \in \bb{Z}$. Note that this is where we use our assumption that at least two $a_i$ are non-integral, and it is essential.
	
	However, this choice of $k$ may not satisfy requirement (2). We have \[\deg{\lfloor kH\rfloor} = \deg{(kH)} - \sum_{i}{\{ka_i\}} - \{ka_P\} > \deg{(kH)} - (m+1),\] since $0 \leq \{x\} < 1$ for all $x$. Since $\deg{H} > 0$, for $k \gg 0$ we will have $\deg{(kH)} \geq m+1$ and requirement (2) will be satisfied. Hence replace our initial choice of $k$ with $k + \ell \prod_{i}{c_i}$ for $\ell$ large enough to give $\deg(kH) \geq m+1$ - this choice will still satisfy requirement (1) since $\left\{\left(k +\ell \prod_{i}{c_i}\right)a_i\right\} = \{ka_i\}$ for each $i$ in either case.\\
	
	\noindent\underline{Choice of $r$, $\cal{B}_{(k,r)} = 0$}:\\
	
	Let $r = \lfloor \deg{(kH-ka_PP)}\rfloor = \left\lfloor \sum_{i}{ka_i}\right\rfloor$. Then \begin{align*} r + \lfloor ka_P\rfloor & = \left\lfloor \sum_{i}{ka_i}\right\rfloor + \lfloor ka_P\rfloor \\ & \geq \sum_{i}{\lfloor ka_i\rfloor} + \lfloor ka_P\rfloor \\ & = \deg{\lfloor kH\rfloor} \geq 0\end{align*} by requirement (2) of our choice of $k$, so $-r \leq \lfloor ka_P\rfloor$ and $\cal{F}_{(k,r)} = H^0(H_{(k,r)}) = H^0(kH-ka_PP-rP)$. 
	
	We have \[\begin{split}\deg{\lfloor kH-ka_PP-rP\rfloor} & = \sum_{i}{\lfloor ka_i\rfloor} + \lfloor -r\rfloor \\ & = \sum_{i}{\lfloor ka_i\rfloor} - \left\lfloor\sum_{i}{ka_i}\right\rfloor \\ & = \left\{\sum_{i}{ka_i}\right\} - \sum_{i}{\{ka_i\}} \\ & < 1 - \sum_{i}{\{ka_i\}} \leq 0 \end{split}\] since $\{x\} < 1$ for all $x$ and $\sum_{i}{\{ka_i\}} \geq 1$ by requirement (1) of our choice of $k$. It follows that $\cal{B}_{(k,r)} = \cal{F}_{(k,r)}/\cal{F}_{(k,r+1)} = 0$.\\
	
	\noindent\underline{Choice of $n$, $\cal{B}_{(nk,nr)} \neq 0$}:\\
	
	Next, choose $n = c_P\prod_{i}{c_i}$, so that $nkH$ is an integral divisor. To show that $\cal{B}_{(nk,nr)} \neq 0$, recall that it suffices to show that $\deg{\lfloor H_{(nk,nr)}\rfloor} \geq 0$. We have \begin{align*}nr + \lfloor nka_P\rfloor & = n\left(\left\lfloor\sum_{i}{ka_i}\right\rfloor+ka_P\right) \\ & = n\left(\deg{(kH)}-\left\{\sum_{i}{ka_i}\right\}\right) >  0,\end{align*} since we chose $k$ satisfying $\deg{(kH)} \geq m+1 > 1$. Hence $-nr \leq nka_P$, meaning $H_{(nk,nr)} = nkH-nka_P-nrP$. This divisor has degree \begin{align*}\deg{H_{(nk,nr)}} = & \deg{(nkH)} - nka_P - nr \\ & = n\left(\deg{(kH)}-ka_P -\left\lfloor\sum_{i}{ka_i}\right\rfloor\right) \\ & = n\left\{\sum_{i}{ka_i}\right\}\geq 0\end{align*} since $\{x\} \geq 0$ for any $x$. Hence $\cal{B}_{(nk,nr)} \neq 0$, and we are done.
\end{proof}

Our goal is to show that for each $H$ as above, the ring $\cal{B}$ is not integrally closed. We now know that there exist $k, r$ and $n$ with $\cal{B}_{(k,r)} = 0$ and $\cal{B}_{(nk,nr)} \neq 0$. Let $K$ be the algebra consisting of fractions of homogeneous elements of $\cal{B}$. If there exists $f \in K_{(k,r)}$ with $f^n \in \cal{B}_{(nk,nr)}$ then the monic polynomial $x^n-f^n \in \cal{B}[x]$ has a root in $K$ which does not lie in $\cal{B}$ (since such a root would lie in $\cal{B}_{(k,r)}$), which would prove non-normality. We now show the existence of such an element.

\begin{proposition}
	Let $H, k, r$ and $n$ be as above. There exist integers $(k^\prime,r^\prime)$ such that $\cal{B}_{(k^\prime,r^\prime)}$ and $\cal{B}_{(k^\prime+k,r^\prime + r)}$ are both nonzero. Hence there exists a nonzero function in $K_{(k,r)}$.
\end{proposition}

\begin{proof}
	Recall from above that $\cal{B}_{(k,r)} \neq 0$ if and only if $\deg{\lfloor H_{(k,r)}\rfloor} \geq 0$.

	We have $\deg{\lfloor kH\rfloor} = \deg{(kH)} - \sum_{i}{\{ka_i\}} - \{ka_P\} > \deg{(kH)} - t$, where $t$ is the number of terms in $kH$ with non-integral coefficients (possibly including $ka_P$). Hence choose $k^\prime$ such that $\deg{(k^\prime H)} > t$ and choose $r^\prime = -\lfloor k^\prime a_P\rfloor$. Then $H_{(k^\prime,r^\prime)} = k^\prime H = k^\prime H - k^\prime a_PP -r^\prime P$ and $\deg{\lfloor k^\prime H\rfloor} > 0$, so $\cal{B}_{(k^\prime, r^\prime)} \neq 0$.
	
	Now let $k$ and $r$ be as in the proof of \thref{nkr} and recall that for these values we have $-r\leq \lfloor ka_P\rfloor$. Hence \[H_{(k+k^\prime,r+r^\prime)} = (k+k^\prime)H-(k+k^\prime)a_PP - (r+r^\prime)P = H_{(k,r)} + H_{(k^\prime,r^\prime)}.\] Then $\deg{\lfloor H_{(k+k^\prime,r+r^\prime)}\rfloor} \geq \deg{\lfloor H_{(k,r)}\rfloor} + \deg{\lfloor H_{(k^\prime,r^\prime)}\rfloor}$. Since $\deg{\lfloor H_{(k,r)}\rfloor}$ is fixed, we may increase $k^\prime$ (and thus increase $\deg{\lfloor H_{(k^\prime,r^\prime)}\rfloor}$) to ensure that $\deg{\lfloor H_{(k+k^\prime,r+r^\prime)}\rfloor} \geq 0$, if necessary. 
	
	It follows that $\cal{B}_{(k+k^\prime,r+r^\prime)} \neq 0$ as required, and since $\cal{B}_{(k^\prime,r^\prime)} \neq 0$ as well, taking the quotient of a nonzero element of the former by a nonzero element of the latter gives a nonzero element in $K_{(k,r)}$.
\end{proof}

We can now prove the theorem:

\begin{proof}[Proof of \thref{nonnormal}]
	As always we are free to assume that $q = 1$ and drop the superscript.
	
	By some previous remarks, it suffices to find a nonzero element $f \in K_{(k,r)}$ with $f^n \in \cal{B}_{(nk,nr)}$. The result above shows that $K_{(k,r)} \neq 0$, so we may choose $f \in K_{(k,r)}$ to be nonzero, so $f^n \in K_{(nk,nr)} \neq 0$. We will show that $K_{(nk,nr)}$ is a line, and since it contains $\cal{B}_{(nk,nr)} \neq 0$, it follows that the two are equal, giving the result.
	
	Let $k^\prime$, $r^\prime$ be arbitrary integers, and suppose $g, h \in K_{(k^\prime,r^\prime)}$ are nonzero. Their quotient must lie in $K_{(0,0)}$, which consists of fractions of homogeneous elements of $\cal{B}$ of equal degree. Since $\cal{B}_{(k^\prime,r^\prime)} = H^0(H_{(k,r)})/H^0(H_{(k^\prime,r^\prime+1)})$ by the previous proof, it is clear that $\dim{\cal{B}_{(k^\prime,r^\prime)}} \leq 1$. Therefore any fraction of elements of $\cal{B}$ of equal degree is constant, so $K_{(0,0)} = \bbk$. Then $\dim{K_{(k^\prime,r^\prime)}} \leq 1$ for any $(k^\prime,r^\prime)$, and since $K_{(nk,nr)} \neq 0$, it follows that $\dim{K_{(nk,nr)}} = 1$ as required.
\end{proof}

We will now work through an example to demonstrate this theorem.

\begin{example}
	With all notation as in the rest of this section, let \[H = \frac{1}{2}Q_1 + \frac{1}{3}Q_2 - \frac{1}{2}P.\] We will show that the ring $\cal{B}$ arising from $H$ as described in \thref{nonnormal} is not integrally closed, since $H$ has two non-integral coefficients at $Q_1$ and $Q_2$.
	
	First, we find $k, r$ and $n$. Since $\sum_i{\{a_i\}} = \frac{1}{2}+\frac{1}{3} = \frac{5}{6} < 1$, we set $k = a_1\cdot a_2 - 1 = 2\cdot 3 - 1 = 5.$ Now $r = \left\lfloor\sum_{i}{ka_i}\right\rfloor = \left\lfloor \frac{5}{2}+\frac{5}{3}\right\rfloor = \left\lfloor \frac{25}{6}\right\rfloor = 4$. Finally, $n = c_P\cdot c_1\cdot c_2 = 2\cdot 2\cdot 3 = 12$.
	
	This gives $H_{(k,r)} = \frac{5}{2}Q_1 + \frac{5}{3}Q_2 - 4P$. Then $\lfloor H_{(k,r)} \rfloor= 2Q_1+Q_2-4P$. This divisor has no global sections because it has negative degree, so we see that $\cal{B}_{(k,r)} = H^0(H_{(k,r)})/H^0(H_{(k,r+1)}) = 0$ as required.
	
	On the other hand, we have $H_{(nk,nr)} = 30Q_1 + 20Q_2 - 48P$, which has positive degree, so $\cal{B}_{(nk,nr)} \neq 0$. Indeed we have \[\cal{B}_{(nk,nr)} = H^0(30Q_1+20Q_2-48P)/H^0(30Q_1+20Q_2-49P).\] Let $\proj{1}$ have co-ordinates $x$ and $y$ and suppose $P = [0:1]$, $Q_1 = [1:0]$ and $Q_2 = [1:1]$. Then $\cal{B}_{(nk,nr)}$ is generated by the rational functions \[\frac{x^{48}}{y^{30}(x-y)^{18}}, \frac{x^{48}}{y^{29}(x-y)^{19}}, \frac{x^{48}}{y^{28}(x-y)^{20}}.\] However, note that because we quotient by $H^0(30Q_1+20Q_2-49P)$, we can show that: \[\frac{x^{48}}{y^{30}(x-y)^{18}} + \frac{x^{48}}{y^{29}(x-y)^{19}} = \frac{x^{49}}{y^{30}(x-y)^{20}} = 0,\] since the result is a section of that divisor. It follows that in $\cal{B}_{(nk,nr)}$ we have $\frac{y}{x-y} = -1$, and we will use this fact later.
	
	Now we choose $k^\prime$ and $r^\prime$. We can pick $k^\prime = 4$ since that gives $k^\prime H = 2Q_1 + \frac{4}{3}Q_2 - 2P$ which has degree $\frac{4}{3}$ and only one non-integral coefficient, so $\deg{\lfloor k^\prime H\rfloor} > 0$. Then $r^\prime = -\lfloor k^\prime a_P\rfloor = 2$. We have \[\cal{B}_{(k^\prime,r^\prime)} = H^0(2Q_1+Q_2-2P)/H^0(2Q_1+Q_2-3P)\] and \[\cal{B}_{(k+k^\prime,r+r^\prime)} = H^0(4Q_1+3Q_2-6P)/H^0(4Q_1+3Q_2-7P).\] Hence take $f = \frac{x^2}{y(x-y)} \in \cal{B}_{(k^\prime,r^\prime)}$ and $g = \frac{x^6}{y^4(x-y)^2} \in \cal{B}_{(k+k^\prime,r+r^\prime)}$ and let \[h = \frac{f}{g} = \frac{x^4}{y^3(x-y)} \in K_{(k,r)}.\] We know that $h \notin \cal{B}$ since $\cal{B}_{(k,r)} = 0$. However, we have \[h^n = \frac{x^{48}}{y^{36}(x-y)^{12}} = \frac{(x-y)^6}{y^6}\cdot \frac{x^{48}}{y^{30}(x-y)^{18}} = \frac{x^{48}}{y^{30}(x-y)^{18}} \in \cal{B}_{(nk,nr)},\] since $\frac{(x-y)}{y} = -1$ in $\cal{B}$ as seen above.
	
	It follows that $h$ is a root in $K$, the field of fractions of $\cal{B}$, of the monic polynomial \[z^n - \frac{x^{48}}{y^{30}(x-y)^{12}} \in \cal{B}[z].\] Hence $\cal{B}$ is not integrally closed.
\end{example}

\subsubsection{Divisor Correspondence}

Now we show that prime divisors over $X$ give rise to divisors on $\proj{1}$, and demonstrate how, subject to some conditions, \thref{nonnormal} can be applied to show that the test configurations corresponding to these prime divisors have non-normal central fibres. We refer back to the notation of \thref{primetc} and \thref{assgr}.

\begin{proposition}
	Let $X$ be a smooth $G$-variety of complexity one. The test configuration corresponding to a prime divisor $F$ over $X$ is special if and only if the graded ring $\cal{B}$ associated to the filtration on the section ring of $-K_X$ induced by $F$ is integrally closed.
\end{proposition}

\begin{proof}
	The test configuration is special if and only if the central fibre $X_0$ is normal. In this case we have $X_0 = \pi^{-1}(0) = \Proj{\cal{A}/(z)}$, where $\cal{A}/(z) = \bigoplus_{r \in \bb{Z}}{\cal{F}^rR/\cal{F}^{r+1}R} = \cal{B}$. Since $X_0 = \Proj{\cal{B}}$, and in fact $\cal{B}$ is the section ring of $(X_0,L_0)$, it follows that $X_0$ is normal, and the test configuration special, if and only if $\cal{B}$ is integrally closed.
\end{proof}

Now suppose $X$ is a smooth Fano $G$-variety of complexity one. Let $F$ be a non-central $G$-divisor over $X$. We know that the test configuration corresponding to $X$ is special if and only if the bigraded ring $\cal{B}$ defined above is integrally closed. We will use the $B$-quotient to allow ourselves to check this using divisors on $\proj{1}$. Recall that $\cal{B} = \bigoplus_{r \in \bb{Z}}{\cal{F}^rR/\cal{F}^{r+1}R}$, where $R = \bigoplus_{k \in \bb{Z}}{H^0(X,-kK_X)}$ and $\cal{F}^rR = \bigoplus_{k \in \bb{Z}}{\{f \in H^0(X,-kK_X) \mid \nu_F(f) \geq r\}}$.

We know that we can find a $B$-invariant representative of the class $-K_X$, and given this, the section rings $H^0(X,-kK_X)$ gain a $G$-module structure. The $B$-semi-invariants of weight $\lambda$ in this $G$-module are of the form $f_0e_\lambda$ where $f_0 \in K^B = \bbk(\proj{1})$. If $B(X)$ is the set of $B$-invariant divisors of $X$ and we have $-K_X = \sum_{D\in B(X)}{m_DD}$, then \[H^0(X,-K_X)^{(B)}_{\lambda} = \{f_0e_\lambda \mid f_0 \in K^B, \sum_{D \in B(X)}{[h_D\nu_{P_D}(f_0) + \langle \lambda, \ell_D\rangle + m_D]}D \geq 0\}.\] For a fixed weight $\lambda \in \Lambda$ we have $H^0(X,-K_X)^{(B)}_\lambda \cong H^0(\proj{1},H_\lambda)$, where \[H_\lambda = \sum_{P \in \proj{1}}{\left(\min_{P_D = P}{\frac{\langle \lambda, \ell_D\rangle + m_D}{h_D}}\right)P}.\] 

This is a well defined divisor on $\proj{1}$ (i.e. has no coefficients $\pm\infty)$ provided that (a): for any $P \in \proj{1}$ there exists a $B$-divisor $D$ on $X$ with $P_D = P$ and $h_D > 0$, and (b): $\lambda$ lies in the polyhedral domain \[\cal{P}(-K_X) = \{\lambda \in \Lambda \mid \langle \lambda,\ell_D\rangle \geq -m_D \ \text{for all} \ D \ \text{with} \ h_D = 0\}.\] Condition (a) holds by completeness of $X$, but we need to be careful about condition (b).

Recall the function \[A(-K_X,\lambda) = \sum_{P \in \proj{1}}{\left(\min_{P_D = P}{\frac{\langle \lambda,\ell_D\rangle +m_D}{h_D}}\right)}\] which computes the degree of $H_\lambda$, and the polyhedral domain \[\cal{P}_+(-K_X) = \{\lambda \in \cal{P} \mid A(-K_X,\lambda) \geq 0\}.\] In this notation, $H_\lambda$ is well-defined and has positive degree exactly when $\lambda$ lies in the relative interior of $\cal{P}_+(-K_X)$.

\begin{lemma}\thlabel{0poly}
	We may assume that $0$ is in the relative interior of $\cal{P}_+(-K_X)$.
\end{lemma}

\begin{proof}
	Since $-K_X$ is ample, we have seen that $-K_X^n = \vol(-K_X)$, and this quantity must be positive. We have also seen that $\vol(-K_X)$ can be expressed as an integral over $\lambda_{-K_X} + \cal{P}_+(-K_X)$. It follows that $\cal{P}_+(-K_X)$ has non-empty interior. Let $\lambda$ lie in the interior of $\cal{P}_+(-K_X)$, and replace $-K_X$ with the equivalent divisor $-K_X^\prime = -K_X +\div{e_\lambda}$.
	
	Suppose $h_D = 0$ for some $D \in B(X)$. We know that $\langle \lambda, \ell_D\rangle \geq -m_D$. The former is defined to be $\nu_D(e_\lambda)$, so we have $m_D + \nu_D(e_\lambda) = m^\prime_D \geq 0$, or equivalently $0 = \langle 0, \ell_D\rangle \geq -m_D^\prime$, i.e. $0 \in \cal{P}(-K_X^\prime)$.
	
	Likewise, since $\lambda \in \relint{\cal{P}_+(-K_X)}$, we have $A(-K_X,\lambda) > 0$. But \begin{align*}A(-K_X,\lambda) & = \sum_{P \in \proj{1}}{\left(\min_{P_D = P}{\frac{\langle \lambda,\ell_D\rangle +m_D}{h_D}}\right)} \\ & = \sum_{P \in \proj{1}}{\left(\min_{P_D = P}{\frac{m^\prime_D}{h_D}}\right)} = A(-K_X^\prime,0).\end{align*} So $A(-K_X^\prime,0) > 0$ and the result follows.
\end{proof}

Now looking back to $\cal{B}$, we have \[(\cal{F}^r)^{(B)} = \bigoplus_{k \in \bb{Z}}{\left\{f \in \bigoplus_{\lambda \in \Lambda}{H^0(X,-kK_X)^{(B)}_\lambda} \mid \nu_F(f) \geq r\right\}}.\] By discussions above, we can rewrite this as \[(\cal{F}^r)^{(B)} = \bigoplus_{\lambda \in \Lambda}{\bigoplus_{k \in \bb{Z}}{\{f \in H^0(\proj{1},kH_\lambda) \mid \nu_F(f) \geq r\}}}.\] Since $\nu_F(f) = h_F\ord_{P_F}(f) + \langle \lambda,\ell_F\rangle$, we have \[(\cal{F}^r)^{(B)}_{(\lambda,k)} = \left\{f \in H^0(kH_\lambda) \mid \ord_P(f) \geq \frac{r-\langle \lambda,\ell_F\rangle}{h_F}\right\}.\]

We have shown that if $\deg{H_\lambda} > 0$ and $H_\lambda$ has non-integral coefficients at two points other than $P_F$, the ring $\cal{B}^{h_F}(H_\lambda)$ is not integrally closed. The latter is the sum over $(k, r) \in \bb{Z} \oplus \bb{Z}$ of \[\cal{B}^{h_F}(H_\lambda)_{(k,r)} = \frac{\left\{f \in H^0(kH_\lambda) \mid \ord_{P_F}(f) \geq \frac{r}{h_F}\right\}}{\left\{f \in H^0(kH_\lambda) \mid \ord_{P_F}(f) \geq \frac{r+1}{h_F}\right\}}.\] Note that $\cal{B}^{h_F}(H_\lambda)_{(k,r-\langle \lambda,\ell_F\rangle)} = (\cal{F}^r)^{(B)}_{(\lambda,k)}/(\cal{F}^{r+1})^{(B)}_{(\lambda,k)}$ as defined above. The shift of degrees by $\langle \lambda,\ell_F\rangle$ can be ignored as we sum over $\bb{Z}$ either way.

We can now write \begin{align*} \cal{B}^{(B)} = \bigoplus_{\lambda \in \Lambda}{\bigoplus_{r \in \bb{Z}}{\bigoplus_{k \in \bb{Z}}{\cal{B}^{(B)}_{(\lambda,k,r)}}}} & = \bigoplus_{\lambda \in \Lambda}{\bigoplus_{r \in \bb{Z}}{\bigoplus_{k \in \bb{Z}}{\cal{B}^{h_F}(H_\lambda)_{(k,r)}}}} \\ & = \bigoplus_{\lambda \in \Lambda}{\cal{B}^{h_F}(H_\lambda)}.\end{align*}

\begin{lemma}\thlabel{intcl}
	Let $A$ be an integral domain with field of fractions $K$, and let $K^\prime$ be a subfield of $K$. Let $B = A \inter K^\prime$. If $B$ is not integrally closed in $K^\prime$, then $A$ is not integrally closed in $K$.
\end{lemma}

\begin{proof}
	If $B$ is not integrally closed in $K^\prime$, there exists a monic polynomial in $B[x]$ with a root $f \in K^\prime$ which does not lie in $B$. Since $K^\prime \sub K$ and $B[x] \sub A[x]$, we can also view $f$ as a root in $K$ of a monic polynomial in $A[x]$. If $f \in A$, then since $f \in K^\prime$, we have $f \in A \inter K^\prime = B$, a contradiction. Hence $f \notin A$ and $A$ is therefore not integrally closed in $K$.
\end{proof}

\begin{theorem}
	If there exists $\lambda \in \cal{P}_+(-K_X)$ such that $H_\lambda$ has two non-integral coefficients at points other than $P_F$, then $\cal{B}$ is not integrally closed.
\end{theorem}

\begin{proof}
	By \thref{intcl} it suffices to show that $\cal{B}^{(B)}_0 = \cal{B} \inter K^B$ is not integrally closed. We have shown that $\cal{B}^{(B)} = \bigoplus_{\lambda \in \Lambda}{\cal{B}^{h_F}(H_\lambda)}$, so in particular $\cal{B}^{(B)}_0 = \cal{B}^{h_F}(H_0)$. Hence if $\cal{B}^{h_F}(H_0)$ is not integrally closed, then neither is $\cal{B}$. We have proved already that $\cal{B}^{h_F}(H_0)$ is not integrally closed when $H_0$ has positive degree and two non-integral points distinct from $P_F$. We may assume that $H_0$ has positive degree by \thref{0poly}.
	
	If $H_\lambda$ has non-integral coefficients at two points other than $P_F$, then replace $-K_X$ with $-K_X^\prime = -K_X + \div{e_\lambda}$. Then $H_0^\prime = H_\lambda$ and the result follows using $H_0^\prime$ instead of $H_0$.
\end{proof}

To summarise the results of this subsection, we have:

\begin{corollary}\thlabel{nonint}
	Let $X$ be a smooth Fano $G$-variety of complexity one with anticanonical divisor $-K_X$. Let $F$ be a non-central prime divisor over $X$ corresponding to a point $P_F \in \proj{1}$ on the $B$-quotient. If there exists $\lambda \in \cal{P}_+(-K_X)$ such that $H_\lambda$ has two non-integral coefficients at points other than $P_F$, then the test configuration corresponding to $F$ has non-normal special fibre.
\end{corollary}

In \Cref{apply} we will show that the hypotheses of \thref{nonint} hold for all smooth Fano $\sl{2}$ threefolds. This will allow us to conclude the following:

\begin{theorem}\thlabel{noncentnorm}
	Let $X$ be one of the $\sl{2}$-threefolds mentioned in \thref{threefoldsKstable}. Let $F$ be a non-central prime divisor over $X$ corresponding to a point $P_F \in \proj{1}$. If $X$ has subregular colours lying over two points in $\proj{1}$ distinct from $P_F$, then the test configuration corresponding to $F$ has non-normal central fibre.
\end{theorem}

This result now allows us to prove parts (ii) and (iii) of \thref{Kstablecentral}.

\subsection{Action Interchanging Two Points}

\begin{proof}[Proof of \thref{Kstablecentral}(ii)]
	Suppose a finite subgroup $A \sub \aut{X}$ acts on $\proj{1}$, interchanging two points $P$ and $Q$ corresponding to subregular colours of $X$ and that the $B$-quotient is equivariant with respect to the $A$-action. We have an action on $X$ of an extension $G^\prime$ of $G$ by $A$. Any non-central $G$-invariant prime divisor $F$ over $X$ can only be $G^\prime$-invariant if its corresponding point $P_F \in \proj{1}$ is fixed by $A$. Since $P$ and $Q$ are not fixed by $A$, they are distinct from $P_F$, and since each has a subregular colour lying over it, \thref{noncentnorm} applies, and the test configuration corresponding to $F$ is not special. Therefore we may show that $X$ is $K$-polystable by checking only central divisors.
\end{proof}

\subsection{Three or More Subregular Colours}

\begin{proof}[Proof of \thref{Kstablecentral}(iii)]
	If $X$ has subregular colours lying over three or more distinct points of $\proj{1}$, then for any non-central prime divisor $F$ over $X$ corresponding to a point $P_F$ in $\proj{1}$, there always exist at least two subregular colours lying over two points distinct from $P_F$ and from each other. Then \thref{noncentnorm} applies and we need only check the $\beta$-invariant of central divisors.
\end{proof}

\section{Homogeneous Spaces of $\sl{2}$}\label{sl2}

In this subsection we will describe the coloured data of the complexity one homogeneous spaces of $\sl{2}$, as calculated by Timashev \cite{tim1}. We will show the calculation in full for $\sl{2}$ itself, while simply presenting the pictures for the remaining homogeneous spaces. Note that a complexity one homogeneous space for $\sl{2}$ must be of the form $\sl{2}/H$ for $H$ finite, and that the finite subgroups of $\sl{2}$ are as follows: the cyclic groups $\bb{Z}_n$ for $n \in \bb{N}$, the binary dihedral groups $\tilde{D}_n$ for $n \in \bb{N}$, the binary tetrahedral group $\tilde{T}$, the binary cubic group $\tilde{C}$ and the binary icosahedral group $\tilde{I}$.

\subsection{$\normalfont{SL}_2$}

Consider the action of $G = \sl{2}$ on itself by left multiplication of matrices, let $K = \bbk(G)$ and choose subgroups $B, U$ and $T$ of $G$ consisting of upper triangular, upper unitriangular and diagonal matrices, respectively. The action gives the homogeneous space $G/H = \sl{2}/\{e\}$. The $B$-orbit of the identity is $B$ itself, which is a maximal orbit of codimension 1, so this is a complexity one homogeneous space. Note that $\frak{X}(B) = \bb{Z}\alpha$ where $\alpha$ is the character $\begin{psmallmatrix} a & b \\ 0 & 1/a\end{psmallmatrix} \mapsto a$. For $g = \begin{psmallmatrix} x & y \\ z & w\end{psmallmatrix}$, the functions $g \mapsto z$ and $g \mapsto w$ are semi-invariant of weight $\alpha$, so $\Lambda = \bb{Z}\alpha$ as well. These semi-invariants generate the space $M = \bbk[G]_{\alpha}^{(B)}$, and $K^B =\bbk(z/w)$. Fix a splitting $e \colon \Lambda \to K^{(B)}$ given by $e_{\alpha} = z$. Then all semi-invariants are of the form $fe^k_{\alpha}$ where $f \in K^B$ and $k \in \bb{Z}$.

We are interested in $G$-valuations and valuations of colours of $K$, which are determined by their restrictions to $K^{(B)}$, which are in turn determined by their restrictions to $K^B$ and a functional $\ell \colon \Lambda \to \bb{Q}$.

The rational $B$-quotient map is determined by the invariant $z/w$ and thus looks like $\pi \colon \sl{2} \to \bb{P}^1$, $g \mapsto [z:w]$. The regular semi-invariants thus lie in the space $M$ generated by $z$ and $w$. Since no other semi-invariants can divide $z$ or $w$, there are no subregular semi-invariants. Likewise there are also no central colours. The fibre of a point $p = [\alpha \colon \beta] \in \proj{1}$ is the regular $B$-divisor $D_p = \cal{Z}(\beta z - \alpha w)$, and all $B$-stable divisors are of this form. The chosen splitting $e$ marks out the point $\infty = [0:1]$ with $D_\infty = \cal{Z}(e_\alpha)$ as distinguished. Non-distinguished regular colours $D_p$ for $p\neq \infty$ sit at $(p,\ell,h) = (p,0,1) \in \cal{H}_p$, and $D_\infty$ has $\ell = \nu_\infty(e_\alpha) = 1$, so sits at $(\infty,1,1) \in \cal{H}_\infty$. 

We can calculate $\cal{V}$ using the method of formal curves (see \cite[\S 24]{tim}): fix $m \in \bb{Z}$ and $u(t) \in \bbk(\!(t)\!)$ and let \[x(t) = \begin{pmatrix} t^m & u(t) \\ 0 & t^{-m} \end{pmatrix} \in \frak{X}^*(T)\cdot U(\bbk(\!(t)\!)) \] where $\ord_tu(t) = n \leq -m$. Then any non-central $G$-valuation is proportional to $\nu_{x(t)}$, where $\nu_{x(t)}(f) = \ord_t(f(g\cdot x(t)))$ for any $f \in K^{(B)}$ and generic $g \in G$. Let $p = [\alpha:\beta]$ and \[d_p = \nu_{x(t)}(\beta z - \alpha w) = \ord_t((\beta t^m-\alpha u(t))z - \alpha t^{-m}w).\]

The value of $d_p$ is constant along $\bb{P}^1$ except at the distinguished point, where it jumps by some $h \in \bb{Q}_{\geq 0}$, so that $\nu$ is represented in hyperspace by $(x,\ell,h)$, where $\ell = \nu_{x(t)}(e_\alpha)$. 

Note that for any $p$, $d_p \in [m,-m]$. Now suppose that $m \leq n$. We have $d_p \geq \min{\{\ord_t(\beta t^m-\alpha u(t)),-m\}} = \ord_t(\beta t^m-\alpha u(t)) \geq \min{\{m,n\}} = m$. Now if $d_p > m$, we have $d_p \in (m,-m]$. Otherwise $d_p = m$. Since $h$ is the difference between the maximum possible value of $d_p$ (which is $-m$) and the minimum, which we see lies in the interval $[m,-m)$, we have $h \in (0,-2m]$. Finally, $\ell$ is given by the value of $d_p$, which at non-distinguished points is $m$ and at the distinguished point is $m + h$.

In the case $m > n$, we have $d_p = n$ when $\alpha \neq 0$, and when $\alpha = 0$ (at the distinguished point), we have $d_p = m$. Hence $h = m - n$, $\ell = n$ for nondistinguished points and $\ell = n + h$ for $\infty$.

In either case we have $h > 0$ and the possible $(\ell,h)$ are defined for $p \neq \infty$ by the inequality $2\ell + h \leq 0$, and for $p = \infty$ by $2\ell-h \leq 0$. Re-including the central valuations allows $h \geq 0$. Thus we have valuation cones $\cal{V}_p = \{(\ell,h) \in \cal{H}_p \mid 2\ell + h \leq 0, h \geq 0\}$ for $p \neq \infty$, and $\cal{V}_{\infty} = \{(\ell,h) \in \cal{H}_{\infty}\mid 2\ell-h\leq 0, h \geq 0\}$. The picture of the hyperspace is thus:

\begin{center}
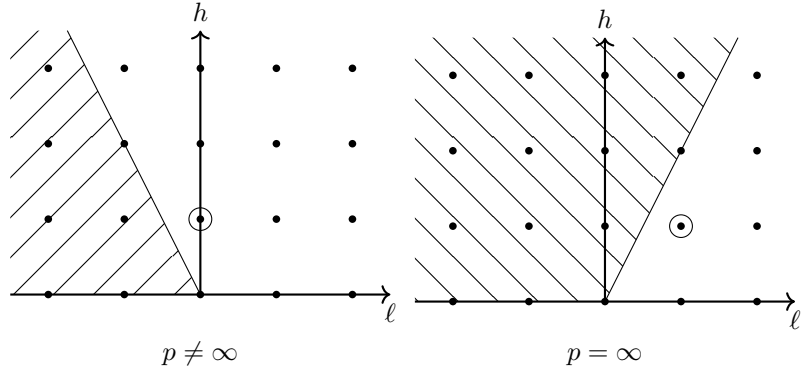
\begin{tikzpicture}
	\foreach \x in {-2,-1,0,1,2} {
		\foreach \y in {0,1,2,3} {
			\fill[color=black] (\x,\y) circle (0.05);
		}
	};
	\draw[thick][->] (-2.5,0) -- (2.5,0);
	\draw[thick][->] (0,0) -- (0,3.5);
	\draw[domain = -1.75:0] plot (\x, {-2*\x}); 
	\node[below] at (2.5,0) {$\ell$};
	\node[above] at (0,3.5) {$h$};
	\node[below] at (0,-0.5) {$p \neq \infty$};
	\draw (0,1) circle [radius=0.15];
	\begin{scope}[on background layer]
	\fill[pattern=my north east lines](-2.5,3.5) to (-1.75,3.5) to (0,0) to (-2.5,0);
	\end{scope}
	\end{tikzpicture}
	\begin{tikzpicture}
	\foreach \x in {-2,-1,0,1,2} {
		\foreach \y in {0,1,2,3} {
			\fill[color=black] (\x,\y) circle (0.05);
		}
	};
	\draw[thick][->] (-2.5,0) -- (2.5,0);
	\draw[thick][->] (0,0) -- (0,3.5);
	\draw[domain = 0:1.75] plot (\x, {2*\x}); 
	\node[below] at (2.5,0) {$\ell$};
	\node[above] at (0,3.5) {$h$};
	\node[below] at (0,-0.5) {$p = \infty$};
	\draw (1,1) circle [radius=0.15];
	\begin{scope}[on background layer]
	\fill[pattern=my north west lines](-2.5,0) to (-2.5,3.5) to (1.75,3.5) to (0,0);
	\end{scope}
	\end{tikzpicture}\end{center}
\begingroup
\captionof{figure}{Coloured data of $\sl{2}$}
\endgroup

\noindent where dashed areas denote the valuation cones and colours are denoted by unfilled circles.

\subsection{$\normalfont{SL}_2/\bb{Z}_m$}\label{SL2Zm}

The coloured data of $\sl{2}/\bb{Z}_m$ has a distinguished point as above, but the two points $0, \infty \in \proj{1}$ are also set apart since they are fixed by the $\bb{Z}_m$ action on $\proj{1}$ obtained via the $B$-quotient.

\begin{center}\resizebox{\linewidth}{!}{
		\begin{tikzpicture}
		\draw[thick][->] (-2.5,0) -- (2.5,0);
		\draw[thick][->] (0,0) -- (0,3.5);
		\draw[domain = -1.75:0] plot (\x, {-2*\x}); 
		\draw[fill=black] (0,3) circle [color=black,radius=0.05];
		\node[right] at (0,3) {$\overline{m}$};
		\draw[fill=black] (-1,0) circle [color=black,radius=0.05];
		\node[below] at (-0.75,0) {$-\frac{\overline{m}-1}{2}$};
		\draw[fill=black] (-1.5,0) circle [color=black,radius=0.05];
		\draw[fill=black] (-1.5,3) circle [color=black,radius=0.05];
		\node[below] at (-1.75,0) {$-\frac{\overline{m}}{2}$};
		\node[below] at (2.5,0) {$\ell$};
		\node[above] at (0,3.5) {$h$};
		\node[below] at (0,-0.5) {$p = 0,\infty$};
		\draw (-1,3) circle [radius=0.15];
		\begin{scope}[on background layer]
		\fill[pattern=my north east lines](-2.5,3.5) to (-1.75,3.5) to (0,0) to (-2.5,0);
		\end{scope}
		\end{tikzpicture}
		\begin{tikzpicture}
		\draw[thick][->] (-2.5,0) -- (2.5,0);
		\draw[thick][->] (0,0) -- (0,3.5);
		\draw[domain = -1.75:0] plot (\x, {-2*\x}); 
		\node[below] at (2.5,0) {$\ell$};
		\node[above] at (0,3.5) {$h$};
		\draw[fill=black] (0,2) circle [color=black,radius=0.05];
		\node[right] at (0,2) {$2$};
		\draw[fill=black] (-1,2) circle [color=black,radius=0.05];
		\draw[fill=black] (-1,0) circle [color=black,radius=0.05];
		\node[below] at (-1,0) {$-1$};
		\node[below] at (0,-0.5) {$p$ general};
		\draw (0,1) circle [radius=0.15];
		\begin{scope}[on background layer]
		\fill[pattern=my north east lines](-2.5,3.5) to (-1.75,3.5) to (0,0) to (-2.5,0);
		\end{scope}
		\end{tikzpicture}
		\begin{tikzpicture}
		\draw[thick][->] (-2.5,0) -- (2.5,0);
		\draw[thick][->] (0,0) -- (0,3.5);
		\draw[domain = 0:1.75] plot (\x, {2*\x}); 
		\node[below] at (2.5,0) {$\ell$};
		\node[above] at (0,3.5) {$h$};
		\draw[fill=black] (0,2) circle [color=black,radius=0.05];
		\node[right] at (0,2) {$2$};
		\draw[fill=black] (1,2) circle [color=black,radius=0.05];
		\draw[fill=black] (1,0) circle [color=black,radius=0.05];
		\node[below] at (1,0) {$1$};
		\node[below] at (0,-0.5) {$p$ distinguished};
		\draw (1,1) circle [radius=0.15];
		\begin{scope}[on background layer]
		\fill[pattern=my north west lines](-2.5,0) to (-2.5,3.5) to (1.75,3.5) to (0,0);
		\end{scope}
		\end{tikzpicture}}\end{center}
\begingroup

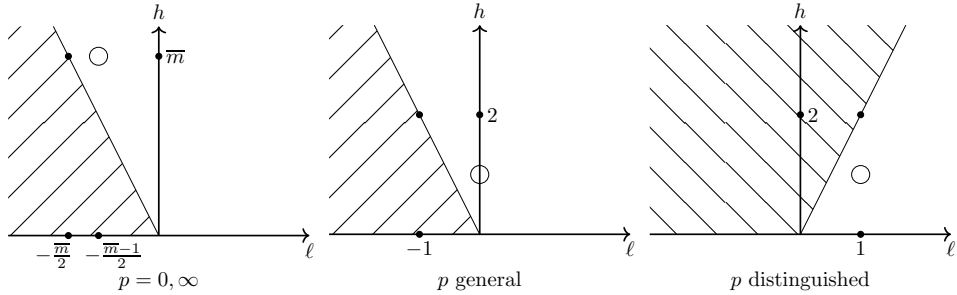
\captionof{figure}{Coloured data of $\sl{2}/\bb{Z}_m$}
\endgroup

\noindent where $\cal{Q} = \Lambda^*$, the $\ell$-axis, is given by $\bb{Z}$ if $m$ is odd and $\frac{1}{2}\bb{Z}$ if $m$ is even.

\subsection{$\normalfont{SL}_2/\tilde{D}_n$}\label{dihedral}

The binary dihedral and polyhedral groups have distinguished points corresponding to the faces, edges and vertices of their corresponding polygons.

\begin{center}\resizebox{\linewidth}{!}{\begin{tikzpicture}
		
		\draw[thick][->] (-2.5,0) -- (2.5,0);
		\draw[thick][->] (0,0) -- (0,3.5);
		\draw[domain = -2.5:0] plot (\x, {-\x}); 
		\draw[fill=black] (0,2) circle [color=black,radius=0.05];
		\node[right] at (0,2) {$n$};
		\draw[fill=black] (-1,0) circle [color=black,radius=0.05];
		\node[below] at (-1,0) {$1-n$};
		\draw[fill=black] (-2,0) circle [color=black,radius=0.05];
		\node[below] at (-2,0) {$-n$};
		\node[below] at (2.5,0) {$\ell$};
		\node[above] at (0,3.5) {$h$};
		\node[below] at (0,-0.5) {$p=p_f$};
		\draw (-1,2) circle [radius=0.15];
		\begin{scope}[on background layer]
		\fill[pattern=my north east lines](-2.5,2.5) to (0,0) to (-2.5,0);
		\end{scope}
		\end{tikzpicture}
		\begin{tikzpicture}
		\draw[thick][->] (-2.5,0) -- (2.5,0);
		\draw[thick][->] (0,0) -- (0,3.5);
		\node[below] at (2.5,0) {$\ell$};
		\node[above] at (0,3.5) {$h$};
		\draw[fill=black] (0,2) circle [color=black,radius=0.05];
		\node[right] at (0,2) {$2$};
		\draw[fill=black] (1,0) circle [color=black,radius=0.05];
		\node[below] at (1,0) {$1$};
		\node[below] at (0,-0.5) {$p=p_e,p_v$};
		\draw (1,2) circle [radius=0.15];
		\begin{scope}[on background layer]
		\fill[pattern=my north west lines](-2.5,0) to (-2.5,3.5) to (0,3.5) to (0,0);
		\end{scope}
		\end{tikzpicture}
		\begin{tikzpicture}
		\draw[thick][->] (-2.5,0) -- (2.5,0);
		\draw[thick][->] (0,0) -- (0,3.5);
		\draw[domain = -2.5:0] plot (\x, {-\x}); 
		\node[below] at (2.5,0) {$\ell$};
		\node[above] at (0,3.5) {$h$};
		\draw[fill=black] (0,1) circle [color=black,radius=0.05];
		\node[right] at (0.2,1) {$1$};
		\draw[fill=black] (-1,0) circle [color=black,radius=0.05];
		\node[below] at (-1,0) {$-1$};
		\node[below] at (0,-0.5) {$p \neq p_f,p_v,p_e$};
		\draw (0,1) circle [radius=0.15];
		\begin{scope}[on background layer]
		\fill[pattern=my north east lines](-2.5,2.5) to (0,0) to (-2.5,0);
		\end{scope}
		\end{tikzpicture}}\end{center}
\begingroup

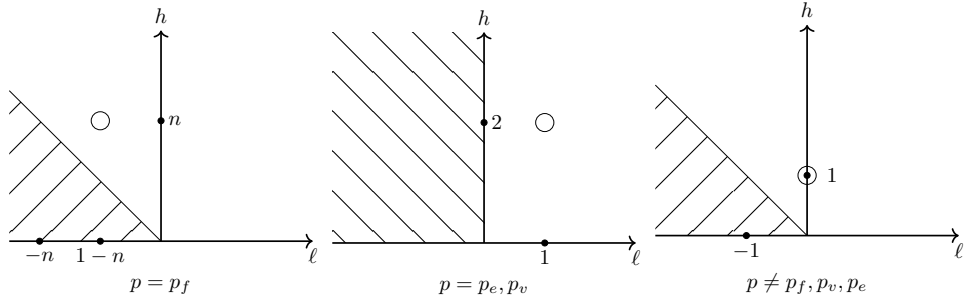
\captionof{figure}{Coloured data of $\sl{2}/\tilde{D}_n$}
\endgroup

\subsection{$\normalfont{SL}_2/\tilde{T}$}\label{tetrahedral}

\begin{center}\resizebox{\linewidth}{!}{\begin{tikzpicture}
		\draw[thick][->] (-2.5,0) -- (2.5,0);
		\draw[thick][->] (0,0) -- (0,3.5);
		\draw[domain = -2.5:0] plot (\x, {-\x}); 
		\draw[fill=black] (0,2) circle [color=black,radius=0.05];
		\node[right] at (0,2) {$2$};
		\draw[fill=black] (-1,0) circle [color=black,radius=0.05];
		\node[below] at (-1,0) {$-1$};
		\draw[fill=black] (-1,0) circle [color=black,radius=0.05];
		\node[below] at (2.5,0) {$\ell$};
		\node[above] at (0,3.5) {$h$};
		\node[below] at (0,-0.5) {$p = p_e$};
		\draw (-1,2) circle [radius=0.15];
		\begin{scope}[on background layer]
		\fill[pattern=my north east lines](-2.5,0) to (-2.5,2.5) to (0,0);
		\end{scope}
		\end{tikzpicture}
		\begin{tikzpicture}
		\draw[thick][->] (-2.5,0) -- (2.5,0);
		\draw[thick][->] (0,0) -- (0,3.5);
		\node[below] at (2.5,0) {$\ell$};
		\node[above] at (0,3.5) {$h$};
		\draw[fill=black] (1,0) circle [color=black,radius=0.05];
		\node[below] at (1,0) {$1$};
		\draw[fill=black] (0,3) circle [color=black,radius=0.05];
		\node[right] at (0,3) {$3$};
		\node[below] at (0,-0.5) {$p = p_v,p_f$};
		\draw (1,3) circle [radius=0.15];
		\begin{scope}[on background layer]
		\fill[pattern=my north west lines](-2.5,0) to (-2.5,3.5) to (0,3.5) to (0,0);
		\end{scope}
		\end{tikzpicture}
		\begin{tikzpicture}
		\draw[thick][->] (-2.5,0) -- (2.5,0);
		\draw[thick][->] (0,0) -- (0,3.5);
		\draw[domain = -2.5:0] plot (\x, {-\x}); 
		\node[below] at (2.5,0) {$\ell$};
		\node[above] at (0,3.5) {$h$};
		\draw[fill=black] (0,1) circle [color=black,radius=0.05];
		\node[right] at (0.2,1) {$1$};
		\draw[fill=black] (-1,0) circle [color=black,radius=0.05];
		\node[below] at (-1,0) {$-1$};
		\node[below] at (0,-0.5) {$p \neq p_f,p_v,p_e$};
		\draw (0,1) circle [radius=0.15];
		\begin{scope}[on background layer]
		\fill[pattern=my north east lines](-2.5,2.5) to (-2.5,0) to (0,0);
		\end{scope}
		\end{tikzpicture}}\end{center}
\begingroup

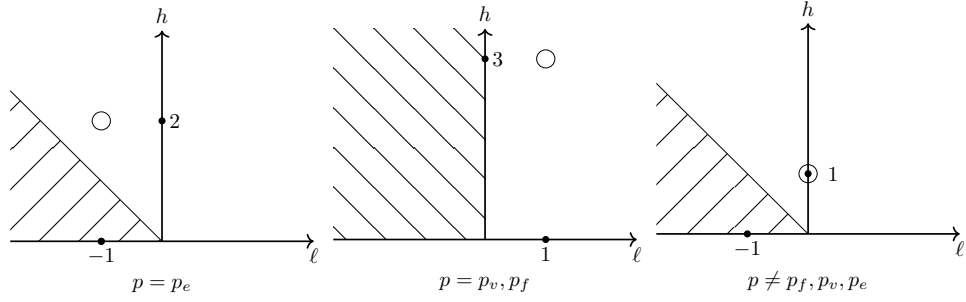
\captionof{figure}{Coloured data of $\sl{2}/\tilde{T}$}
\endgroup

\subsection{$\normalfont{SL}_2/\tilde{C}$}\label{cubic}

\begin{center}\begin{tikzpicture}
	\draw[thick][->] (-2.5,0) -- (2.5,0);
	\draw[thick][->] (0,0) -- (0,3.5);
	\draw[domain = -2.5:0] plot (\x, {-\x}); 
	\draw[fill=black] (0,2) circle [color=black,radius=0.05];
	\node[right] at (0,2) {$2$};
	\draw[fill=black] (-1,0) circle [color=black,radius=0.05];
	\node[below] at (-1,0) {$-1$};
	\draw[fill=black] (-1,0) circle [color=black,radius=0.05];
	\node[below] at (2.5,0) {$\ell$};
	\node[above] at (0,3.5) {$h$};
	\node[below] at (0,-0.5) {$p = p_e$};
	\draw (-1,2) circle [radius=0.15];
	\begin{scope}[on background layer]
	\fill[pattern=my north east lines](-2.5,0) to (-2.5,2.5) to (0,0);
	\end{scope}
	\end{tikzpicture}
	\begin{tikzpicture}
	\draw[thick][->] (-2.5,0) -- (2.5,0);
	\draw[thick][->] (0,0) -- (0,3.5);
	\node[below] at (2.5,0) {$\ell$};
	\node[above] at (0,3.5) {$h$};
	\draw[fill=black] (1,0) circle [color=black,radius=0.05];
	\node[below] at (1,0) {$1$};
	\draw[fill=black] (0,3) circle [color=black,radius=0.05];
	\node[right] at (0,3) {$4$};
	\node[below] at (0,-0.5) {$p = p_f$};
	\draw (1,3) circle [radius=0.15];
	\begin{scope}[on background layer]
	\fill[pattern=my north west lines](-2.5,0) to (-2.5,3.5) to (0,3.5) to (0,0);
	\end{scope}
	\end{tikzpicture}\end{center}

\begin{center}
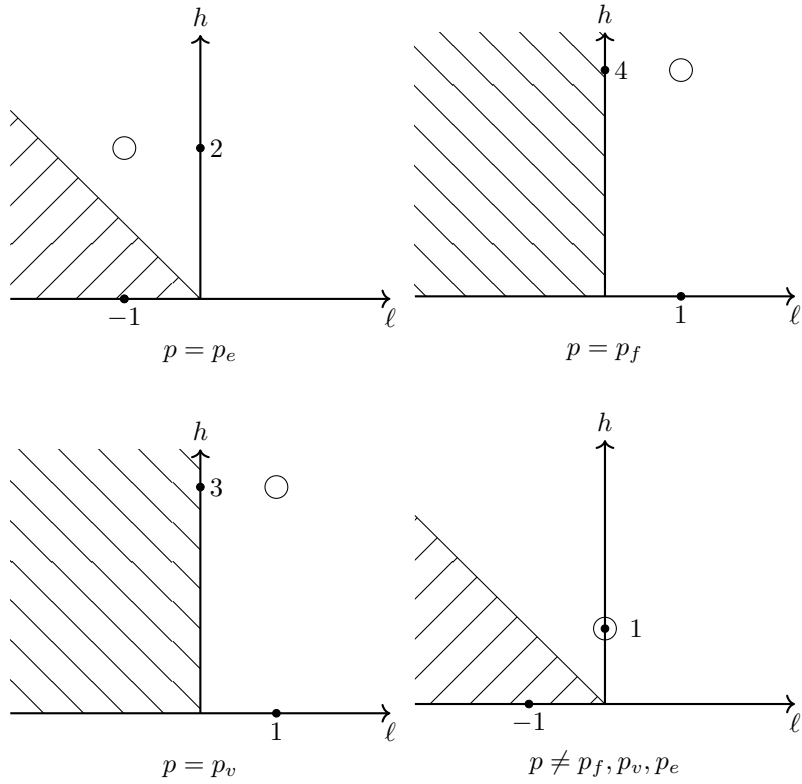
\begin{tikzpicture}
	\draw[thick][->] (-2.5,0) -- (2.5,0);
	\draw[thick][->] (0,0) -- (0,3.5);
	\node[below] at (2.5,0) {$\ell$};
	\node[above] at (0,3.5) {$h$};
	\draw[fill=black] (1,0) circle [color=black,radius=0.05];
	\node[below] at (1,0) {$1$};
	\draw[fill=black] (0,3) circle [color=black,radius=0.05];
	\node[right] at (0,3) {$3$};
	\node[below] at (0,-0.5) {$p = p_v$};
	\draw (1,3) circle [radius=0.15];
	\begin{scope}[on background layer]
	\fill[pattern=my north west lines](-2.5,0) to (-2.5,3.5) to (0,3.5) to (0,0);
	\end{scope}
	\end{tikzpicture}
	\begin{tikzpicture}
	\draw[thick][->] (-2.5,0) -- (2.5,0);
	\draw[thick][->] (0,0) -- (0,3.5);
	\draw[domain = -2.5:0] plot (\x, {-\x}); 
	\node[below] at (2.5,0) {$\ell$};
	\node[above] at (0,3.5) {$h$};
	\draw[fill=black] (0,1) circle [color=black,radius=0.05];
	\node[right] at (0.2,1) {$1$};
	\draw[fill=black] (-1,0) circle [color=black,radius=0.05];
	\node[below] at (-1,0) {$-1$};
	\node[below] at (0,-0.5) {$p \neq p_f,p_v,p_e$};
	\draw (0,1) circle [radius=0.15];
	\begin{scope}[on background layer]
	\fill[pattern=my north east lines](-2.5,2.5) to (-2.5,0) to (0,0);
	\end{scope}
	\end{tikzpicture}\end{center}
\begingroup
\captionof{figure}{Coloured data of $\sl{2}/\tilde{C}$}
\endgroup

\subsection{$\normalfont{SL}_2/\tilde{I}$}\label{icosahedral}

\begin{center}\begin{tikzpicture}
	\draw[thick][->] (-2.5,0) -- (2.5,0);
	\draw[thick][->] (0,0) -- (0,3.5);
	\draw[domain = -2.5:0] plot (\x, {-\x}); 
	\draw[fill=black] (0,2) circle [color=black,radius=0.05];
	\node[right] at (0,2) {$2$};
	\draw[fill=black] (-1,0) circle [color=black,radius=0.05];
	\node[below] at (-1,0) {$-1$};
	\draw[fill=black] (-1,0) circle [color=black,radius=0.05];
	\node[below] at (2.5,0) {$\ell$};
	\node[above] at (0,3.5) {$h$};
	\node[below] at (0,-0.5) {$p = p_e$};
	\draw (-1,2) circle [radius=0.15];
	\begin{scope}[on background layer]
	\fill[pattern=my north east lines](-2.5,0) to (-2.5,2.5) to (0,0);
	\end{scope}
	\end{tikzpicture}
	\begin{tikzpicture}
	\draw[thick][->] (-2.5,0) -- (2.5,0);
	\draw[thick][->] (0,0) -- (0,3.5);
	\node[below] at (2.5,0) {$\ell$};
	\node[above] at (0,3.5) {$h$};
	\draw[fill=black] (1,0) circle [color=black,radius=0.05];
	\node[below] at (1,0) {$1$};
	\draw[fill=black] (0,3) circle [color=black,radius=0.05];
	\node[right] at (0,3) {$3$};
	\node[below] at (0,-0.5) {$p = p_f$};
	\draw (1,3) circle [radius=0.15];
	\begin{scope}[on background layer]
	\fill[pattern=my north west lines](-2.5,0) to (-2.5,3.5) to (0,3.5) to (0,0);
	\end{scope}
	\end{tikzpicture}\end{center}

\begin{center}
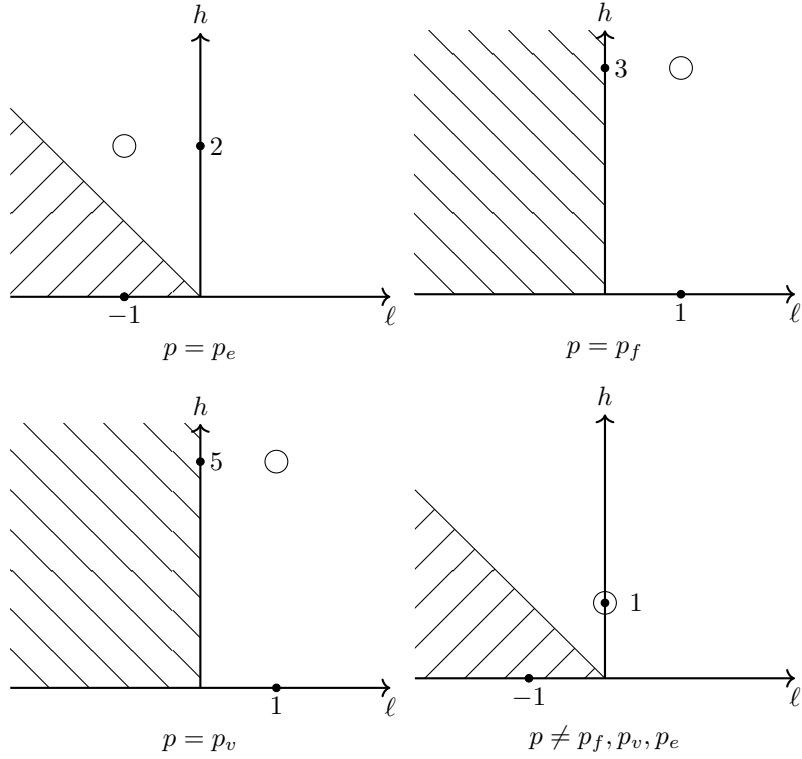
\begin{tikzpicture}
	\draw[thick][->] (-2.5,0) -- (2.5,0);
	\draw[thick][->] (0,0) -- (0,3.5);
	\node[below] at (2.5,0) {$\ell$};
	\node[above] at (0,3.5) {$h$};
	\draw[fill=black] (1,0) circle [color=black,radius=0.05];
	\node[below] at (1,0) {$1$};
	\draw[fill=black] (0,3) circle [color=black,radius=0.05];
	\node[right] at (0,3) {$5$};
	\node[below] at (0,-0.5) {$p = p_v$};
	\draw (1,3) circle [radius=0.15];
	\begin{scope}[on background layer]
	\fill[pattern=my north west lines](-2.5,0) to (-2.5,3.5) to (0,3.5) to (0,0);
	\end{scope}
	\end{tikzpicture}
	\begin{tikzpicture}
	\draw[thick][->] (-2.5,0) -- (2.5,0);
	\draw[thick][->] (0,0) -- (0,3.5);
	\draw[domain = -2.5:0] plot (\x, {-\x}); 
	\node[below] at (2.5,0) {$\ell$};
	\node[above] at (0,3.5) {$h$};
	\draw[fill=black] (0,1) circle [color=black,radius=0.05];
	\node[right] at (0.2,1) {$1$};
	\draw[fill=black] (-1,0) circle [color=black,radius=0.05];
	\node[below] at (-1,0) {$-1$};
	\node[below] at (0,-0.5) {$p \neq p_f,p_v,p_e$};
	\draw (0,1) circle [radius=0.15];
	\begin{scope}[on background layer]
	\fill[pattern=my north east lines](-2.5,2.5) to (-2.5,0) to (0,0);
	\end{scope}
	\end{tikzpicture}\end{center}
\begingroup
\captionof{figure}{Coloured data of $\sl{2}/\tilde{I}$}
\endgroup

\section{Coloured Hyperfans of Smooth Fano $\sl{2}$-Threefolds}\label{combo}

Here we calculate the coloured hyperfans of the smooth Fano $\sl{2}$-threefolds we are interested in.

\subsection{$\sl{2}$-Actions on Symmetric Powers}\label{symaction}

Many of the $\sl{2}$-actions on threefolds to be considered here will be induced by an $\sl{2}$-action on $\proj{n}$ for some $n$, often in the case where $\proj{n}$ is realised as the projectivisation of a symmetric power of $\bbk^2$. We describe these actions and some of their properties here, and will use the results throughout the remainder of this section.

\begin{proposition}\thlabel{simpleSL2}
	Let $G = \sl{2}$. Fix a Borel subgroup $B$ consisting of the upper triangular matrices in $G$. Then $\frak{X}(B) = \bb{Z}\alpha$, where $\alpha$ is the character $\begin{psmallmatrix}a & b \\ 0 & 1/a\end{psmallmatrix} \mapsto a$. The dominant weights are the non-negative integer multiples of $\alpha$, and the simple $G$-module of highest weight $n\alpha$ can be realised as the space $S^n\bbk^2 = \bbk[x,y]_n$ of homogeneous degree $n$ polynomials in 2 variables, where $G$ acts by linear change of variables.
\end{proposition}

\begin{proof}
	The $G$-module $\bbk[x,y]_n$ is indeed simple: if not, it decomposes as a direct sum of $G$-submodules by complete reduciblity of modules for reductive groups. Each of these $G$-submodules must contain a nonzero $U$-invariant (see \cite[Thm 17.5]{hum}). But the only $U$-invariants in $\bbk[x,y]_n$ are scalar multiples of $y^n$, and complementary submodules cannot both nontrivially intersect a single line. A simple check shows that $B$ acts on $\bbk y^n$ with weight $n\alpha$. 
\end{proof}

\begin{proposition}\thlabel{rnc}
	By the above, $G$ acts on $\proj{n} = \bb{P}(S^n\bbk^2)$. Under this action, the rational normal curve $Z \sub \proj{n}$ defined as the image of $\proj{1} = \bb{P}(\bbk^2)$ under the degree $n$ Veronese map, is a $G$-orbit.	
\end{proposition}

\begin{proof}
	Since $G$ acts transitively on $\proj{1}$ under the standard linear action, it suffices to show that $\nu_n$ is $G$-equivariant. But $\nu_n$ maps $[x:y]$ to $[x^n:x^{n-1}y:\ldots:xy^{n-1}:y^n]$ and since the $G$-action on $\proj{n}$ is defined by the same linear changes of the variables $x, y$ as the action on $\proj{1}$, equivariance follows immediately.	
\end{proof}

\begin{proposition}\thlabel{semi}
	Let $\proj{n} = \bb{P}(S^n\bbk^2)$ $(n \geq 2)$ have homogeneous co-ordinates $z_k$, $0 \leq k \leq n$. Then (where they are regular) $z_n$ is a $B$-semi-invariant of weight $n\alpha$ and $z_{n-2}z_n - z_{n-1}^2$ is a semi-invariant of weight $(2n-4)\alpha$.
\end{proposition}

\begin{proof}
	Since $z_n$ corresponds to $y^n \in S^n\bbk^2$, it is a semi-invariant of weight $n\alpha$ by \thref{simpleSL2}. For $z_{n-2}z_n - z_{n-1}^2$, note that $\begin{psmallmatrix}a & b \\ 0 & 1/a\end{psmallmatrix} \in B$ maps the co-ordinate function $z_k$ to the linear polynomial given by making the replacement $x^{n-m}y^m \mapsto z_m$ in the expression $(ax+by)^{n-m}y^m/a^m$. It is then straightforward to check the claim.
\end{proof}

\begin{proposition}\thlabel{Tweights}
	Each co-ordinate function $z_k$ on $\proj{n} = S^n\bbk^2$ is a $T$-eigenvector of weight $(2k-n)\alpha$, where $T \sub B$ is a maximal torus. It follows that a homogeneous $B$-eigenfunction of degree $d$ must be a linear combination of monomials $z_{k_1}\cdots z_{k_d}$ with $\sum_{i=1}^{d}{k_i} = \frac{1}{2}(m + dn)$. In particular, a $G$-invariant divisor of degree $d$ in $\proj{n}$ must be defined by a linear combination of such monomials with $\sum_{i=1}^{d}{k_i} = \frac{dn}{2}$.
\end{proposition}

\begin{proof}
	An element $\begin{psmallmatrix}a & 0 \\ 0 & 1/a\end{psmallmatrix} \in T$ acts on $z_k$ by $z_k \mapsto (ax)^{n-k}(y/a)^k \mapsto a^{n-2k}z_k$, so $z_k$ has weight $(2k-n)\alpha$. The $B$-weight of a $B$-eigenfunction must equal the $T$-weight of the same function, which must in turn equal the weight of any of its individual terms. Hence for a homogeneous polynomial of degree $d$, constructed from monomials $z_{k_1}\cdots z_{k_d}$, to be a $B$-eigenfunction of weight $m\alpha$ we must have $\sum_{i=1}^{d}{(2k_i-n)} = m$, or $\sum_{i=1}^{d}{k_i} = \frac{1}{2}(m+dn)$. Finally, a $G$-invariant divisor must be defined by a $G$-semi-invariant homogeneous polynomial. Since $G$ (being a perfect group) has no nontrivial characters, such a polynomial must in particular have $B$-weight 0, from which the final claim follows.
\end{proof}

\subsection{Blow-up of $\proj{3}$ Along Three Lines (4.6)}\label{46}

\subsubsection*{Hyperfan of $\proj{3}$}

Let $G = \sl{2}$ act on $X = \bb{P}(M_2(\bbk)) \cong \bb{P}^3$ by left multiplication of matrices. Then the orbit $B \cdot \left( \begin{smallmatrix} 1 & 0 \\ 0 & 1\end{smallmatrix}\right)$ is $\big\{\left( \begin{smallmatrix} x & y \\ 0 & 1\end{smallmatrix}\right)\in \proj{3}\big\}$. This point has stabiliser $\bb{Z}_2$ and the orbit therefore has dimension $\dim{B} = 2$. This must be a maximal orbit and so this action has complexity one. The $G$-orbit of the same point is $\big\{\left( \begin{smallmatrix} x & y \\ z & w \end{smallmatrix}\right) \in \proj{3} \mid xw-yz \neq 0\big\} = \pgl{2}$. This is an open subset of $\bb{P}^3$, so we are in the quasihomogeneous case, and $\bb{P}^3$ is an embedding of $\pgl{2} = \sl{2}/\bb{Z}_2 = G/H$. The degenerate matrices constitute a $G$-stable prime divisor $D = \cal{Z}(xw-yz)$ of $\proj{3}$, which contains all of its closed orbits. 

There is a family of colours parameterised by points in $\proj{1}$: namely, for $p = [\alpha \colon \beta] \in \proj{1}$, the divisor $D_p = \cal{Z}(\beta z - \alpha w)$ is a colour.

\subsubsection*{Coloured Hyperspace}

We know from \ref{SL2Zm} that the weight lattice $\Lambda$ of $\sl{2}/\bb{Z}_2$ is $\bb{Z}\cdot 2\alpha$, so we identify its dual $\cal{Q}$ with $\frac{1}{2}\bb{Z}$, and the field of $B$-invariants can be generated by $z/w$. In the case of $\sl{2}/\bb{Z}_2$, to keep the convention of our analysis of $\sl{2}$, we would choose a distinguished semi-invariant $e_{2\alpha} = z^2 \in \bbk[G/H]^{(B)}_{2\alpha}$. This corresponds to the rational function $F = z^2/(xw-yz)$ on $\proj{3}$. It is still a semi-invariant of weight $2\alpha$ and we will set $e_{2\alpha} = F$. 

In the hyperspace, the choice of $e_{2\alpha}$ means that $\infty = [0:1]$ is distinguished, since its pullback under the $B$-quotient map is the divisor of $F$. For all other points, the valuation cone is defined by $\ell + h \leq 0$ (not $2\ell + h$ as before since we identify $\cal{Q}$ with $\frac{1}{2}\bb{Z}$), and the colour $D_p$ lies at $(0,1)$ in $\cal{H}_p$. For $p = \infty$, the valuation cone is $\ell - h \leq 0$, and the colour $D_p$ lies at $(2,1)$. Indeed, we have $\ell_{D_\infty} = \nu_{D_\infty}(e_{2\alpha}) = \nu_{z}(z^2/(xw-yz)) = 2$.

Let $D := \cal{Z}(xw-yz) \sub \proj{3}$ be the divisor of degenerate matrices, with associated valuation $\nu_D$. This valuation is $G$-invariant and geometric, so lies in $\cal{V}$. To locate $\nu_D$ in the hyperspace $\cal{H}$, first note that $\nu_D(z/w) = 0$, so $\nu_D$ has trivial restriction to $K^B$ and is thus central. We also have $\nu_D(e_{2\alpha}) = \nu_D(z^2/(xw-yz)) = -1$, so $\nu_D$ sits at $(-1,0)$ in the centre of the hyperspace. Since any positive rational multiple of $\nu_D$ is also a $G$-valuation, the cone of central valuations $\cal{K} = \cal{V} \inter \cal{Z}$ is $\bb{Q}_{\leq 0}$.

\subsubsection*{Coloured Data}

As has been noted, all closed $G$-orbits in $X$ lie in the divisor of degenerate matrices $D$. Since we are projectivising $2 \times 2$ matrices, $D$ must consist exclusively of (projectivisations of) rank 1 matrices. It is not difficult to check that each closed orbit consists of matrices with a given kernel, so they are parameterised by $\proj{1}$. For $p \in \proj{1}$ we write $Y_p$ for the closed orbit of matrices whose kernel is the line in $\bbk^2$ represented by $p$. Each colour $D_p$ contains the closed orbit $Y_p$ and this orbit is contained in no other $B$-divisor, so the coloured data of the $G$-germs in $X$ are as follows: $\cal{V}_{Y_p} = \{\nu_D\}$, $\cal{D}^B_{Y_p} = \{D_p\}$; $\cal{V}_D = \{\nu_D\}, \cal{D}^B_{D} = \emptyset$.

Thus for $p \neq \infty$, the minimal $G$-germ $Y_p$ corresponds to the coloured cone in $\cal{H}_p$ spanned by the colour $D_p$ at $(0,1)$ and the $G$-divisor $D$ at $(-1,0)$, i.e. it is the upper-left quarter-plane. Similarly for $p = \infty$ the coloured cone is spanned by $D_\infty = (2,1)$ and $D$. We can see that for any $p$, the coloured cones spanned by the minimal $G$-germs cover $\cal{V}_p$ entirely, in accordance with completeness of $\proj{3}$. Hence the coloured hyperfan of $\proj{3}$ looks as follows:

\begin{center}
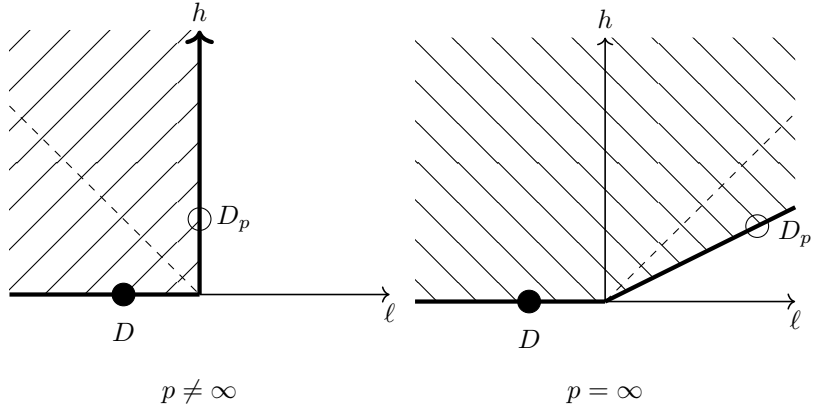
\begin{tikzpicture}
	\draw[ultra thick](-2.5,0) -- (0,0);
	\draw[semithick][->](0,0) -- (2.5,0);
	\draw[ultra thick][->] (0,0) -- (0,3.5);
	\draw[dashed][domain = -2.5:0] plot (\x, {-\x}); 
	\node[below] at (2.5,0) {$\ell$};
	\node[above] at (0,3.5) {$h$};
	\draw[fill=black] (-1,0) circle [color=black,radius=0.15];
	\node[below] at (-1,-0.25) {$D$};
	\node[below] at (0,-1) {$p \neq \infty$};
	\draw (0,1) circle [radius=0.15];
	\node[right] at (0.1,1) {$D_p$};
	\fill[pattern=my north east lines](-2.5,0) to (0,0) to (0,3.5) to (-2.5,3.5);
	\end{tikzpicture}
	\begin{tikzpicture}
	\draw[ultra thick](-2.5,0) -- (0,0);
	\draw[semithick][->](0,0) -- (2.5,0);
	\draw[semithick][->] (0,0) -- (0,3.5);
	\draw[dashed][domain = 0:2.5] plot (\x, {\x}); 
	\draw[ultra thick](0,0) -- (2.5,1.25);
	\node[below] at (2.5,0) {$\ell$};
	\node[above] at (0,3.5) {$h$};
	\draw[fill=black] (-1,0) circle [color=black,radius=0.15];
	\node[below] at (-1,-0.25) {$D$};
	\node[below] at (0,-1) {$p = \infty$};
	\draw (2,1) circle [radius=0.15];
	\node[right] at (2.15,0.9) {$D_p$};
	\fill[pattern=my north west lines](-2.5,0) to (-2.5,3.5) to (2.5,3.5) to (2.5,1.25) to (0,0);
	\end{tikzpicture}\end{center}
\begingroup
\captionof{figure}{Coloured hyperfan of $\proj{3}$ (linear action)}
\endgroup

\noindent where filled circles represent $G$-divisors, unfilled circles represent colours, thick lines indicate rays spanned by $G$-germs and $B$-divisors, hatched areas show the coloured cones generated by minimal $G$-germs, and dashed lines show the boundaries of the valuation cones.

\subsubsection*{Blow-up of One Line}

The closed $G$-orbits $Y_p \sub \proj{3}$, where $p = [\alpha \colon \beta] \in \proj{1}$, consist of matrices whose kernel is the line in $\bbk^2$ represented by $p$. That is, $Y_p = \cal{Z}(\beta x - \alpha y, \beta z - \alpha w)$. Each of these closed orbits is a $G$-stable line in $\proj{3}$, and they are mutually disjoint. We will obtain our example by blowing up three of these lines, which can be chosen arbitrarily. First we investigate what happens to the coloured data and hyperspace after one blow-up, and the rest follows easily.

Let $0 = [1:0] \in \proj{1}$ and consider $Y_0 = \cal{Z}(y,w) \sub \proj{3}$. Let $X := \bl_{Y_0}(\proj{3}) = \cal{Z}(yv-wu) \sub \proj{1} \times \proj{3}$. Note that under the blow-up, the colours $D_p$ and closed orbits $Y_p$ of $\proj{3}$ where $p \neq 0$ pull back isomorphically to $X$, and since $Y_0$ is $G$-stable, the blow-up is equivariant. 

The exceptional divisor of this blow-up is $E_0=\cal{Z}(y,w) \sub \proj{1} \times \proj{3}$ and the strict transform of the divisor of degenerate matrices is $\tilde{D}=\cal{Z}(xw-yz,uz-vx,yv-uw) \sub \proj{1} \times \proj{3}$. These are the only $G$-stable prime divisors in $X$, and their intersection is the curve $\cal{Z}(uz-vx,y,w)$. Together $\tilde{D}, E_0$, their intersection and the closed orbits $Y_p$ $(p\neq 0)$ constitute all $G$-germs of $X$. Meanwhile, the colours of $X$ are the colours $D_p$ $(p\neq 0)$ of $\proj{3}$ and the strict transform $\tilde{D}_0 = \cal{Z}(w,v)$. 

Hence the coloured data of the $G$-germs of $X$ are as follows: for $p \neq 0$ we have $\cal{V}_{Y_p} = \{\nu_{\tilde{D}}\}$, $\cal{D}^B_{Y_p} = \{\tilde{D}_p\}$. Then also $\cal{V}_{\tilde{D} \inter E_0} = \{\nu_{\tilde{D}},\nu_{E_0}\}$, $\cal{D}^B_{\tilde{D} \inter E_0} = \emptyset$, $\cal{V}_{\tilde{D}} = \{\nu_{\tilde{D}}\}$, $\cal{D}^B_{\tilde{D}} = \emptyset$, $\cal{V}_{E_0} = \{\nu_{E_0}\}$, $\cal{D}^B_{E_0} = \emptyset$.

The set-up of the hyperspace is unchanged from the example of $\proj{3}$: $\Lambda$ is generated by $2\alpha$, $\cal{Q}$ is identified with $\frac{1}{2}\bb{Z}$, we choose the splitting $e_{2\alpha} = z^2/(xw-yz)$ (marking $\infty$) as the distinguished point), and the valuation cones are defined by $\ell + h \leq 0$ for $p \neq \infty$ and $\ell - h \leq 0$ for $p = \infty$. It remains to locate the colours and $G$-divisors. 

The central divisor $\tilde{D}$ still sits at $(-1,0)$ in every section of the hyperspace, as before.  For $p = 0$, the colour $\tilde{D}_{0}$ sits at $(0,1)$ and the $G$-divisor $E_0$ sits at $(-1,1)$. The coloured cone defined by the minimal $G$-germ $\tilde{D} \inter E_0$ spans the rays defined by $\tilde{D}$ and $E_0$. For $p \neq 0$, the colour $\tilde{D}_p$ sits at $(0,1)$ as before and the cone defined by $\tilde{Y}_p$ spans the rays defined by $D$ and $\tilde{D}_p$. Again we see that the coloured cones defined by the various $G$-germs all cover the valuation cone in each slice of the hyperspace, as required by completeness of $X$. Thus the coloured hyperfan looks like:

\begin{center}\resizebox{\linewidth}{!}{\begin{tikzpicture}
		\draw[ultra thick](-2.5,0) -- (0,0);
		\draw[semithick][->](0,0) -- (2.5,0);
		\draw[ultra thick][->] (0,0) -- (0,3.5);
		\draw[dashed][domain = -2.5:0] plot (\x, {-\x}); 
		\node[below] at (2.5,0) {$\ell$};
		\node[above] at (0,3.5) {$h$};
		\draw[fill=black] (-1,0) circle [color=black,radius=0.15];
		\node[below] at (-1,-0.25) {$\tilde{D}$};
		\node[below] at (0,-1) {$p \neq 0, \infty$};
		\draw (0,1) circle [radius=0.15];
		\node[right] at (0.1,1) {$\tilde{D}_p$};
		\fill[pattern=my north east lines](-2.5,0) to (0,0) to (0,3.5) to (-2.5,3.5);
		\end{tikzpicture}
		\begin{tikzpicture}
		\draw[ultra thick](-2.5,0) -- (0,0);
		\draw[semithick][->](0,0) -- (2.5,0);
		\draw[semithick][->] (0,0) -- (0,3.5);
		\draw[ultra thick][domain = -2.5:0] plot (\x, {-\x}); 
		\node[below] at (2.5,0) {$\ell$};
		\node[above] at (0,3.5) {$h$};
		\draw[fill=black] (-1,0) circle [color=black,radius=0.15];
		\node[below] at (-1,-0.25) {$\tilde{D}$};
		\node[below] at (0,-1) {$p = 0$};
		\draw (0,1) circle [radius=0.15];
		\node[right] at (0.1,1) {$\tilde{D}_0$};
		\draw[fill=black] (-1,1) circle  [color=black,radius=0.15];
		\node[above] at (-1,1.1) {$E_0$};
		\fill[pattern=my north east lines](-2.5,0) to (-2.5,2.5) to (0,0);
		\end{tikzpicture}
		\begin{tikzpicture}
		\draw[ultra thick](-2.5,0) -- (0,0);
		\draw[semithick][->](0,0) -- (2.5,0);
		\draw[semithick][->] (0,0) -- (0,3.5);
		\draw[dashed][domain = 0:2.5] plot (\x, {\x}); 
		\draw[ultra thick](0,0) -- (2.5,1.25);
		\node[below] at (2.5,0) {$\ell$};
		\node[above] at (0,3.5) {$h$};
		\draw[fill=black] (-1,0) circle [color=black,radius=0.15];
		\node[below] at (-1,-0.25) {$\tilde{D}$};
		\node[below] at (0,-1) {$p = \infty$};
		\draw (2,1) circle [radius=0.15];
		\node[right] at (2.15,0.9) {$\tilde{D}_p$};
		\fill[pattern=my north west lines](-2.5,0) to (-2.5,3.5) to (2.5,3.5) to (2.5,1.25) to (0,0);
		\end{tikzpicture}}\end{center}
\begingroup

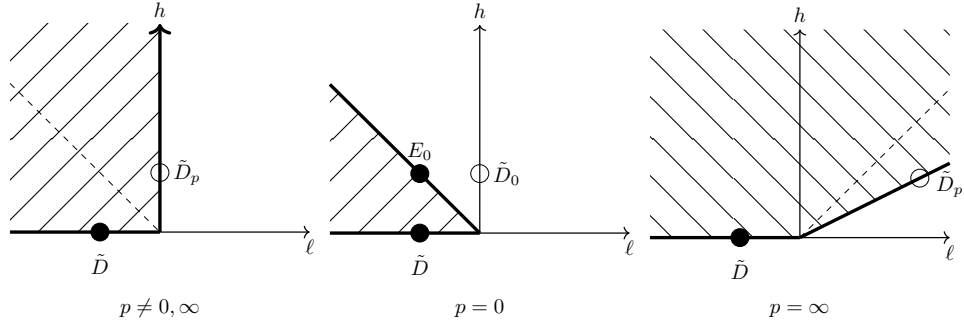
\captionof{figure}{Coloured hyperfan of the blow-up of $\proj{3}$ along one line}
\endgroup

\subsubsection*{Blow-up of Three Lines}

Now we can go back to $\proj{3}$, choose three arbitrary non-distinguished points (say $q, r, s \in \proj{1}\setminus{\{\infty}\})$ and successively blow up their corresponding $G$-orbits, in this case $Y_q$, $Y_r$ and $Y_s$. From the calculations above for the first blow up it is clear what happens as far as $G$-germs and the hyperspace are concerned: the slices of hyperspace corresponding to all points $p \neq q, r, s$ will be unchanged from their description above, and in each of the slices corresponding to $q, r, s$ there will be a new $G$-divisor (the exceptional divisor of the blow-up) sitting at $(-1,1)$, while the corresponding colour does not move from $(0,1)$. The minimal $G$-germs are $Y_p$ for $p \neq q, r, s$ and $\tilde{D} \inter E_p$ for $p=q, r, s$. The former define coloured cones bounded by $\tilde{D}$ and $\tilde{D}_p$ in $\cal{H}_p$ ($p \neq q, r, s)$ and the latter define coloured cones bounded by $\tilde{D}$ and $E_p$ in $\cal{H}_p$ ($p = q, r, s)$. Hence we get the following coloured hyperfan:

\begin{center}\resizebox{\linewidth}{!}{\begin{tikzpicture}
		\draw[ultra thick](-2.5,0) -- (0,0);
		\draw[semithick][->](0,0) -- (2.5,0);
		\draw[ultra thick][->] (0,0) -- (0,3.5);
		\draw[dashed][domain = -2.5:0] plot (\x, {-\x}); 
		\node[below] at (2.5,0) {$\ell$};
		\node[above] at (0,3.5) {$h$};
		\draw[fill=black] (-1,0) circle [color=black,radius=0.15];
		\node[below] at (-1,-0.25) {$\tilde{D}$};
		\node[below] at (0,-1) {$p \neq q,r,s, \infty$};
		\draw (0,1) circle [radius=0.15];
		\node[right] at (0.1,1) {$\tilde{D}_p$};
		\fill[pattern=my north east lines](-2.5,0) to (0,0) to (0,3.5) to (-2.5,3.5);
		\end{tikzpicture}
		\begin{tikzpicture}
		\draw[ultra thick](-2.5,0) -- (0,0);
		\draw[semithick][->](0,0) -- (2.5,0);
		\draw[semithick][->] (0,0) -- (0,3.5);
		\draw[ultra thick][domain = -2.5:0] plot (\x, {-\x}); 
		\node[below] at (2.5,0) {$\ell$};
		\node[above] at (0,3.5) {$h$};
		\draw[fill=black] (-1,0) circle [color=black,radius=0.15];
		\node[below] at (-1,-0.25) {$\tilde{D}$};
		\node[below] at (0,-1) {$p = q,r,s$};
		\draw (0,1) circle [radius=0.15];
		\node[right] at (0.1,1) {$\tilde{D}_p$};
		\draw[fill=black] (-1,1) circle  [color=black,radius=0.15];
		\node[above] at (-1,1.1) {$E_p$};
		\fill[pattern=my north east lines](-2.5,0) to (-2.5,2.5) to (0,0);
		\end{tikzpicture}
		\begin{tikzpicture}
		\draw[ultra thick](-2.5,0) -- (0,0);
		\draw[semithick][->](0,0) -- (2.5,0);
		\draw[semithick][->] (0,0) -- (0,3.5);
		\draw[dashed][domain = 0:2.5] plot (\x, {\x}); 
		\draw[ultra thick](0,0) -- (2.5,1.25);
		\node[below] at (2.5,0) {$\ell$};
		\node[above] at (0,3.5) {$h$};
		\draw[fill=black] (-1,0) circle [color=black,radius=0.15];
		\node[below] at (-1,-0.25) {$\tilde{D}$};
		\node[below] at (0,-1) {$p = \infty$};
		\draw (2,1) circle [radius=0.15];
		\node[right] at (2.1,0.9) {$\tilde{D}_p$};
		\fill[pattern=my north west lines](-2.5,0) to (-2.5,3.5) to (2.5,3.5) to (2.5,1.25) to (0,0);
		\end{tikzpicture}}\end{center}
\begingroup

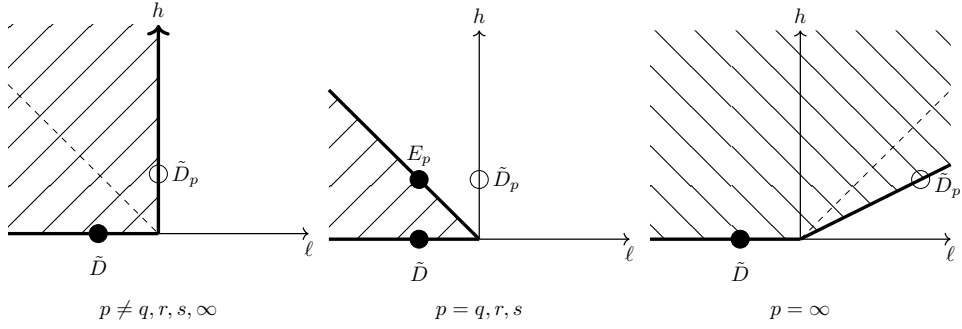
\captionof{figure}{Coloured hyperfan of the blow-up of $\proj{3}$ along three lines}
\endgroup

\subsection{Blow-up of $\proj{1} \times \proj{2}$ (3.17)}\label{317}

\subsubsection*{Orbits and $G$-germs Before the Blow-up}

Let $G = \sl{2}$ act on $\proj{1} \times \proj{2}$, linearly on the first factor and quadratically on the second. The $G$-orbit of the point $P = ([1:0],[1:0:1])$ is \[\{([a:c],[a^2+b^2:ac+bd:c^2+d^2]) \mid ad-bc=1\}.\] The stabiliser $G_P$ is $\bb{Z}_4$, so this orbit is open, and the $B$-orbit is easily checked to be 2 dimensional. Hence $\proj{1} \times \proj{2}$ is a complexity one $G$-variety and an embedding of $G/H = \sl{2}/\bb{Z}_4$. 

The divisors $\Delta = \cal{Z}(x_0^2z_2 + x_1^2z_0 - 2x_0x_1z_1)$ and $F= \cal{Z}(z_0z_2-z_1^2)$ are $G$-stable. Consider the $G$-stable curve $C = \cal{Z}(x_1z_0-x_0z_1,x_1z_1-x_0z_2) \sub \proj{1} \times \proj{2}$. Note that $C = F \inter \Delta$. The orbits on $\proj{1} \times \proj{2}$ are as follows: $C$ is itself a closed orbit, then $\Delta\setminus{C}$ and $F\setminus{C}$ are orbits, and the open orbit described above is $\proj{1}\times \proj{2}\setminus{(F \union \Delta)}$. This is shown in detail by calculations after the blow-up in a later subsection. Hence the proper $G$-germs of $\proj{1} \times \proj{2}$ are $C$, $F$ and $\Delta$.

\subsubsection*{(Semi-) Invariant Functions}

Recall that for $\sl{2}/\bb{Z}_4$, the weight lattice $\Lambda$ is generated by $2\alpha$, $\cal{Q} = \Lambda^*$ is identified with $\frac{1}{2}\bb{Z}$, the field of invariants is generated by $z^2/w^2$ and we chose a semi-invariant regular function $F$ from the module $M = \bbk[G]^{(B \times H)}_{(2\alpha,\epsilon^2)}$ spanned by $z^2$ and $w^2$ to give a splitting $e_{2\alpha} = F^2/(zw)$. 

On $\proj{1} \times \proj{2}$, the function $f_0 = g_0/h_0 = x_1^2(z_0z_2-z_1^2)/(x_0z_2-x_1z_1)^2$ is $B$-invariant and under the isomorphism $G/H \cong \proj{1} \times \proj{2}\setminus{(F \union \Delta)}$ corresponds to the invariant $z^2/w^2$. Hence this function defines a rational $B$-quotient map $\pi \colon \proj{1} \times \proj{2} \dashrightarrow \proj{1}$, $P \mapsto [g_0(P):h_0(P)]$. 

For $p = [\alpha : \beta] \in \proj{1}$, the pullbacks $\pi^*(p) = \cal{Z}(\beta g_0 - \alpha h_0)$ define a family of regular colours $D_p = \pi^*(p)$ except at three points: \[p = \infty = [0:1]: \pi^*(p) = \cal{Z}(x_1^2(z_0z_2-z_1^2)) = F \union \cal{Z}(x_1^2) = F \union D_\infty,\]\[p = -1 = [1:-1]: \pi^*(p) = \cal{Z}(z_2(x_0^2z_2+x_1^2z_0-2x_0x_1z_1)) = \Delta \union \cal{Z}(z_2) = \Delta \union D_{-1},\]\[p = 0 = [1:0]: \pi^*(p) = \cal{Z}((x_0z_2-x_1z_1)^2) = D_0.\] Note that $D_\infty$ and $D_0$ correspond to points in $\proj{1}$ of multiplicity 2, i.e. they are subregular colours and thus have $h$-coordinate 2 in hyperspace, in accordance with the calculation of the hyperspace of $\sl{2}/\bb{Z}_4$ in \Cref{SL2Zm}.

Now we choose as the splitting semi-invariant \[e_{2\alpha} = \frac{z_2^2(x_0^2z_2+x_1^2z_0-2x_0x_1z_1)}{x_1(z_0z_2-z_1^2)(x_0z_2-x_1z_1)}.\] This corresponds in the homogeneous space to the function $(z^2+w^2)^2/(zw)$, and hence to the choice $F = z^2 + w^2 \in M$ and thus marks out the point $-1 \in \proj{1}$ as distinguished.

\subsubsection*{Coloured Data and Hyperfan}

We first note here that the curve $C$ is contained in the $G$-divisors $\Delta$ and $F$, and in every colour $D_p$ for $p \neq -1,\infty$, i.e. in colours lying over points in $\proj{1}$ whose pullback does not contain $\Delta$ or $F$. Hence the coloured data of $C$ is: $\cal{V}_C = \{\nu_\Delta,\nu_F\}$, $\cal{D}^B_C = \{D_p \mid p\neq -1,\infty\}$, and the remaining coloured data is $\cal{V}_\Delta = \{\nu_\Delta\}, \cal{V}_F = \{\nu_F\}$, $\cal{D}^B_\Delta = \cal{D}^B_F = \emptyset$. We see from this that $C$ defines a coloured hypercone of type II in $\cal{H}$.

In accordance with \Cref{SL2Zm}, the subregular colours $D_0, D_\infty$ sit at $(-1,2)$ in their respective slices of hyperspace, the colour $D_{-1}$ sitting over the distinguished point sits at $(2,1)$, and the non-distinguished regular colours $D_p$ for $p \neq 0,-1,\infty$ lie at $(0,1)$. Finally, the $G$-divisors $F$ and $\Delta$ go to $(p,\ell,h) = (\infty,-1,1)$ and $(p,\ell,h) = (-1,1,1)$ respectively. Then the polytope defined by $C$ is $\cal{P} = \cal{P}_0 + \cal{P}_\infty + \cal{P}_{-1} = \{-1/2\}+\{-1\}+\{1\} = \{-1/2\}$. Hence the coloured hyperfan of $\proj{1}\times \proj{2}$ is as follows:

\begin{center}\begin{tikzpicture}
	
	\draw[ultra thick] (-2.5,0) to (0,0);
	\draw[semithick][->] (0,0) -- (2.5,0);
	\draw[semithick][->] (0,0) -- (0,3.5);
	\draw[dashed][domain = -2.5:0] plot (\x, {-\x});
	\draw[ultra thick] (0,0) -- (-1.75,3.5); 
	\draw[fill=black] (0,2) circle [color=black,radius=0.05];
	\node[right] at (0,2) {$2$};
	\draw[fill=black] (-1,0) circle [color=black,radius=0.05];
	\node[below] at (-1,0) {$-1$};
	\draw[fill=black] (-1,0) circle [color=black,radius=0.05];
	\node[below] at (2.5,0) {$\ell$};
	\node[above] at (0,3.5) {$h$};
	\node[below] at (0,-0.5) {$p = 0$};
	\draw (-1,2) circle [radius=0.15];
	\node[above] at (-1,2.1) {$D_0$};
	\begin{scope}[on background layer]
	\fill[pattern=my north east lines](-2.5,0) to (-2.5,3.5) to (-1.75,3.5) to (0,0);
	\end{scope}
	\end{tikzpicture}
	\begin{tikzpicture}
	\draw[semithick][->] (0,0) -- (2.5,0);
	\draw[ultra thick] (-2.5,0) -- (0,0);
	\draw[semithick][->] (0,0) -- (0,3.5);
	\draw[ultra thick][domain = -2.5:0] plot (\x, {-\x}); 
	\draw[fill=black] (0,2) circle [color=black,radius=0.05];
	\node[right] at (0,2) {$2$};
	\draw[fill=black] (-1,0) circle [color=black,radius=0.05];
	\node[below] at (-1,0) {$-1$};
	\draw[fill=black] (-1,0) circle [color=black,radius=0.05];
	\node[below] at (2.5,0) {$\ell$};
	\node[above] at (0,3.5) {$h$};
	\node[below] at (0,-0.5) {$p = \infty$};
	\draw (-1,2) circle [radius=0.15];
	\node[above] at (-1,2.1) {$D_\infty$};
	\draw[fill=black] (-1,1) circle [color=black,radius=0.15];
	\node[right] at (-0.9,1) {$F$};		
	\begin{scope}[on background layer]
	\fill[pattern=my north east lines](-2.5,0) to (-2.5,2.5) to (0,0);
	\end{scope}
	\end{tikzpicture}\end{center}

\begin{center}
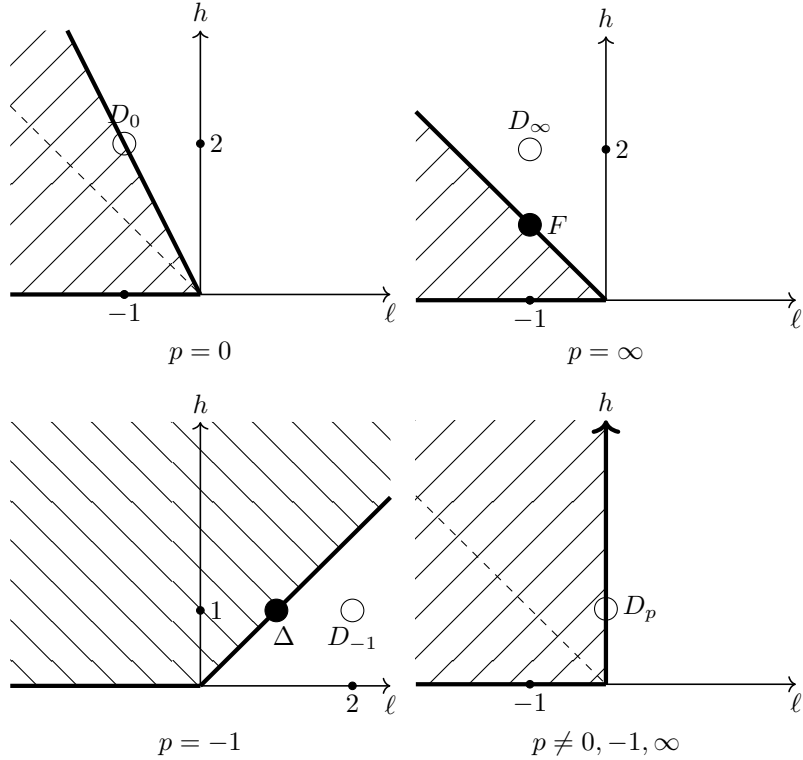
\begin{tikzpicture}
	\draw[ultra thick] (-2.5,0) to (0,0);
	\draw[semithick][->] (0,0) -- (2.5,0);
	\draw[semithick][->] (0,0) -- (0,3.5);
	\draw[ultra thick][domain = 0:2.5] plot (\x, {\x}); 
	\node[below] at (2.5,0) {$\ell$};
	\node[above] at (0,3.5) {$h$};
	\draw[fill=black] (0,1) circle [color=black,radius=0.05];
	\node[right] at (0,1) {$1$};
	\draw[fill=black] (2,0) circle [color=black,radius=0.05];
	\node[below] at (2,0) {$2$};
	\node[below] at (0,-0.5) {$p = -1$};
	\draw[fill=black] (1,1) circle [color=black,radius=0.15];
	\node[below] at (1.1,0.9) {$\Delta$};
	\draw (2,1) circle [radius=0.15];
	\node[below] at (2,0.9) {$D_{-1}$};
	\begin{scope}[on background layer]
	\fill[pattern=my north west lines](-2.5,0) to (-2.5,3.5) to (2.5,3.5) to (2.5,2.5) to (0,0);
	\end{scope}
	\end{tikzpicture}
	\begin{tikzpicture}
	\draw[ultra thick] (-2.5,0) to (0,0);
	\draw[semithick][->] (0,0) -- (2.5,0);
	\draw[ultra thick][->] (0,0) -- (0,3.5);
	\draw[dashed][domain = -2.5:0] plot (\x, {-\x}); 
	\node[below] at (2.5,0) {$\ell$};
	\node[above] at (0,3.5) {$h$};
	\draw[fill=black] (-1,0) circle [color=black,radius=0.05];
	\node[below] at (-1,0) {$-1$};
	\node[below] at (0,-0.5) {$p \neq 0,-1,\infty$};
	\draw (0,1) circle [radius=0.15];
	\node[right] at (0.1,1) {$D_p$};
	\begin{scope}[on background layer]
	\fill[pattern=my north east lines](-2.5,0) to (-2.5,3.5) to (0,3.5) to (0,0);
	\end{scope}
	\end{tikzpicture}\end{center}
\begingroup
\captionof{figure}{Coloured hyperfan of $\proj{1} \times \proj{2}$}
\endgroup

\subsubsection*{Blow-up}

We now blow up $C$ to obtain the variety $X = \cal{Z}(x_0y_0z_2+x_1y_1z_0-x_0y_1z_1-x_1y_0z_1) \sub \proj{1} \times \proj{1} \times \proj{2}$. Then $X$ contains the $G$-stable divisors $\tilde{\Delta} = \cal{Z}(x_0y_1-x_1y_0)$, the strict transform of the divisor $\Delta$ defined above, $E = \cal{Z}(x_1z_0-x_0z_1,x_1z_1-x_0z_2)$, the exceptional divisor of the blow-up, and $\tilde{F} = \cal{Z}(y_1z_0-y_0z_1,y_0z_2-y_1z_1)$, the strict transform of the above divisor $\tilde{F}$. Let $D = E \union \tilde{F}$.

\subsubsection*{Orbits and $G$-germs After the Blow-up}

Claim: the $G$-orbits on $X$ are $X\setminus{(D\union \tilde{\Delta})}$, which is open, $\tilde{\Delta}\setminus{(D \inter \tilde{\Delta})}$, $E\setminus{(E \inter \tilde{\Delta})}$, $\tilde{F}\setminus{(\tilde{F} \inter \tilde{\Delta})}$ and $D \inter \tilde{\Delta}$, which is closed. We note that $D \inter \tilde{\Delta} = E\inter \tilde{\Delta} = \tilde{F}\inter \tilde{\Delta} = E\inter \tilde{F}$. Thus the $G$-germs of $X$ are $\tilde{\Delta}$, $E$, $\tilde{F}$ and $D \inter \tilde{\Delta}$, with the latter being minimal.

\begin{proposition}
	The open $\sl{2}$-orbit on $X$ is $X\setminus{(D \union \tilde{\Delta})}$.
\end{proposition}

\begin{proof}
	For $(p,q) \in \proj{1} \times \proj{1}$, let $X_{p,q} = X \inter (\{(p,q)\}\times \proj{2})$. Since the torus $\bbk^* \sub \sl{2}$ fixes $0:=[0:1], \infty:=[1:0] \in \proj{1}$, it must also leave $X_{0,\infty}$ stable. We claim that $\bbk^*$ acts transitively on $X_{0,\infty}\setminus{D_{0,\infty}}$:
	
	Indeed, suppose $Q = (0,\infty,[q_0:q_1:q_2]) \in X_{0,\infty}\setminus{D_{0,\infty}}$. The equation for $X$ demands that $q_1 = 0$, and this means that the equations for $E, \tilde{F}$ reduce to $q_0 = 0$ and $q_2 = 0$, respectively. Hence we must have $Q = (0,\infty,[q_0:0:q_2])$, with $q_0q_2 \neq 0$. Now consider $P = (0,\infty,[1:0:1])$, whose image under $A = \begin{psmallmatrix}a & 0 \\ 0 & 1/a\end{psmallmatrix} \in \bbk^*$ is $(0,\infty,[a^2:0:1/a^2]) = (0,\infty,[a^4:0:1])$. By setting $a$ to be any fourth root of $q_0/q_2$, we thus have that $Q = A\cdot P$, proving the claim.
	
	Now let $S = (p,q,r) \in X\setminus{(D \union \tilde{\Delta})}$. Since $p, q \in \proj{1}$ are distinct, there exists $M \in \sl{2}$ with $M\cdot S = (0,\infty,M\cdot r) \in X_{0,\infty}\setminus{D_{0,\infty}}$. Now by the above there exists $A \in \bbk^*$ with $A\cdot(M\cdot S) = P$, hence $\sl{2} \cdot P = X\setminus{(D \union \tilde{\Delta})}$ as promised.
\end{proof}

\begin{proposition}
	$\tilde{\Delta}\setminus{(D \inter \tilde{\Delta})}$ is an $\sl{2}$-orbit on $X$.
\end{proposition}

\begin{proof}
	First, note that the Borel subgroup $B \sub \sl{2}$ fixes $\infty \in \proj{1}$, which we use to show that $B$ acts transitively on $X_{\infty,\infty}\setminus{D_{\infty,\infty}}$. The equations for $X$ and $D$ here reduce to $z_2 = 0$, $z_1 = 0$ respectively. Let $P = (\infty,\infty,[0:1:0]) \in X_{\infty,\infty}\setminus{D_{\infty,\infty}}$, so that $A= \begin{psmallmatrix} a & b \\ 0 & 1/a\end{psmallmatrix}\cdot P = (\infty,\infty,[2ab:1:0])$. If $r=[r_0:r_1:0] \in X_{\infty,\infty}\setminus{D_{\infty,\infty}}$, then setting $a = 1, b = r_0/2r_1$ gives $A\cdot P = (\infty,\infty,r)$, so we are done. Now for any $Q = (p,p,q) \in \tilde{\Delta}\setminus{(D \inter \tilde{\Delta})}$ there exists $M \in \sl{2}$ with $M\cdot Q = (\infty,\infty,M\cdot q)$, so there exists $A \in B$ such that $A\cdot M\cdot Q = P$ as required.
\end{proof}

\begin{proposition}
	$E\setminus{(E \inter \tilde{\Delta})}$ and $\tilde{F}\setminus{(\tilde{F} \inter \tilde{\Delta})}$ are $\sl{2}$-orbits on $X$.
\end{proposition}

\begin{proof}
	Let $Q = (p,q,r) \in E\setminus{(E \inter \tilde{\Delta})}$. As above, there exists $M \in \sl{2}$ with $M\cdot Q = (0,\infty,M\cdot r) \in D_{0,\infty}$. Now since $E$ is $\sl{2}$ stable, we have in fact that $M\cdot Q \in (E)_{0,\infty}$, which is a singleton. Hence $E\setminus(E \inter \tilde{\Delta})$ is an $\sl{2}$-orbit, and the case for $\tilde{F}$ is symmetric. 
\end{proof}

\begin{proposition}
	The final $\sl{2}$-orbit on $X$ is $\tilde{\Delta} \inter D$. Hence the $G$-orbits on $X$ are $D \union \tilde{\Delta}$, $\tilde{\Delta}\setminus{(D \inter \tilde{\Delta})}$, $E\setminus{(E \inter \tilde{\Delta})}$, $\tilde{F}\setminus{(\tilde{F} \inter \tilde{\Delta})}$ and $D \inter \tilde{\Delta}$
\end{proposition}

\begin{proof}
	Let $Q = (p,p,q) \in \tilde{\Delta} \inter D$, noting that assuming the equations for $\tilde{\Delta}$, the equations for $E$ and $\tilde{F}$ become the same, i.e. $\tilde{\Delta} \inter E = \tilde{\Delta} \inter \tilde{F}$. It follows that $D_{\infty,\infty}$ is a singleton, say $P$. We can as before choose $M \in \sl{2}$ such that $M\cdot Q \in D_{\infty,\infty}$, showing that $\tilde{\Delta} \inter D = \sl{2}\cdot P$. 
	
	Now it is clear that the orbits described so far cover $X$, so they must constitute an exhaustive list.
\end{proof}

\subsubsection*{Hyperfan of $X$}

As in previous examples, blowing up the curve $C$ does not change the position of any divisor in hyperspace, but it adds the new $G$-divisor $E$. Taking the same invariant and semi-invariant rational functions used above (i.e. $f_0$ and $e_{2\alpha}$) we see that $E$ sits over $0 \in \proj{1}$ and lies at $(p, \ell, h) = (0,-1,1) \in \cal{H}$. The colour $\tilde{D}_0$ no longer contains the minimal $G$-germ $E \inter \tilde{\Delta}$, which now has coloured data $\cal{V}_{E\inter \tilde{\Delta}} = \{\nu_E,\nu_{\tilde{F}},\nu_{\tilde{\Delta}}\}$, $\cal{D}^B_{E \inter \tilde{\Delta}} = \{\tilde{D}_p \mid p \neq -1, 0, \infty\}$.

We thus have $\cal{P} = \cal{P}_0 + \cal{P}_\infty = \cal{P}_{-1} = \{-1\} + \{-1\} + \{1\} = \{-1\}$. Hence all things considered the coloured hyperfan for $X$ looks like:

\begin{center}\begin{tikzpicture}
	
	\draw[ultra thick] (-2.5,0) to (0,0);
	\draw[semithick][->] (0,0) -- (2.5,0);
	\draw[semithick][->] (0,0) -- (0,3.5);
	\draw[ultra thick][domain = -2.5:0] plot (\x, {-\x});
	\draw[fill=black] (0,2) circle [color=black,radius=0.05];
	\node[right] at (0,2) {$2$};
	\draw[fill=black] (-1,0) circle [color=black,radius=0.05];
	\node[below] at (-1,0) {$-1$};
	\draw[fill=black] (-1,0) circle [color=black,radius=0.05];
	\node[below] at (2.5,0) {$\ell$};
	\node[above] at (0,3.5) {$h$};
	\node[below] at (0,-0.5) {$p = 0$};
	\draw (-1,2) circle [radius=0.15];
	\node[above] at (-1,2.1) {$\tilde{D}_0$};
	\node[right] at (-0.9,1) {$E$};
	\draw[fill=black] (-1,1) circle [color=black,radius=0.15];
	\begin{scope}[on background layer]
	\fill[pattern=my north east lines](-2.5,0) to (-2.5,2.5) to (0,0);
	\end{scope}
	\end{tikzpicture}
	\begin{tikzpicture}
	\draw[semithick][->] (0,0) -- (2.5,0);
	\draw[ultra thick] (-2.5,0) -- (0,0);
	\draw[semithick][->] (0,0) -- (0,3.5);
	\draw[ultra thick][domain = -2.5:0] plot (\x, {-\x}); 
	\draw[fill=black] (0,2) circle [color=black,radius=0.05];
	\node[right] at (0,2) {$2$};
	\draw[fill=black] (-1,0) circle [color=black,radius=0.05];
	\node[below] at (-1,0) {$-1$};
	\draw[fill=black] (-1,0) circle [color=black,radius=0.05];
	\node[below] at (2.5,0) {$\ell$};
	\node[above] at (0,3.5) {$h$};
	\node[below] at (0,-0.5) {$p = \infty$};
	\draw (-1,2) circle [radius=0.15];
	\node[above] at (-1,2.1) {$\tilde{D}_\infty$};
	\draw[fill=black] (-1,1) circle [color=black,radius=0.15];
	\node[right] at (-0.9,1) {$\tilde{F}$};		
	\begin{scope}[on background layer]
	\fill[pattern=my north east lines](-2.5,0) to (-2.5,2.5) to (0,0);
	\end{scope}
	\end{tikzpicture}\end{center}

\begin{center}
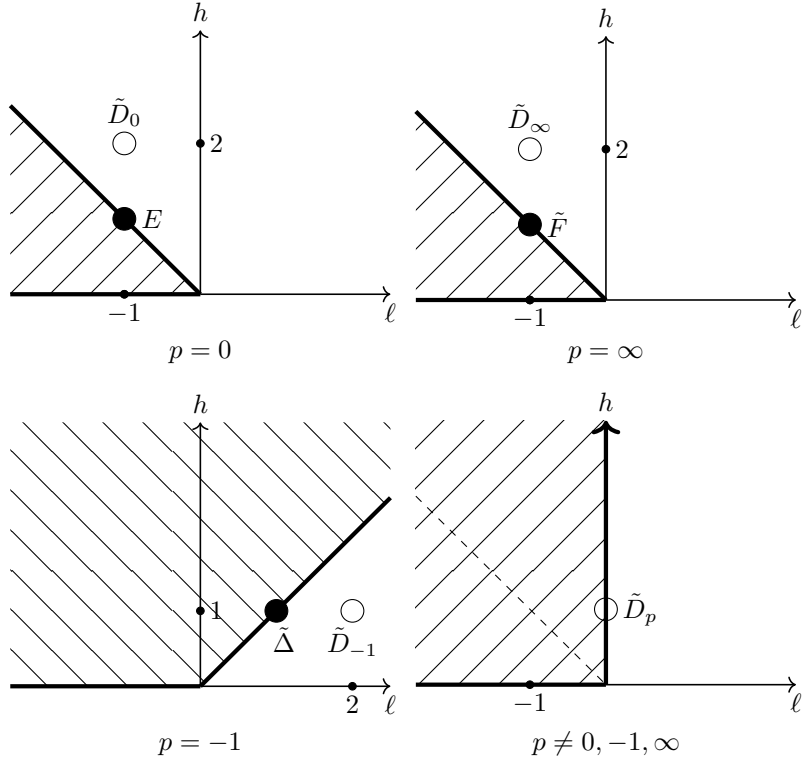
\begin{tikzpicture}
	\draw[ultra thick] (-2.5,0) to (0,0);
	\draw[semithick][->] (0,0) -- (2.5,0);
	\draw[semithick][->] (0,0) -- (0,3.5);
	\draw[ultra thick][domain = 0:2.5] plot (\x, {\x}); 
	\node[below] at (2.5,0) {$\ell$};
	\node[above] at (0,3.5) {$h$};
	\draw[fill=black] (0,1) circle [color=black,radius=0.05];
	\node[right] at (0,1) {$1$};
	\draw[fill=black] (2,0) circle [color=black,radius=0.05];
	\node[below] at (2,0) {$2$};
	\node[below] at (0,-0.5) {$p = -1$};
	\draw[fill=black] (1,1) circle [color=black,radius=0.15];
	\node[below] at (1.1,0.9) {$\tilde{\Delta}$};
	\draw (2,1) circle [radius=0.15];
	\node[below] at (2,0.9) {$\tilde{D}_{-1}$};
	\begin{scope}[on background layer]
	\fill[pattern=my north west lines](-2.5,0) to (-2.5,3.5) to (2.5,3.5) to (2.5,2.5) to (0,0);
	\end{scope}
	\end{tikzpicture}
	\begin{tikzpicture}
	\draw[ultra thick] (-2.5,0) to (0,0);
	\draw[semithick][->] (0,0) -- (2.5,0);
	\draw[ultra thick][->] (0,0) -- (0,3.5);
	\draw[dashed][domain = -2.5:0] plot (\x, {-\x}); 
	\node[below] at (2.5,0) {$\ell$};
	\node[above] at (0,3.5) {$h$};
	\draw[fill=black] (-1,0) circle [color=black,radius=0.05];
	\node[below] at (-1,0) {$-1$};
	\node[below] at (0,-0.5) {$p \neq 0,-1,\infty$};
	\draw (0,1) circle [radius=0.15];
	\node[right] at (0.1,1) {$\tilde{D}_p$};
	\begin{scope}[on background layer]
	\fill[pattern=my north east lines](-2.5,0) to (-2.5,3.5) to (0,3.5) to (0,0);
	\end{scope}
	\end{tikzpicture}\end{center}

\begingroup
\captionof{figure}{Coloured hyperfan of the blow-up of $\proj{1} \times \proj{2}$ along $C$}
\endgroup

\subsection{Blow-up of the Divisor $W$ in $\proj{2} \times \proj{2}$ (3.13)}\label{313}

\subsubsection*{Definition and Structure of $W$}

Let $G=\sl{2}$ act diagonally on $\proj{2} \times \proj{2}$ with the action on each factor $\proj{2} = \bb{P}(S^2\bbk^2)$ described in \Cref{symaction}. Then the divisor $W = \cal{Z}(x_0y_2-2x_1y_1+x_2y_0)$ is $G$-stable. The point $P = ([1:0:1],[0:1:0]) \in W$ has a 2 dimensional $B$-orbit, and a $G$-stabiliser of order 8 generated by $\begin{psmallmatrix} i & 0 \\ 0 & -i\end{psmallmatrix}$ and $\begin{psmallmatrix} 0 & i \\ i & 0\end{psmallmatrix}$, so let $H = G_P = \tilde{D}_2$, the binary dihedral group of order 8. Then we see that $W$ is a quasihomogeneous complexity-one $G$-variety containing the homogeneous space $G/H$.

Now consider the conic $C = \cal{Z}(x_0x_2-x_1^2) \sub \proj{2}$. It is $G$-stable under the action induced from $S^2\bbk^2$ by \thref{rnc}, and hence the divisors $C \times \proj{2}$ and $\proj{2} \times C$ on $\proj{2} \times \proj{2}$ are also $G$-stable. Let $E_\infty$ and $E_0$ respectively be the intersections of these divisors with $W$. Their union is the complement of $G/H$ in $W$ and their intersection is a $G$-stable curve $Z = (C \times C) \inter W$. The equations of $W$ also force $Z = \diag{(\proj{2}\times\proj{2})} \inter W$.

Thus the $G$-germs of $W$ are exactly $Z, E_0$ and $E_\infty$, with the latter two containing the former, which is minimal. 

\subsubsection*{(Semi-) Invariant Functions}

For $G/H = \sl{2}/\tilde{D}_2$, the weight lattice is $\Lambda = \bb{Z}(2\alpha)$, $\cal{Q} = \Lambda^*$ is identified with $\frac{1}{2}\bb{Z}$ and the field of invariants is generated by $f_f^2/f_v^2$. In $W$, the equations defining $E_0$ and $E_\infty$ are invariant, and $x_2$, $y_2$ are $B$-semi-invariant of weight $2\alpha$. Hence $f = y_2^2(x_0x_2-x_1^2)/x_2^2(y_0y_2-y_1^2)$ is invariant, and one can check that it does indeed correspond to $f_f^2/f_v^2$ on the open orbit.

Now $f$ defines a $B$-quotient $\pi \colon W \dashrightarrow \proj{1}$, $P \mapsto [y_2^2(x_0x_2-x_1^2):x_2^2(y_0y_2-y_1^2)]$. The pullback of $p=[\alpha:\beta] \in \proj{1}$ is $\cal{Z}(\beta y_2^2(x_0x_2-x_1^2)-\alpha x_2^2(y_0y_2-y_1^2))$ and defines a regular colour for all $p$ except for the following: \[p=[1:0] = 0: \pi^*(p) = \cal{Z}(x_2^2(y_0y_2-y_1^2)) = \cal{Z}(x_2^2) \union \cal{Z}(y_0y_2-y_1^2) = D_0 \union E_0\]\[ p=[0:1]=\infty: \pi^*(p) = \cal{Z}(y_2^2(x_0x_2-x_1^2)) = \cal{Z}(y_2^2) \union \cal{Z}(x_0x_2-x_1^2) = D_\infty \union E_\infty\]\[ p = [-1:1] = -1: \pi^*(p) = \cal{Z}(-(x_1y_2-x_2y_1)^2) = D_{-1}.\]

Here we see three subregular colours ($D_0,D_\infty,D_{-1}$) of multiplicity 2 corresponding to the three subregular semi-invariants on $G/H$, and the two $G$-divisors $E_0,E_\infty$ defining the minimal $G$-germ $Z$. One can check that every colour except $D_0,D_\infty$ contains $Z$.

The function $x_2y_2/(x_1y_2-x_2y_1)$ is semi-invariant of weight $2\alpha$ and corresponds to $f_ef_v/f_f$ in the homogeneous space, so we choose this as our splitting $e_{2\alpha}$. Its divisor is $D_0 + D_\infty - D_{-1}$, so these points are distinguished by it.

\subsubsection*{Hyperfan Before the Blow-up}

We recall that the $G$-germs of $W$ are $E_0, E_\infty$ and $Z$, with the latter being minimal. The coloured data are thus $\cal{V}_Z = \{\nu_{E_0},\nu_{E_\infty}\}$, $\cal{D}^B_Z = \{D_p \mid p\neq 0,\infty\}$ and $\cal{V}_{E_i} = \{\nu_{E_i}\}$, $\cal{D}^B_{E_i} = \emptyset$ for $i=0,\infty$. Thus $Z$ defines a supported coloured hypercone of type II in $\cal{H}$.

From our choice of invariant and splitting semi-invariant, the $G$-divisors and colours map to the following points in hyperspace: $E_0 \mapsto (0,0,1)$, $D_0 \mapsto (0,1,2)$, $E_\infty \mapsto (\infty,0,1)$, $D_\infty \mapsto (\infty,1,2)$, $D_{-1} \mapsto (-1,-1,2)$ and $D_p \mapsto (p,0,1)$ for $p \neq 0,\infty,-1$. Therefore the polytope defined by $Z$ is given by $\cal{P}_{-1} = \{-1/2\}$, $\cal{P}_p = \{0\}$ for $p\neq -1$, so $\cal{P} = \{-1/2\}$. Hence the coloured hyperfan of $W$ looks like:

\begin{center}\begin{tikzpicture}
	\draw[ultra thick] (-2.5,0) -- (0,0);
	\draw[semithick][->] (0,0) -- (2.5,0);
	\draw[ultra thick][->] (0,0) -- (0,3.5);
	\draw[dashed][domain = -2.5:0] plot (\x, {-\x}); 
	\node[below] at (2.5,0) {$\ell$};
	\node[above] at (0,3.5) {$h$};
	\node[right] at (0.1,1) {$D_p$};	
	\draw[fill=black] (-1,0) circle [color=black,radius=0.05];
	\node[below] at (-1,0) {$-1$};
	\node[below] at (0,-0.5) {$p \neq p_f,p_v,p_e$};
	\draw (0,1) circle [radius=0.15];
	\begin{scope}[on background layer]
	\fill[pattern=my north east lines](-2.5,3.5) to (-2.5,0) to (0,0) to (0,3.5);
	\end{scope}
	\end{tikzpicture}
	\begin{tikzpicture}
	\draw[ultra thick] (-2.5,0) -- (0,0);
	\draw[semithick][->] (0,0) -- (2.5,0);
	\draw[semithick][->] (0,0) -- (0,3.5);
	\draw[ultra thick][domain=-1.75:0] plot (\x, {-2*\x});
	\draw[dashed][domain = -2.5:0] plot (\x, {-\x}); 
	\draw[fill=black] (0,2) circle [color=black,radius=0.05];
	\node[right] at (0,2) {$2$};
	\draw[fill=black] (-1,0) circle [color=black,radius=0.05];
	\node[below] at (-1,0) {$-1$};
	\draw[fill=black] (-1,0) circle [color=black,radius=0.05];
	\node[below] at (2.5,0) {$\ell$};
	\node[above] at (0,3.5) {$h$};
	\node[below] at (0,-0.5) {$p_f = -1$};
	\draw (-1,2) circle [radius=0.15];
	\node[right] at (-0.9,2) {$D_{-1}$};
	\begin{scope}[on background layer]
	\fill[pattern=my north east lines](-2.5,0) to (-2.5,3.5) to (-1.75,3.5) to (0,0);
	\end{scope}
	\end{tikzpicture}\end{center}

\begin{center}\begin{tikzpicture}
	\draw[semithick][->] (0,0) -- (2.5,0);
	\draw[ultra thick] (-2.5,0) -- (0,0);
	\draw[ultra thick][->] (0,0) -- (0,3.5);
	\node[below] at (2.5,0) {$\ell$};
	\node[above] at (0,3.5) {$h$};
	\draw[fill=black] (1,0) circle [color=black,radius=0.05];
	\node[below] at (1,0) {$1$};
	\draw[fill=black] (0,2) circle [color=black,radius=0.05];
	\node[right] at (0,2) {$2$};
	\node[below] at (0,-0.5) {$p_e = 0$};
	\draw (1,2) circle [radius=0.15];
	\node[right] at (1.1,2) {$D_0$};
	\draw[fill=black] (0,1) circle [color=black,radius=0.15];
	\node[right] at (0.1,1) {$E_0$};
	\begin{scope}[on background layer]
	\fill[pattern=my north west lines](-2.5,0) to (-2.5,3.5) to (0,3.5) to (0,0);
	\end{scope}
	\end{tikzpicture}
	\begin{tikzpicture}
	\draw[semithick][->] (0,0) -- (2.5,0);
	\draw[ultra thick] (-2.5,0) -- (0,0);
	\draw[ultra thick][->] (0,0) -- (0,3.5);
	\node[below] at (2.5,0) {$\ell$};
	\node[above] at (0,3.5) {$h$};
	\draw[fill=black] (1,0) circle [color=black,radius=0.05];
	\node[below] at (1,0) {$1$};
	\draw[fill=black] (0,2) circle [color=black,radius=0.05];
	\node[right] at (0,2) {$2$};
	\node[below] at (0,-0.5) {$p_v = \infty$};
	\draw (1,2) circle [radius=0.15];
	\node[right] at (1.1,2) {$D_\infty$};
	\draw[fill=black] (0,1) circle [color=black,radius=0.15];
	\node[right] at (0.1,1) {$E_\infty$};
	\begin{scope}[on background layer]
	\fill[pattern=my north west lines](-2.5,0) to (-2.5,3.5) to (0,3.5) to (0,0);
	\end{scope}
	\end{tikzpicture}\end{center}
\begingroup

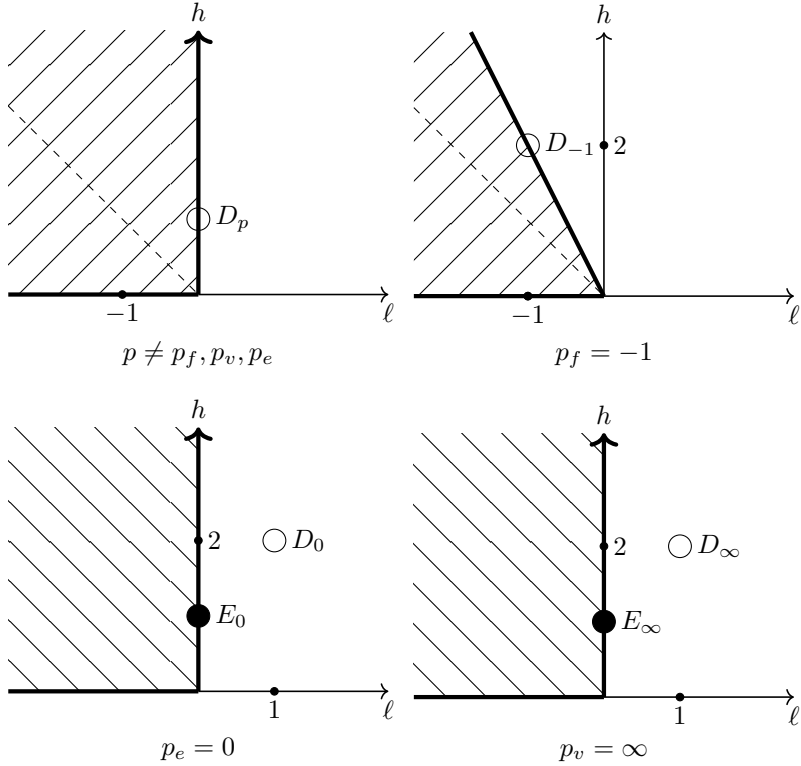
\captionof{figure}{Coloured hyperfan of the divisor $W$ on $\proj{2} \times \proj{2}$}
\endgroup

\subsubsection*{Blow-up}

To obtain the variety we want, we blow up $W$ along $Z$. Since $Z$ defines a hypercone of type II, it has a minimal $B$-chart $U = W\setminus{(D_0 \union D_\infty)}$. We will simplify matters by blowing up this chart instead.

Hence consider $U$ as an affine chart of $W$, i.e. we set $x_2 = y_2 = 1$ in $W$ to obtain $W \inter U = \cal{Z}(x_0-2x_1y_1+y_0) \sub \aff{4}$. Eliminate $x_0 = 2x_1y_1-y_0$ so that $W \inter U = \Spec{\bbk[x_1,y_0,y_1]} = \aff{3}$. Then $E_0 \inter U = \cal{Z}(y_0-y_1^2) \sub \aff{3}$, $E_\infty \inter U = \cal{Z}(2x_1y_1-y_0-x_1^2) \sub \aff{3}$ and $Z \inter U = \cal{Z}(y_0-y_1^2,2x_1y_1-y_0-x_1^2) = \cal{Z}(y_1-x_1,y_0-y_1^2) \sub \aff{3}$.

Now take $X = \bl_{U \inter Z}(W \inter Z) = \cal{Z}(z_0(y_0-y_1^2) - z_1(y_1-x_1)) \sub \aff{3} \times \proj{1}$. The exceptional divisor is $E = \cal{Z}(y_0-y_1^2,y_1-x_1)$, and we have strict transforms $\tilde{E}_0 = \cal{Z}(y_0-y_1^2,z_1)$ and $\tilde{E}_{\infty} = \cal{Z}(z_0(x_1-y_1)-z_1,2x_1y_1-y_0-x_1^2)$ of the $G$-divisors from downstairs. Any two of these three $G$-divisors intersect in the curve $Y = \cal{Z}(y_0-y_1^2,y_1-x_1,z_1)$, which is hence the unique minimal $G$-germ of the blow-up. 

The invariant rational function on $W$ becomes $f = (2x_1y_1-y_0-x_1^2)/(y_0-y_1^2)$ on $U \inter W$ and hence also on $X$. Under the induced $B$-quotient $\pi$ to $\proj{1}$ we see that $\pi^*([1:-1]) = \cal{Z}(x_1-y_1,y_0-y_1^2) \union \cal{Z}(x_1-y_1,z_0) = E \union \tilde{D}_{-1}$. Hence $E$ sits in the slice of hyperspace corresponding to $-1 \in \proj{1}$, and the colour $\tilde{D}_{-1}$ does not contain the minimal $G$-germ $Y$, as $D_{-1}$ did before the blow-up. 

Choosing a uniformising element $\delta = (2x_1y_1-y_0-x_1^2+(y_0-y_1^2))/(y_0-y_1^2) = -(y_1-x_1)^2/(y_0-y_1^2)$ of the DVR corresponding to $-1$ and taking an affine chart $z_0 = 1$ of $X$, a simple calculation shows that $h_E = \nu_E(\delta) = 1$.

Likewise, the splitting semi-invariant $e_{2\alpha}$ from above becomes $1/(x_1-y_1)$ on $X$, giving $\ell_E = \nu_E(e_{2\alpha}) = -1$. Hence $E \mapsto (-1,-1,1) \in \cal{H}$. The positions in hyperspace of all other $G$-divisors and colours remain as always unchanged by the blow-up.

Thus the curve $Y$ has coloured data $\cal{V}_Y = \{\nu_E,\nu_{\tilde{E}_0},\nu_{\tilde{E}_{\infty}}\}$, $\cal{D}^B_Y = \{D_p \mid p \neq 0,\infty,-1\}$. It defines a supported coloured hypercone of type II in $\cal{H}$ with associated polytope $\cal{P} = \cal{P}_{-1} = \{-1\}$. Hence $X$ (and therefore the variety $\bl_Z{W}$) has the following coloured hyperfan:

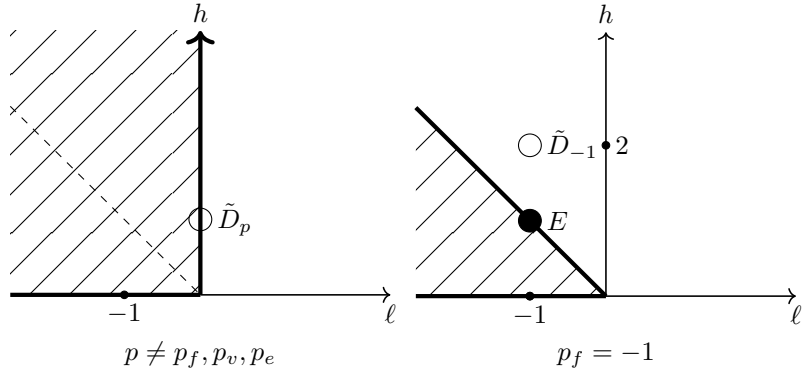
\begin{center}\begin{tikzpicture}
	\draw[ultra thick] (-2.5,0) -- (0,0);
	\draw[semithick][->] (0,0) -- (2.5,0);
	\draw[ultra thick][->] (0,0) -- (0,3.5);
	\draw[dashed][domain = -2.5:0] plot (\x, {-\x}); 
	\node[below] at (2.5,0) {$\ell$};
	\node[above] at (0,3.5) {$h$};
	\node[right] at (0.1,1) {$\tilde{D}_p$};	
	\draw[fill=black] (-1,0) circle [color=black,radius=0.05];
	\node[below] at (-1,0) {$-1$};
	\node[below] at (0,-0.5) {$p \neq p_f,p_v,p_e$};
	\draw (0,1) circle [radius=0.15];
	\begin{scope}[on background layer]
	\fill[pattern=my north east lines](-2.5,3.5) to (-2.5,0) to (0,0) to (0,3.5);
	\end{scope}
	\end{tikzpicture}
	\begin{tikzpicture}
	\draw[ultra thick] (-2.5,0) -- (0,0);
	\draw[semithick][->] (0,0) -- (2.5,0);
	\draw[semithick][->] (0,0) -- (0,3.5);
	\draw[ultra thick][domain = -2.5:0] plot (\x, {-\x}); 
	\draw[fill=black] (0,2) circle [color=black,radius=0.05];
	\node[right] at (0,2) {$2$};
	\draw[fill=black] (-1,0) circle [color=black,radius=0.05];
	\node[below] at (-1,0) {$-1$};
	\draw[fill=black] (-1,0) circle [color=black,radius=0.05];
	\node[below] at (2.5,0) {$\ell$};
	\node[above] at (0,3.5) {$h$};
	\node[below] at (0,-0.5) {$p_f = -1$};
	\draw[fill=black] (-1,1) circle [color=black,radius=0.15];
	\node[right] at (-0.9,1) {$E$};
	\draw (-1,2) circle [radius=0.15];
	\node[right] at (-0.9,2) {$\tilde{D}_{-1}$};
	\begin{scope}[on background layer]
	\fill[pattern=my north east lines](-2.5,0) to (-2.5,2.5) to (0,0);
	\end{scope}
	\end{tikzpicture}\end{center}

\begin{center}\begin{tikzpicture}
	\draw[semithick][->] (0,0) -- (2.5,0);
	\draw[ultra thick] (-2.5,0) -- (0,0);
	\draw[ultra thick][->] (0,0) -- (0,3.5);
	\node[below] at (2.5,0) {$\ell$};
	\node[above] at (0,3.5) {$h$};
	\draw[fill=black] (1,0) circle [color=black,radius=0.05];
	\node[below] at (1,0) {$1$};
	\draw[fill=black] (0,2) circle [color=black,radius=0.05];
	\node[right] at (0,2) {$2$};
	\node[below] at (0,-0.5) {$p_e = 0$};
	\draw (1,2) circle [radius=0.15];
	\node[right] at (1.1,2) {$\tilde{D}_0$};
	\draw[fill=black] (0,1) circle [color=black,radius=0.15];
	\node[right] at (0.1,1) {$\tilde{E}_0$};
	\begin{scope}[on background layer]
	\fill[pattern=my north west lines](-2.5,0) to (-2.5,3.5) to (0,3.5) to (0,0);
	\end{scope}
	\end{tikzpicture}
	\begin{tikzpicture}
	\draw[semithick][->] (0,0) -- (2.5,0);
	\draw[ultra thick] (-2.5,0) -- (0,0);
	\draw[ultra thick][->] (0,0) -- (0,3.5);
	\node[below] at (2.5,0) {$\ell$};
	\node[above] at (0,3.5) {$h$};
	\draw[fill=black] (1,0) circle [color=black,radius=0.05];
	\node[below] at (1,0) {$1$};
	\draw[fill=black] (0,2) circle [color=black,radius=0.05];
	\node[right] at (0,2) {$2$};
	\node[below] at (0,-0.5) {$p_v = \infty$};
	\draw (1,2) circle [radius=0.15];
	\node[right] at (1.1,2) {$\tilde{D}_\infty$};
	\draw[fill=black] (0,1) circle [color=black,radius=0.15];
	\node[right] at (0.1,1) {$\tilde{E}_\infty$};
	\begin{scope}[on background layer]
	\fill[pattern=my north west lines](-2.5,0) to (-2.5,3.5) to (0,3.5) to (0,0);
	\end{scope}
	\end{tikzpicture}\end{center}
\begingroup

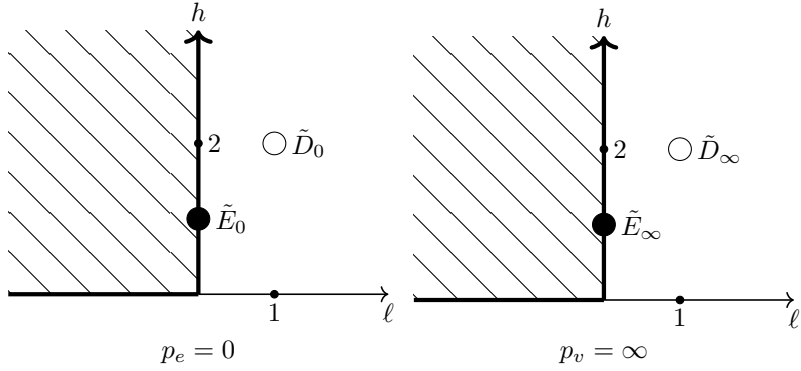
\captionof{figure}{Coloured hyperfan of the blow-up of $W$ along $Z$}
\endgroup

\subsection{Blow-up of $\proj{3}$ Along the Twisted Cubic (2.27)}\label{227}

\subsubsection*{Homogeneous Space and Structure}

Let $G=\sl{2}$ act on $\proj{3} = \bb{P}(S^3\bbk^2)$ as in \Cref{symaction}. The point $P = [1:0:0:1]$ has as its $G$-stabiliser the binary dihedral group $H = \tilde{D}_3$ of degree 3 and order 12. The same point has $B$-stabiliser equal to the group of sixth roots of unity, so $\dim{B\cdot P} = 2$. Hence $G\cdot P$ is the homogeneous space $G/H$, realising $\proj{3}$ (in a different way to previous examples) as a quasihomogeneous complexity-one $G$-variety.

Let $Z$ be the rational normal curve of degree 3 in $\proj{3}$, i.e. $Z = \cal{Z}(x_0x_2-x_1^2,x_0x_3-x_1x_2,x_1x_3-x_2^2)$. Then $Z$ is a closed orbit in $\proj{3}$ by \thref{rnc}. We will see that $Z$ is the unique minimal $G$-germ of $\proj{3}$, and is contained in its unique $G$-divisor.

\subsubsection*{(Semi-) Invariant Functions}

By \Cref{dihedral}, there are semi-invariant functions $f_v, f_e, f_f$ in $\bbk[G]^{(B\times H)}$ of respective biweights $(3\alpha,(-1,3)), (3\alpha,(-1,1))$ and $(2\alpha,(1,2))$. 

In $\proj{3}$, $x_3$ is semi-invariant of $B$-weight $3\alpha$, and $x_1x_3-x_2^2$ is semi-invariant of $B$-weight $2\alpha$ by \thref{semi}. Hence the latter corresponds to $f_f$. On $G/H$, we see that, acting with $\begin{psmallmatrix} 0 & i \\ i & 0\end{psmallmatrix} \in H$ on the right, $x_3$ has $H$-weight $(-1,3)$, so $x_3$ corresponds to $f_v$ Finally, $2x_2^3-3x_1x_2x_3+x_0x_3^2$ has $B$-weight $3\alpha$ and thus by process of elimination it corresponds to $f_e$.

Now the function $f_f^3/f_e^2$ is $B$-invariant, so gives the $B$-quotient map $\pi$ to $\proj{1}$, i.e. $\pi(P) = [f_f^3(P):f_e^2(P)]$ for $P \in \proj{3}$. This defines a family of regular colours $D_p = \pi^*(p)$ in $\proj{3}$ for all $p=[\alpha:\beta] \in \proj{1}$ except: \[p = 0: \pi^*(p) = \cal{Z}(f_e^2) = D_0,\]\[p = \infty: \pi^*(p) = \cal{Z}(f_f^3) = D_\infty,\]\[p = -4: \pi^*(p) = \cal{Z}(f_v^2) \union \cal{Z}(3x_1^2x_2^2-4x_1^3x_3-x_0^2x_3^2-4x_0x_2^3+6x_0x_1x_2x_3) = D_{-4} \union F\] where $F$ is a $G$-divisor. The subregular colours $D_0, D_\infty$ and $D_{-4}$ have multiplicities 2, 3 and 2, respectively, in accordance with the pictures above.

Finally, we choose a splitting $e_{2\alpha} = f_vf_e/f_f^2 \in K^{(B)}_{2\alpha}$.

\subsubsection*{Coloured Data and Hyperfan of $\proj{3}$}

The minimal $G$-germ is contained in $F$ and in every colour except $D_{-4}$, so has coloured data $\cal{V}_Z = \{\nu_F\}, D^B_Z = \{D_p \mid p\neq -4\}$. It therefore defines a supported coloured hypercone of type II in $\cal{H}$.

By our choice of splitting we see that $D_0 \mapsto (\ell,h) = (1,2)$ in hyperspace, $D_\infty \mapsto (-2,3)$, $D_{-4} \mapsto (1,2)$, $F \mapsto (0,1)$ and as always the regular colours $D_p \mapsto (0,1)$ for $p \neq 0,\infty,-4$. 

Therefore the polytope defined by $Z$ is given by $\cal{P} = \cal{P}_0 + \cal{P}_\infty + \cal{P}_{-4} = \{1/2 - 1/3 + 0\} = \{-1/6\}$.

Hence the coloured hyperfan of $\proj{3}$ looks like:

\begin{center}\begin{tikzpicture}
	\draw[ultra thick](-2.5,0) -- (0,0);
	\draw[semithick][->] (0,0) -- (2.5,0);
	\draw[semithick][->] (0,0) -- (0,3.5);
	\draw[ultra thick](0,0) -- (-7/3,3.5);
	\draw[dashed](0,0) -- (-2.5,2.5);
	\draw[fill=black] (0,3) circle [color=black,radius=0.05];
	\node[right] at (0,3) {$3$};
	\draw[fill=black] (-2,0) circle [color=black,radius=0.05];
	\node[below] at (-2,0) {$-2$};
	\node[below] at (2.5,0) {$\ell$};
	\node[above] at (0,3.5) {$h$};
	\node[below] at (0,-0.5) {$p_f=\infty$};
	\draw (-2,3) circle [radius=0.15];
	\node[right] at (-2,3) {$D_{\infty}$};
	\begin{scope}[on background layer]
	\fill[pattern=my north east lines](-2.5,3.5) to (-7/3,3.5) to (0,0) to (-2.5,0);
	\end{scope}
	\end{tikzpicture}
	\begin{tikzpicture}
	\draw[ultra thick](-2.5,0) -- (0,0);
	\draw[semithick][->] (0,0) -- (2.5,0);
	\draw[ultra thick][->] (0,0) -- (0,3.5);
	\node[below] at (2.5,0) {$\ell$};
	\node[above] at (0,3.5) {$h$};
	\draw[fill=black] (0,2) circle [color=black,radius=0.05];
	\node[right] at (0,2) {$2$};
	\draw[fill=black] (1,0) circle [color=black,radius=0.05];
	\node[below] at (1,0) {$1$};
	\node[below] at (0,-0.5) {$p_v=-4$};
	\draw (1,2) circle [radius=0.15];
	\node[right] at (1.1,2) {$D_{-4}$};
	\draw[fill=black] (0,1) circle [color=black,radius=0.15];
	\node[right] at (0,1) {$F$};
	\begin{scope}[on background layer]
	\fill[pattern=my north west lines](-2.5,0) to (-2.5,3.5) to (0,3.5) to (0,0);
	\end{scope}
	\end{tikzpicture}\end{center}

\begin{center}
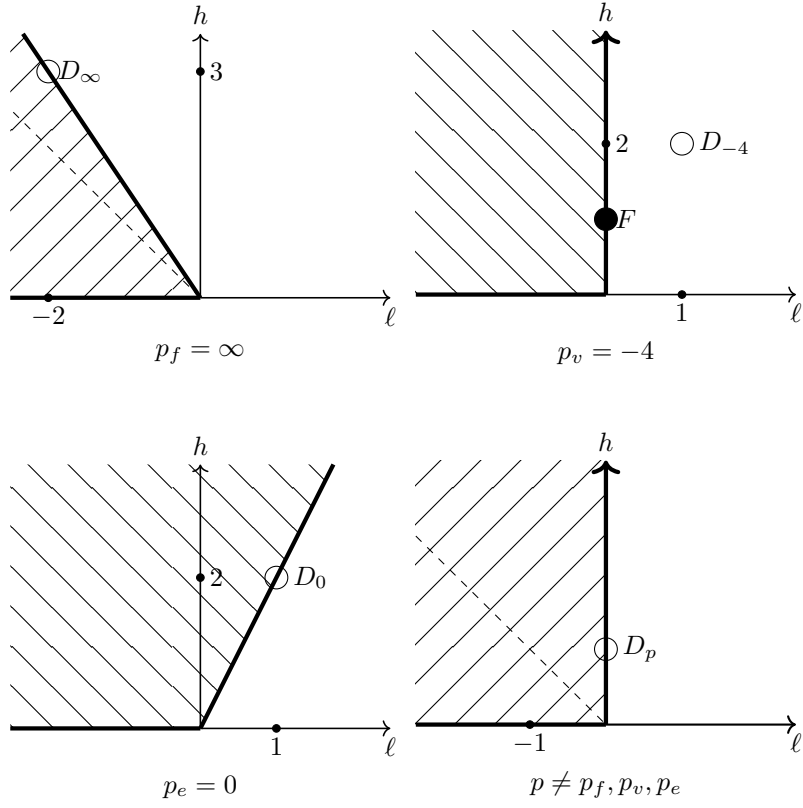
\begin{tikzpicture}
	\draw[ultra thick](-2.5,0) -- (0,0);
	\draw[semithick][->] (0,0) -- (2.5,0);
	\draw[semithick][->] (0,0) -- (0,3.5);
	\draw[ultra thick] (0,0) -- (1.75,3.5);
	\node[below] at (2.5,0) {$\ell$};
	\node[above] at (0,3.5) {$h$};
	\draw[fill=black] (0,2) circle [color=black,radius=0.05];
	\node[right] at (0,2) {$2$};
	\draw[fill=black] (1,0) circle [color=black,radius=0.05];
	\node[below] at (1,0) {$1$};
	\node[below] at (0,-0.5) {$p_e=0$};
	\draw (1,2) circle [radius=0.15];
	\node[right] at (1.1,2) {$D_0$};
	\begin{scope}[on background layer]
	\fill[pattern=my north west lines](-2.5,0) to (-2.5,3.5) to (1.75,3.5) to (0,0);
	\end{scope}
	\end{tikzpicture}
	\begin{tikzpicture}
	\draw[ultra thick] (-2.5,0) -- (0,0);
	\draw[thick][->] (0,0) -- (2.5,0);
	\draw[ultra thick][->] (0,0) -- (0,3.5);
	\draw[dashed][domain = -2.5:0] plot (\x, {-\x}); 
	\node[below] at (2.5,0) {$\ell$};
	\node[above] at (0,3.5) {$h$};
	\node[right] at (0.1,1) {$D_p$};
	\draw[fill=black] (-1,0) circle [color=black,radius=0.05];
	\node[below] at (-1,0) {$-1$};
	\node[below] at (0,-0.5) {$p \neq p_f,p_v,p_e$};
	\draw (0,1) circle [radius=0.15];
	\begin{scope}[on background layer]
	\fill[pattern=my north east lines](-2.5,3.5) to (0,3.5) to (0,0) to (-2.5,0);
	\end{scope}
	\end{tikzpicture}\end{center}
\begingroup
\captionof{figure}{Coloured hyperfan of $\proj{3}$ (cubic action)}
\endgroup

\subsubsection*{Blow-up}

Now to get the variety we need, we blow up $Z$. To simplify what happens, we will take affine charts. Take the minimal $B$-chart for $Z$, which is $U_Z = \proj{3}\setminus{\cal{Z}(x_3)}$. By setting $x_3 =1$, we get $U_Z \inter \proj{3} = \aff{3} = \Spec{\bbk[x_0,x_1,x_2]}$. Then $Z \inter U_Z$ becomes $\cal{Z}(x_0-x_1x_2,x_1-x_2^2)$. 

Now $X = \bl_{Z \inter U_Z}(\aff{3}) = \cal{Z}(z_0(x_1-x_2^2)-z_1(x_0-x_1x_2)) \sub \aff{3} \times \proj{1}$. The exceptional divisor is $E= \cal{Z}(x_1-x_2^2,x_0-x_1x_2)$. Take another affine chart $V$ defined by $z_1 =1$. Using the equation for $X$ we can eliminate $x_0$ to obtain $X \inter V = \aff{3} = \Spec{\bbk[x_1,x_2,z_0]}$. Now $E \inter V = \cal{Z}(x_1-x_2^2)$.

On $V \inter X$, the $B$-invariant above becomes $(x_2^2-x_1)/(2x_2(1+z_0))^2$, and the splitting semi-invariant becomes $2x_2(1+z_0)/(x_1-x_2^2)$. Hence $E$ sits over $\infty$ in hyperspace and is mapped to $(-1,1) \in \cal{H}_\infty$.

Since the blow-up is an isomorphism away from $Z$, nothing else in hyperspace moves from its previous position. There is a new minimal $G$-germ $Y = E \inter \tilde{F}$ with coloured data $\cal{V}_Y = \{\nu_E,\nu_{\tilde{F}}\}, \cal{D}^B_Y = \{D_p \mid p\neq \infty,-4\}$. The coloured hyperfan for $X$ is thus:

\begin{center}\begin{tikzpicture}
	\draw[ultra thick](-2.5,0) -- (0,0);
	\draw[semithick][->] (0,0) -- (2.5,0);
	\draw[semithick][->] (0,0) -- (0,3.5);
	\draw[ultra thick](0,0) -- (-2.5,2.5);
	\draw[fill=black] (0,3) circle [color=black,radius=0.05];
	\node[right] at (0,3) {$3$};
	\draw[fill=black] (-2,0) circle [color=black,radius=0.05];
	\node[below] at (-2,0) {$-2$};
	\node[below] at (2.5,0) {$\ell$};
	\node[above] at (0,3.5) {$h$};
	\node[below] at (0,-0.5) {$p_f=\infty$};
	\draw[fill=black] (-1,1) circle [color=black,radius=0.15];
	\node[right] at (-1,1) {$E$};
	\draw (-2,3) circle [radius=0.15];	
	\node[right] at (-2,3) {$\tilde{D}_{\infty}$};
	\begin{scope}[on background layer]
	\fill[pattern=my north east lines](-2.5,0) to (-2.5,2.5) to (0,0);
	\end{scope}
	\end{tikzpicture}
	\begin{tikzpicture}
	\draw[ultra thick](-2.5,0) -- (0,0);
	\draw[semithick][->] (0,0) -- (2.5,0);
	\draw[ultra thick][->] (0,0) -- (0,3.5);
	\node[below] at (2.5,0) {$\ell$};
	\node[above] at (0,3.5) {$h$};
	\draw[fill=black] (0,2) circle [color=black,radius=0.05];
	\node[right] at (0,2) {$2$};
	\draw[fill=black] (1,0) circle [color=black,radius=0.05];
	\node[below] at (1,0) {$1$};
	\node[below] at (0,-0.5) {$p_v=-4$};
	\draw (1,2) circle [radius=0.15];
	\node[right] at (1.1,2) {$\tilde{D}_{-4}$};
	\draw[fill=black] (0,1) circle [color=black,radius=0.15];
	\node[right] at (0,1) {$\tilde{F}$};
	\begin{scope}[on background layer]
	\fill[pattern=my north west lines](-2.5,0) to (-2.5,3.5) to (0,3.5) to (0,0);
	\end{scope}
	\end{tikzpicture}\end{center}

\begin{center}
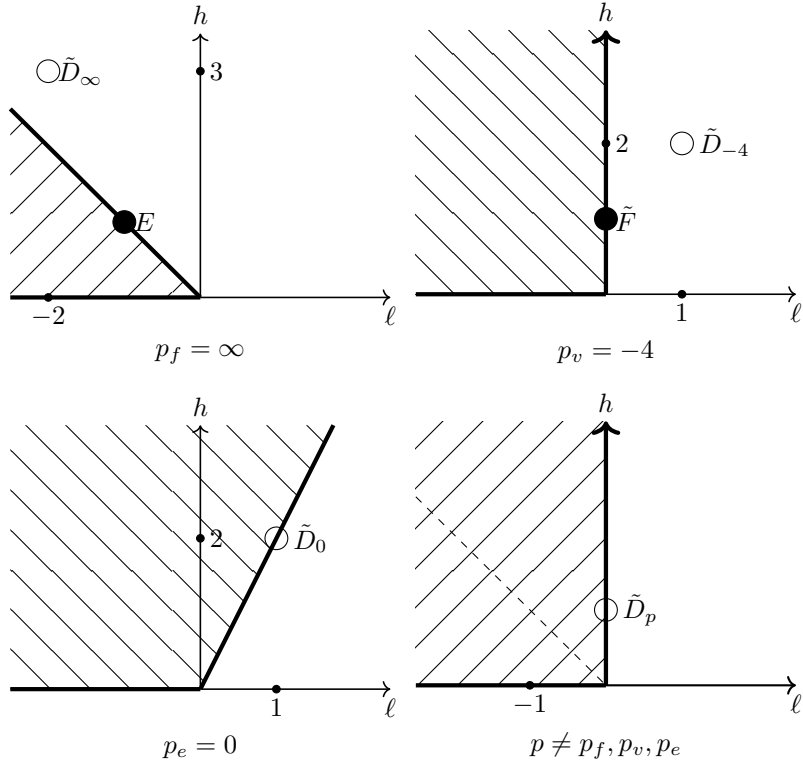
\begin{tikzpicture}
	\draw[ultra thick](-2.5,0) -- (0,0);
	\draw[semithick][->] (0,0) -- (2.5,0);
	\draw[semithick][->] (0,0) -- (0,3.5);
	\draw[ultra thick] (0,0) -- (1.75,3.5);
	\node[below] at (2.5,0) {$\ell$};
	\node[above] at (0,3.5) {$h$};
	\draw[fill=black] (0,2) circle [color=black,radius=0.05];
	\node[right] at (0,2) {$2$};
	\draw[fill=black] (1,0) circle [color=black,radius=0.05];
	\node[below] at (1,0) {$1$};
	\node[below] at (0,-0.5) {$p_e=0$};
	\draw (1,2) circle [radius=0.15];
	\node[right] at (1.1,2) {$\tilde{D}_0$};
	\begin{scope}[on background layer]
	\fill[pattern=my north west lines](-2.5,0) to (-2.5,3.5) to (1.75,3.5) to (0,0);
	\end{scope}
	\end{tikzpicture}
	\begin{tikzpicture}
	\draw[ultra thick] (-2.5,0) -- (0,0);
	\draw[thick][->] (0,0) -- (2.5,0);
	\draw[ultra thick][->] (0,0) -- (0,3.5);
	\draw[dashed][domain = -2.5:0] plot (\x, {-\x}); 
	\node[below] at (2.5,0) {$\ell$};
	\node[above] at (0,3.5) {$h$};
	\node[right] at (0.1,1) {$\tilde{D}_p$};
	\draw[fill=black] (-1,0) circle [color=black,radius=0.05];
	\node[below] at (-1,0) {$-1$};
	\node[below] at (0,-0.5) {$p \neq p_f,p_v,p_e$};
	\draw (0,1) circle [radius=0.15];
	\begin{scope}[on background layer]
	\fill[pattern=my north east lines](-2.5,3.5) to (0,3.5) to (0,0) to (-2.5,0);
	\end{scope}
	\end{tikzpicture}\end{center}
\begingroup
\captionof{figure}{Coloured hyperfan of the blow-up of $\proj{3}$ along the twisted cubic}
\endgroup

\subsection{Blow-up of the Quadric Threefold (2.21)}\label{221}

\subsubsection*{Homogeneous Space and Structure}

Let $G = \sl{2}$ act on $\proj{4} = \bb{P}(S^4\bbk^2)$ as in \Cref{symaction}. Then the quadric hypersurface $Q = \cal{Z}(3x_2^2-4x_1x_3+x_0x_4)$ is a smooth, $G$-stable threefold. The stabiliser $G_P$ of the point $P = [0:1:0:0:1] \in Q$ is the binary tetrahedral group $\widetilde{T}$, hence the orbit $G\cdot P$ is 3 dimensional and thus open. The same point has the group of sixth roots of unity as its $B$-stabiliser, so has a 2-dimensional $B$ orbit, so that $Q$ is a quasihomogeneous complexity one $G$-variety containing the homogeneous space $G/H = \sl{2}/\tilde{T}$.

Let $Z$ be the rational normal curve of degree 4 in $\proj{4}$, i.e. $Z = \cal{Z}(x_0x_2-x_1^2,x_0x_3-x_1x_2,x_1x_4-x_2x_3,x_2x_4-x_3^2)$. Then $Z$ is a closed $G$-orbit in $Q$ by \thref{rnc}. We will see that $Z$ is the unique minimal $G$-germ in $Q$, and is contained in a unique $G$-divisor.

\subsubsection*{(Semi-) Invariant Functions}

We know from \Cref{tetrahedral} that there are semi-invariant regular functions $f_e, f_f, f_v$ in $\bbk[G]^{(B \times H)}$ of respective biweights $(6\alpha,1), (4\alpha, \epsilon^{-1})$ and $(4\alpha,\epsilon)$, where $\epsilon$ is a primitive cube root of unity.

On $Q$, $x_4$ and $x_2x_4-x_3^2$ have $B$-weight $4\alpha$ by \thref{semi}, and checking on the homogeneous space $G/H = G\cdot[0:1:0:0:1]$ we see that acting by $\begin{psmallmatrix}\epsilon & 0 \\ 0 & \epsilon^{-1}\end{psmallmatrix}$ on the right, they have $H$-weights $\epsilon^{-1}$ and $\epsilon$ respectively. Hence $x_4 = f_f$, $x_2x_4-x_3^2 = f_v$. Finally, the function $2x_3^3+x_1x_4^2-3x_2x_3x_4$ is $B$-semi-invariant of weight $6\alpha$ and $H$-invariant, so this is $f_e$.

Now $f_v^3/f_e^2$ is a $B$-invariant rational function on $Q$, so the $B$-quotient map $\pi$ is given by $\pi(P) = [f_v^3(P):f_e^2(P)]$ for $P \in Q$, and defines a family of regular colours $D_p = \pi^*(p)$ for all $p=[\alpha:\beta] \in \proj{1}$ except: \[p = 0: \pi^*(p) = \cal{Z}(f_e^2) = D_0,\]\[p = \infty: \pi^*(p) = \cal{Z}(f_v^3) = D_\infty,\]\[p = -4: \pi^*(p) = \cal{Z}(f_f^3)\union\cal{Z}(4x_2^3+x_1^2x_4+x_0x_3^2-6x_1x_2x_3) = D_{-4} \union F\] where $F$ is a $G$-invariant divisor. Note that the subregular colours $D_0, D_\infty$ and $D_{-4}$ have multiplicities 2, 3 and 3, respectively, as we should expect from the hyperspace of $G/H$.

Finally, we choose as a splitting semi-invariant $e_{2\alpha} = f_vf_f/f_e \in K^{(B)}_{2\alpha}$.

\subsubsection*{Coloured Data and Hyperfan of Q}

The minimal $G$-germ $Z$ is contained in $F$ and in every colour except $D_{-4}$, so has coloured data $\cal{V}_Z = \{\nu_F\}$, $\cal{D}^B_Z = \{D_p \mid p\neq -4\}$. Hence it defines a supported coloured hypercone of type II in $\cal{H}$. 

Our choice of invariant and splitting functions mean that $D_0$ sits at $(\ell,h) = (-1,2)$ in hyperspace, $D_{-4}, D_\infty \mapsto (1,3)$, $F \mapsto (0,1)$ and as always for regular colours we have $D_p \mapsto (0,1)$ for $p\neq 0,\infty,-4$. 

Hence the polytope defined by $Z$ is given by $\cal{P} = \cal{P}_0 + \cal{P}_\infty + \cal{P}_{-4} = \{-1/2 + 1/3 + 0\} = \{-1/6\}$. Therefore the full coloured hyperfan defined by $Q$ is as follows:

\begin{center}\begin{tikzpicture}
	\draw[ultra thick](-2.5,0) -- (0,0);
	\draw[semithick][->] (0,0) -- (2.5,0);
	\draw[ultra thick][->] (0,0) -- (0,3.5);
	\draw[dashed][domain = -2.5:0] plot (\x, {-\x}); 
	\node[below] at (2.5,0) {$\ell$};
	\node[above] at (0,3.5) {$h$};
	\node[right] at (0.2,1) {$D_p$};
	\draw[fill=black] (-1,0) circle [color=black,radius=0.05];
	\node[below] at (-1,0) {$-1$};
	\node[below] at (0,-0.5) {$p \neq p_f,p_v,p_e$};
	\draw (0,1) circle [radius=0.15];
	\begin{scope}[on background layer]
	\fill[pattern=my north east lines](-2.5,3.5) to (-2.5,0) to (0,0) to (0,3.5);
	\end{scope}
	\end{tikzpicture}
	\begin{tikzpicture}
	\draw[ultra thick](-2.5,0) -- (0,0);
	\draw[semithick][->] (0,0) -- (2.5,0);
	\draw[semithick][->] (0,0) -- (0,3.5);
	\draw[dashed][domain = -2.5:0] plot (\x, {-\x}); 
	\draw[ultra thick](0,0) -- (-1.75,3.5);
	\draw[fill=black] (0,2) circle [color=black,radius=0.05];
	\node[right] at (0,2) {$2$};
	\draw[fill=black] (-1,0) circle [color=black,radius=0.05];
	\node[below] at (-1,0) {$-1$};
	\draw[fill=black] (-1,0) circle [color=black,radius=0.05];
	\node[below] at (2.5,0) {$\ell$};
	\node[above] at (0,3.5) {$h$};
	\node[below] at (0,-0.5) {$p_e=0$};
	\draw (-1,2) circle [radius=0.15];
	\node[right] at (-1,2) {$D_0$};
	\begin{scope}[on background layer]
	\fill[pattern=my north east lines](-2.5,0) to (-2.5,3.5) to (-1.75,3.5) to (0,0);
	\end{scope}
	\end{tikzpicture}\end{center}

\begin{center}
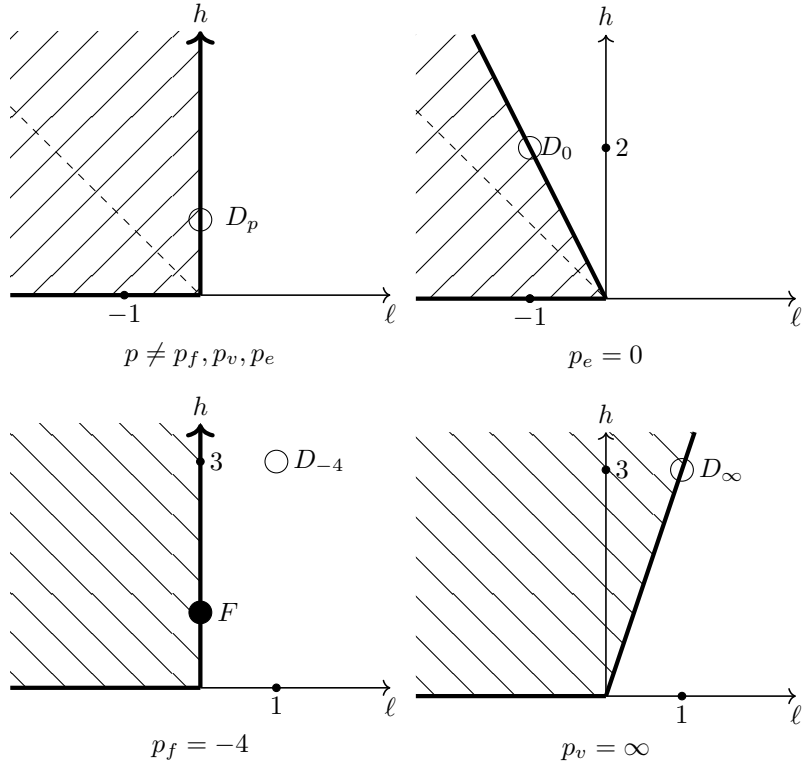
\begin{tikzpicture}
	\draw[ultra thick] (-2.5,0) -- (0,0);
	\draw[semithick][->] (0,0) -- (2.5,0);
	\draw[ultra thick][->] (0,0) -- (0,3.5);
	\node[below] at (2.5,0) {$\ell$};
	\node[above] at (0,3.5) {$h$};
	\draw[fill=black] (1,0) circle [color=black,radius=0.05];
	\node[below] at (1,0) {$1$};
	\draw[fill=black] (0,3) circle [color=black,radius=0.05];
	\node[right] at (0,3) {$3$};
	\draw[fill=black] (0,1) circle [color=black,radius=0.15];
	\node[right] at (0.1,1) {$F$};
	\node[below] at (0,-0.5) {$p_f=-4$};
	\draw (1,3) circle [radius=0.15];
	\node[right] at (1.1,3) {$D_{-4}$};
	\begin{scope}[on background layer]
	\fill[pattern=my north west lines](-2.5,0) to (-2.5,3.5) to (0,3.5) to (0,0);
	\end{scope}
	\end{tikzpicture}
	\begin{tikzpicture}
	\draw[ultra thick] (-2.5,0) -- (0,0);
	\draw[semithick][->] (0,0) -- (2.5,0);
	\draw[semithick][->] (0,0) -- (0,3.5);
	\draw[ultra thick] (0,0) -- (3.5/3,3.5);
	\node[below] at (2.5,0) {$\ell$};
	\node[above] at (0,3.5) {$h$};
	\draw[fill=black] (1,0) circle [color=black,radius=0.05];
	\node[below] at (1,0) {$1$};
	\draw[fill=black] (0,3) circle [color=black,radius=0.05];
	\node[right] at (0,3) {$3$};
	\node[below] at (0,-0.5) {$p_v=\infty$};
	\draw (1,3) circle [radius=0.15];
	\node[right] at (1.1,3) {$D_\infty$};
	\begin{scope}[on background layer]
	\fill[pattern=my north west lines](-2.5,0) to (-2.5,3.5) to (3.5/3,3.5) to (0,0);
	\end{scope}
	\end{tikzpicture}\end{center}
\begingroup
\captionof{figure}{Coloured hyperfan of the quadric threefold $Q$}
\endgroup

\subsubsection*{Blow-up}

The variety we want is obtained by blowing up $Z$. We calculate the effect on the hyperfan by blowing up in an affine chart. Since $Z$ defines a coloured hypercone of type II, it has a minimal $B$-chart $U_Z$ with the same coloured data. Indeed $Z$ is contained in every colour except $D_{-4} = \cal{Z}(x_4)$, so $U_Z = \proj{4}\setminus{\cal{Z}(x_4)}$. Setting $x_4 = 1$ allows us to eliminate $x_0$ using the equation for $Q$, so $U_Z \inter Q \cong \aff{3} = \Spec{\bbk[x_1,x_2,x_3]}$. Then $U_Z \inter Z = \cal{Z}(x_1-x_2x_3,x_2-x_3^2)$.

Now $X = \bl_{U_Z \inter Z}(\aff{3}) = \cal{Z}(z_0(x_2-x_3^2)-z_1(x_1-x_2x_3)) \sub \aff{3} \times \proj{1}$. The exceptional divisor $E$ is given by $\cal{Z}(x_1-x_2x_3,x_2-x_3^2)$. Now we take another affine chart $V$ by setting $z_1=1$, which allows us to eliminate $x_1$ using the equation for $X$. Thus $V \inter X = \aff{3} = \Spec{\bbk[x_2,x_3,z_0]}$, and $E \inter V = \cal{Z}(x_2-x_3^2)$. 

On $V \inter X$, the invariant function $f_v^3/f_e^2$ becomes $(x_2-x_3^2)/(z_0-2x_3)^2$, and the splitting semi-invariant is $1/(z_0-2x_3)$, so we see that $E \mapsto (p,\ell,h) = (\infty,0,1)$ in hyperspace.

Since the blow-up is an isomorphism away from $Z$, all other colours and $G$-divisors lie at the same points in hyperspace as before. The blow-up introduces a new minimal $G$-germ $Y = E \inter \tilde{F}$ which must have coloured data $\cal{V}_Y = \{\nu_E,\nu_{\tilde{F}}\}$, $\cal{D}^B_Y = \{D_p \mid p\neq \infty,-4\}$. Hence the coloured hyperfan for $X$ looks like:

\begin{center}\begin{tikzpicture}
	\draw[ultra thick](-2.5,0) -- (0,0);
	\draw[semithick][->] (0,0) -- (2.5,0);
	\draw[ultra thick][->] (0,0) -- (0,3.5);
	\draw[dashed][domain = -2.5:0] plot (\x, {-\x}); 
	\node[below] at (2.5,0) {$\ell$};
	\node[above] at (0,3.5) {$h$};
	\node[right] at (0.2,1) {$\tilde{D}_p$};
	\draw[fill=black] (-1,0) circle [color=black,radius=0.05];
	\node[below] at (-1,0) {$-1$};
	\node[below] at (0,-0.5) {$p \neq p_f,p_v,p_e$};
	\draw (0,1) circle [radius=0.15];
	\begin{scope}[on background layer]
	\fill[pattern=my north east lines](-2.5,3.5) to (-2.5,0) to (0,0) to (0,3.5);
	\end{scope}
	\end{tikzpicture}
	\begin{tikzpicture}
	\draw[ultra thick](-2.5,0) -- (0,0);
	\draw[semithick][->] (0,0) -- (2.5,0);
	\draw[semithick][->] (0,0) -- (0,3.5);
	\draw[dashed][domain = -2.5:0] plot (\x, {-\x}); 
	\draw[ultra thick](0,0) -- (-1.75,3.5);
	\draw[fill=black] (0,2) circle [color=black,radius=0.05];
	\node[right] at (0,2) {$2$};
	\draw[fill=black] (-1,0) circle [color=black,radius=0.05];
	\node[below] at (-1,0) {$-1$};
	\draw[fill=black] (-1,0) circle [color=black,radius=0.05];
	\node[below] at (2.5,0) {$\ell$};
	\node[above] at (0,3.5) {$h$};
	\node[below] at (0,-0.5) {$p_e=0$};
	\draw (-1,2) circle [radius=0.15];
	\node[right] at (-1,2) {$\tilde{D}_0$};
	\begin{scope}[on background layer]
	\fill[pattern=my north east lines](-2.5,0) to (-2.5,3.5) to (-1.75,3.5) to (0,0);
	\end{scope}
	\end{tikzpicture}\end{center}

\begin{center}\begin{tikzpicture}
	\draw[ultra thick] (-2.5,0) -- (0,0);
	\draw[semithick][->] (0,0) -- (2.5,0);
	\draw[ultra thick][->] (0,0) -- (0,3.5);
	\node[below] at (2.5,0) {$\ell$};
	\node[above] at (0,3.5) {$h$};
	\draw[fill=black] (1,0) circle [color=black,radius=0.05];
	\node[below] at (1,0) {$1$};
	\draw[fill=black] (0,3) circle [color=black,radius=0.05];
	\node[right] at (0,3) {$3$};
	\draw[fill=black] (0,1) circle [color=black,radius=0.15];
	\node[right] at (0.1,1) {$\tilde{F}$};
	\node[below] at (0,-0.5) {$p_f=-4$};
	\draw (1,3) circle [radius=0.15];
	\node[right] at (1.1,3) {$\tilde{D}_{-4}$};
	\begin{scope}[on background layer]
	\fill[pattern=my north west lines](-2.5,0) to (-2.5,3.5) to (0,3.5) to (0,0);
	\end{scope}
	\end{tikzpicture}
	\begin{tikzpicture}
	\draw[ultra thick] (-2.5,0) -- (0,0);
	\draw[semithick][->] (0,0) -- (2.5,0);
	\draw[ultra thick][->] (0,0) -- (0,3.5);
	\node[below] at (2.5,0) {$\ell$};
	\node[above] at (0,3.5) {$h$};
	\draw[fill=black] (1,0) circle [color=black,radius=0.05];
	\node[below] at (1,0) {$1$};
	\draw[fill=black] (0,3) circle [color=black,radius=0.05];
	\node[right] at (0,3) {$3$};
	\draw[fill=black] (0,1) circle [color=black,radius=0.15];
	\node[right] at (0.1,1) {$E$};
	\node[below] at (0,-0.5) {$p_v=\infty$};
	\draw (1,3) circle [radius=0.15];
	\node[right] at (1.1,3) {$\tilde{D}_\infty$};
	\begin{scope}[on background layer]
	\fill[pattern=my north west lines](-2.5,0) to (-2.5,3.5) to (0,3.5) to (0,0);
	\end{scope}
	\end{tikzpicture}\end{center}
\begingroup

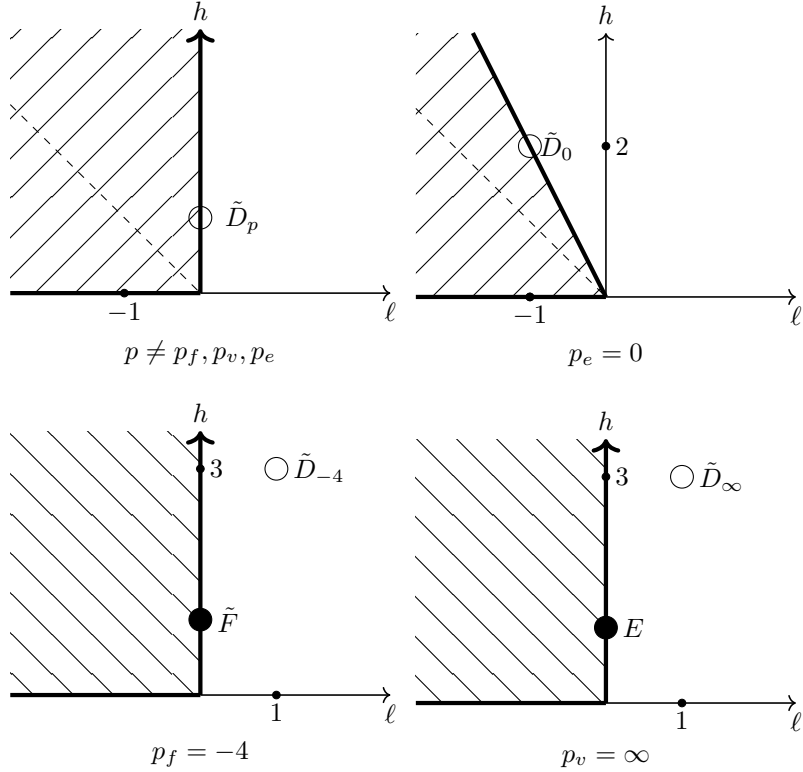
\captionof{figure}{Coloured hyperfan of the blow-up of $Q$ along $Z$}
\endgroup

\subsection{$V_5$ (1.15)}\label{115}

\subsubsection*{Description and Homogeneous Space}

Let $G=\sl{2}$ act on $\bbk^2$ with the standard linear action, and hence on $\proj{6} = \bb{P}(S^6\bbk^2)$. By \cite{furu}, we can realise $V_5$ as the closure in $\proj{6}$ of the $G$-orbit of the point $P = [0:1:0:0:0:-1:0]$. This explicitly shows that $V_5$ is a quasihomogeneous $G$-variety of complexity one. The stabiliser of $P$ in $G$ contains the matrices $y, r$ and $\omega$ which, in \Cref{cubic}, were shown to generate the binary cubic group $H = \tilde{C}$. Since no other finite subgroup of $G$ contains $\tilde{C}$, we see that $G\cdot P \cong G/H$. Hence there are three subregular semi-invariants $f_v$, $f_e$ and $f_f$ of respective $(B \times H)$ biweights $(8\alpha,1), (12\alpha,-1)$ and $(6\alpha,-1)$ and multiplicities $3, 2$ and $4$. 

The description of $V_5$ in \cite{furu} realises it as the subvariety of $\proj{6}$ defined by the equations \begin{align*} x_0x_4-4x_1x_3+3x_2^2 & = 0, \\ x_0x_5-3x_1x_4+2x_2x_3 & = 0, \\ x_0x_6-9x_2x_4+8x_3^2 & = 0, \\ x_1x_6-3x_2x_5+2x_3x_4 & = 0, \\ x_2x_6-4x_3x_5+3x_4^2 & = 0.\end{align*} It is clear from the above then that the rational normal curve $Z$ of degree 6 defined as the image of the Veronese embedding $\nu_6 \colon \proj{1} \to \proj{6}$ lies inside $V_5$ as a minimal $G$-germ. We will show later that $Z$ is in fact the unique minimal $G$-germ of $V_5$, contained in a unique $G$-stable divisor.

\subsubsection*{(Semi)-Invariant Functions}

One can find semi-invariants of the correct $B$-weights by using a torus action, under which each co-ordinate function has a given weight. In this case $x_k$ has weight $-6+2k$. It is easy to see immediately that $x_6$ is $B$ semi-invariant of weight 6, so must be the subregular semi-invariant $f_f$. Likewise $x_4x_6-x_5^2$ is $B$ semi-invariant of weight 8, so represents $f_v$. Now from \Cref{cubic}, we know that $f_vf_f/f_e$ is a rational function, so has degree 0. Hence $f_e$ must have degree 3 and $B$- (hence $T$-) weight $12\alpha$, and therefore must be a linear combination of the monomials $x_3x_6^2, x_4x_5x_6$ and $x_5^3$. A simple check shows that $f_e = x_3x_6^2-3x_4x_5x_6+2x_5^3$ suffices.

Hence $f_v^3/f_e^2$ is an invariant rational function defining a rational $B$-quotient $\pi \colon V_5 \dashrightarrow \proj{1}$, $P \mapsto [f_v^3(P):f_e^2(P)]$. The pullback of $p = [\alpha:\beta] \in \proj{1}$ defines a regular colour for all $p$ except the following: \[p = 0: \pi^*(p) = \cal{Z}(f_e^2) = D_0;\]\[p = \infty: \pi^*(p) = \cal{Z}(f_v^3) = D_\infty;\]\[p=-4: \pi^*(p) = \cal{Z}(x_6^4) \union \cal{Z}(x_1x_5+3x_3^2-4x_2x_3) = D_{-4} \union F;\] where $F$ is then a unique $G$-divisor containing the $G$-germ $Z$. Likewise it is straightforward that $Z$ is contained in every colour except $D_{-4}$. 

We choose as mentioned the semi-invariant splitting function $e_{2\alpha} = f_vf_f/f_e$ of $B$-weight $2\alpha$, and the colours $D_0, D_\infty$ and $D_{-4}$ are distinguished by it.

\subsubsection*{Coloured Hyperfan}

We already know from previous discussions and our choice of splitting that regular colours $D_p$ for $p \neq 0, \infty -4$ are mapped to points $(p,0,1)$ of hyperspace, and that $D_0 \mapsto (0,-1,2)$, $D_\infty \mapsto (\infty,1,3)$ and $D_{-4} \mapsto (-4,1,4)$. Finally, it is easy to check on the $B$-chart $V_5\setminus{\cal{Z}(x_6)}$ that $F$ is mapped to $(-4,0,1)$.

The coloured data of the minimal $G$-germ $Z$ is $\cal{V}_Z = \{\nu_F\}$, $\cal{D}^B_Z = \{D_p \mid p \neq -4\}$. Hence $Z$ is a $G$-germ of type II and defines a supported coloured hypercone of type II generated by its coloured data and the polytope $\cal{P} = \cal{P}_0 + \cal{P}_\infty + \cal{P}_{-4} = \{-1/2\} + \{1/3\} + \{0\} = \{-1/6\}$. Hence the coloured hyperfan of $V_5$ is as follows:

\begin{center}\begin{tikzpicture}
	\draw[ultra thick] (-2.5,0) -- (0,0);
	\draw[semithick][->] (0,0) -- (2.5,0);
	\draw[semithick][->] (0,0) -- (0,3.5);
	\draw[ultra thick] (0,0) -- (-1.75,3.5);
	\draw[dashed,domain = -2.5:0] plot (\x, {-\x}); 
	\draw[fill=black] (0,2) circle [color=black,radius=0.05];
	\node[right] at (0,2) {$2$};
	\draw[fill=black] (-1,0) circle [color=black,radius=0.05];
	\node[below] at (-1,0) {$-1$};
	\draw[fill=black] (-1,0) circle [color=black,radius=0.05];
	\node[below] at (2.5,0) {$\ell$};
	\node[above] at (0,3.5) {$h$};
	\node[below] at (0,-0.5) {$0 = p_e$};
	\draw (-1,2) circle [radius=0.15];
	\node[right] at (-0.9,2) {$D_0$};
	\begin{scope}[on background layer]
	\fill[pattern=my north east lines](-2.5,0) to (-2.5,3.5) to (-1.75,3.5) to (0,0);
	\end{scope}
	\end{tikzpicture}
	\begin{tikzpicture}
	\draw[ultra thick] (-2.5,0) -- (0,0);
	\draw[semithick][->] (0,0) -- (2.5,0);
	\draw[ultra thick][->] (0,0) -- (0,3.5);
	\node[below] at (2.5,0) {$\ell$};
	\node[above] at (0,3.5) {$h$};
	\draw[fill=black] (1,0) circle [color=black,radius=0.05];
	\node[below] at (1,0) {$1$};
	\draw[fill=black] (0,3) circle [color=black,radius=0.05];
	\node[right] at (0,3) {$4$};
	\node[below] at (0,-0.5) {$-4 = p_f$};
	\draw[fill=black] (0,1) circle [color=black,radius=0.15];
	\node[right] at (0.1,1) {$F$};
	\draw (1,3) circle [radius=0.15];
	\node[right] at (1.1,3) {$D_{-4}$};
	\begin{scope}[on background layer]
	\fill[pattern=my north west lines](-2.5,0) to (-2.5,3.5) to (0,3.5) to (0,0);
	\end{scope}
	\end{tikzpicture}\end{center}

\begin{center}
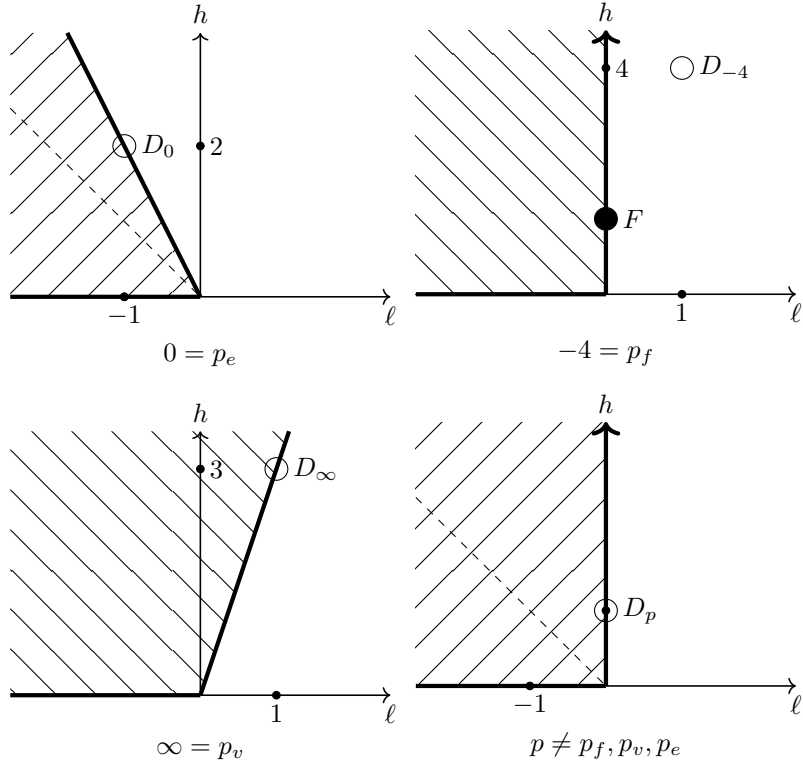
\begin{tikzpicture}
	\draw[ultra thick] (-2.5,0) -- (0,0);
	\draw[semithick][->] (0,0) -- (2.5,0);
	\draw[semithick][->] (0,0) -- (0,3.5);
	\draw[ultra thick] (0,0) -- (7/6,3.5);	
	\node[below] at (2.5,0) {$\ell$};
	\node[above] at (0,3.5) {$h$};
	\draw[fill=black] (1,0) circle [color=black,radius=0.05];
	\node[below] at (1,0) {$1$};
	\draw[fill=black] (0,3) circle [color=black,radius=0.05];
	\node[right] at (0,3) {$3$};
	\node[below] at (0,-0.5) {$\infty = p_v$};
	\draw (1,3) circle [radius=0.15];
	\node[right] at (1.1,3) {$D_\infty$};
	\begin{scope}[on background layer]
	\fill[pattern=my north west lines](-2.5,0) to (-2.5,3.5) to (7/6,3.5) to (0,0);
	\end{scope}
	\end{tikzpicture}
	\begin{tikzpicture}
	\draw[ultra thick] (-2.5,0) -- (0,0);
	\draw[semithick][->] (0,0) -- (2.5,0);
	\draw[ultra thick][->] (0,0) -- (0,3.5);
	\draw[dashed,domain = -2.5:0] plot (\x, {-\x}); 
	\node[below] at (2.5,0) {$\ell$};
	\node[above] at (0,3.5) {$h$};
	\draw[fill=black] (0,1) circle [color=black,radius=0.05];
	\draw[fill=black] (-1,0) circle [color=black,radius=0.05];
	\node[below] at (-1,0) {$-1$};
	\node[below] at (0,-0.5) {$p \neq p_f,p_v,p_e$};
	\draw (0,1) circle [radius=0.15];
	\node[right] at (0.1,1) {$D_p$};	
	\begin{scope}[on background layer]
	\fill[pattern=my north east lines](-2.5,0) to (-2.5,3.5) to (0,3.5) to (0,0);
	\end{scope}
	\end{tikzpicture}\end{center}
\begingroup
\captionof{figure}{Coloured hyperfan of $V_5$}
\endgroup

\subsection{$V_{22}$ (1.10)}\label{110}

\subsubsection*{Description and Homogeneous Space}

Let $G=\sl{2}$ act on $\bbk^2$ with the standard linear action, and hence on $\proj{12} = \bb{P}(S^{12}\bbk^2)$. By \cite{furu}, we can realise $V_{22}$ as the closure in $\proj{12}$ of the $G$-orbit of the point \[P = [0:1:0:0:0:0:11:0:0:0:0:1:0].\] This explicitly shows that $V_{22}$ is a quasihomogeneous $G$-variety of complexity one. The stabiliser of $P$ in $G$ contains $\begin{psmallmatrix}\epsilon & 0 \\ 0 & \epsilon^{-1}\end{psmallmatrix}$ for $\epsilon$ a primitive tenth root of unity. Furthermore, the function $x_{12}$ on $V_{22}$ is a $B$-eigenfunction of weight $12\alpha$. The only finite subgroup $H$ of $\sl{2}$ containing an element of order 10 and such that $\sl{2}/H$ has a semi-invariant of weight $12\alpha$ is $\tilde{I}$, the binary icosahedral group. Hence $V_{22}$ is an embedding of $\sl{2}/\tilde{I}$.

Then by \ref{icosahedral} there are three subregular semi-invariants $f_v$, $f_e$ and $f_f$ of respective $B$-weights $12\alpha, 30\alpha$ and $20\alpha$ and multiplicities $5, 2$ and $3$. Note that $\tilde{I}$ has no weights.

The description of $V_{22}$ in \cite{furu} realises it as the subvariety of $\proj{12}$ defined by the equations \[\sum_{\lambda = 0}^{\rho}{{8 \choose \lambda}{8\choose \rho-\lambda}(x_{\lambda}x_{\rho+4-\lambda} -4x_{\lambda + 1}x_{\rho+3-\lambda} + 3x_{\lambda + 2}x_{\rho+2-\lambda})} = 0\] for $0 \leq \rho \leq 16$. It is easy to check using these equations that the rational normal curve $Z$ of degree 12, defined as the image of the Veronese embedding $\nu_{12} \colon \proj{1} \to \proj{12}$, lies inside $V_{22}$ as a minimal $G$-germ. We will see that $Z$ is in fact the unique minimal $G$-germ of $V_{22}$, contained in a unique $G$-stable divisor.

\subsubsection*{(Semi)-Invariant Functions}

One can find semi-invariants of the correct $B$-weights by using a torus action, under which each co-ordinate function has a given weight. In this case $x_k$ has weight $(-12+2k)\alpha$. We have mentioned that $x_{12}$ is $B$ semi-invariant of weight $12\alpha$, so must be the subregular semi-invariant $f_v$. Likewise $x_{10}x_{12}-x_{11}^2$ is $B$ semi-invariant of weight $20\alpha$, so represents $f_f$. Now from \Cref{icosahedral}, we know that $f_vf_f/f_e$ is a rational function, so has degree 0. Hence $f_e$ must have degree 3 and $B$- (hence $T$-) weight $30\alpha$, and therefore must be a linear combination of the monomials $x_{10}x_{11}x_{12}, x_{11}^3$ and $x_9x_{12}^2$. A simple check shows that $f_e = 3x_{10}x_{11}x_{12} - 2x_{11}^3 - x_9x_{12}^2$ suffices.

Hence $f_f^3/f_e^2$ is an invariant rational function defining a rational $B$-quotient $\pi \colon V_{22} \dashrightarrow \proj{1}$, $P \mapsto [f_f^3(P):f_e^2(P)]$. The pullback of $p = [\alpha:\beta] \in \proj{1}$ defines a regular colour for all $p$ except the following: \[p = 0: \pi^*(p) = \cal{Z}(f_e^2) = D_0;\]\[p = \infty: \pi^*(p) = \cal{Z}(f_v^3) = D_\infty.\] We also see that the subregular colour $\cal{Z}(f_v)$ lies over $p = -4 = [1:-4]$, so we write $D_{-4} = \cal{Z}(f_v)$, recalling that this colour has multiplicity 5.

We choose as mentioned the semi-invariant splitting function $e_{2\alpha} = f_vf_f/f_e$ of $B$-weight $2\alpha$, and the colours $D_0, D_\infty$ and $D_{-4}$ are distinguished by it.

\subsubsection*{Coloured Hyperfan}

We already know from previous discussions and our choice of splitting that regular colours $D_p$ for $p \neq 0, \infty -4$ are mapped to points $(p,0,1)$ of hyperspace, and that $D_0 \mapsto (0,-1,2)$, $D_\infty \mapsto (\infty,1,3)$ and $D_{-4} \mapsto (-4,1,5)$.

Consider the coloured data of the minimal $G$-germ $Z$: we know that $\cal{D}^B_Z = \{D_p \mid p \neq -4\}$. This is infinite, so $Z$ is a $G$-germ of type II and defines a coloured hypercone generated by its coloured data and the polytope $\cal{P}$. However, the slice $\cal{H}_{-4}$ does not contain a non-central element of $\cal{D}^B_Z$, and hence must contain a non-central element of $\cal{V}_Z$ since $Z$ is of type II. Hence there must be a $G$-divisor $F \sub V_{22}$ containing $Z$ and lying over $-4 \in \proj{1}$. In the slice $\cal{H}_{-4}$, $\cal{C}_Z$ is then generated by  $\cal{P}$, which is central, and $F$, and by completeness it must cover the valuation cone. Hence we must have $F \mapsto (-4,0,h)$ for some $h \geq 0$, and the hyperfan is the same for any such $h$, so take $h=1$. Then $\cal{P} = \cal{P}_0 + \cal{P}_\infty + \cal{P}_{-4} = \{-1/2\} + \{1/3\} + \{0\} = \{-1/6\}$. Hence the coloured hyperfan of $V_{22}$ is as follows:

\begin{center}\begin{tikzpicture}
	\draw[ultra thick] (-2.5,0) -- (0,0);
	\draw[semithick][->] (0,0) -- (2.5,0);
	\draw[semithick][->] (0,0) -- (0,3.5);
	\draw[ultra thick] (0,0) -- (-1.75,3.5);
	\draw[dashed,domain = -2.5:0] plot (\x, {-\x}); 
	\draw[fill=black] (0,2) circle [color=black,radius=0.05];
	\node[right] at (0,2) {$2$};
	\draw[fill=black] (-1,0) circle [color=black,radius=0.05];
	\node[below] at (-1,0) {$-1$};
	\draw[fill=black] (-1,0) circle [color=black,radius=0.05];
	\node[below] at (2.5,0) {$\ell$};
	\node[above] at (0,3.5) {$h$};
	\node[below] at (0,-0.5) {$0 = p_e$};
	\draw (-1,2) circle [radius=0.15];
	\node[right] at (-0.9,2) {$D_0$};
	\begin{scope}[on background layer]
	\fill[pattern=my north east lines](-2.5,0) to (-2.5,3.5) to (-1.75,3.5) to (0,0);
	\end{scope}
	\end{tikzpicture}
	\begin{tikzpicture}
	\draw[ultra thick] (-2.5,0) -- (0,0);
	\draw[semithick][->] (0,0) -- (2.5,0);
	\draw[ultra thick][->] (0,0) -- (0,3.5);
	\node[below] at (2.5,0) {$\ell$};
	\node[above] at (0,3.5) {$h$};
	\draw[fill=black] (1,0) circle [color=black,radius=0.05];
	\node[below] at (1,0) {$1$};
	\draw[fill=black] (0,3) circle [color=black,radius=0.05];
	\node[right] at (0,3) {$5$};
	\node[below] at (0,-0.5) {$-4 = p_v$};
	\draw[fill=black] (0,1) circle [color=black,radius=0.15];
	\node[right] at (0.1,1) {$F$};
	\draw (1,3) circle [radius=0.15];
	\node[right] at (1.1,3) {$D_{-4}$};
	\begin{scope}[on background layer]
	\fill[pattern=my north west lines](-2.5,0) to (-2.5,3.5) to (0,3.5) to (0,0);
	\end{scope}
	\end{tikzpicture}\end{center}

\begin{center}
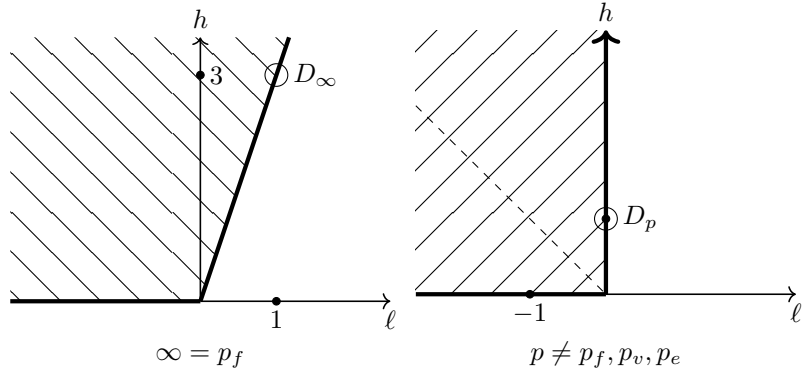
\begin{tikzpicture}
	\draw[ultra thick] (-2.5,0) -- (0,0);
	\draw[semithick][->] (0,0) -- (2.5,0);
	\draw[semithick][->] (0,0) -- (0,3.5);
	\draw[ultra thick] (0,0) -- (7/6,3.5);	
	\node[below] at (2.5,0) {$\ell$};
	\node[above] at (0,3.5) {$h$};
	\draw[fill=black] (1,0) circle [color=black,radius=0.05];
	\node[below] at (1,0) {$1$};
	\draw[fill=black] (0,3) circle [color=black,radius=0.05];
	\node[right] at (0,3) {$3$};
	\node[below] at (0,-0.5) {$\infty = p_f$};
	\draw (1,3) circle [radius=0.15];
	\node[right] at (1.1,3) {$D_\infty$};
	\begin{scope}[on background layer]
	\fill[pattern=my north west lines](-2.5,0) to (-2.5,3.5) to (7/6,3.5) to (0,0);
	\end{scope}
	\end{tikzpicture}
	\begin{tikzpicture}
	\draw[ultra thick] (-2.5,0) -- (0,0);
	\draw[semithick][->] (0,0) -- (2.5,0);
	\draw[ultra thick][->] (0,0) -- (0,3.5);
	\draw[dashed,domain = -2.5:0] plot (\x, {-\x}); 
	\node[below] at (2.5,0) {$\ell$};
	\node[above] at (0,3.5) {$h$};
	\draw[fill=black] (0,1) circle [color=black,radius=0.05];
	\draw[fill=black] (-1,0) circle [color=black,radius=0.05];
	\node[below] at (-1,0) {$-1$};
	\node[below] at (0,-0.5) {$p \neq p_f,p_v,p_e$};
	\draw (0,1) circle [radius=0.15];
	\node[right] at (0.1,1) {$D_p$};	
	\begin{scope}[on background layer]
	\fill[pattern=my north east lines](-2.5,0) to (-2.5,3.5) to (0,3.5) to (0,0);
	\end{scope}
	\end{tikzpicture}\end{center}
\begingroup
\captionof{figure}{Coloured hyperfan of $V_{22}$}
\endgroup

\section{Applying \thref{Kstablecentral}}\label{apply}
	
In this section we show that \thref{Kstablecentral} applies to the varieties whose combinatorial data we have just calculated.

\subsection{Action with no Fixed Points}

The first case of \thref{Kstablecentral} applies to $\proj{3}$ (1.17), the blow-up of $\proj{3}$ in two lines (3.25) and the blow-up of $\proj{3}$ in three lines (4.6).

Realise $\proj{1}$ as the unit sphere and consider the action of the symmetric group $A = S_3$ on $\proj{1}$ consisting of rotations permuting the vertices of an equilateral triangle inscribed in the equator. This action has no fixed points, one orbit of order 2 consisting of the north and south poles, and various other orbits of orders 3 and 6.

If $X = \proj{3}$, then $S_3 \sub \aut{X} = \pgl{4}$. Recall that the slices of the coloured hyperfan are all identical to each other except for that of the distinguished point, which lies two units to the left on the $\ell$-axis relative to the other slices. Hence we can shift the slice of the distinguished point $p_0$ one unit to the right and shift the slice of some other point $p_1$ one unit to the left. Now $\cal{H}_{p_0}$ and $\cal{H}_{p_1}$ look the same, so we identify these two points with the north and south poles of the sphere. It follows that the $A$-action preserves the coloured hyperfan of $X$, in the sense that for $a \in A$ and $p \in \proj{1}$, we permute the slices of $\cal{H}$ via $\cal{H}_{p} \to \cal{H}_{a\cdot p}$ and the coloured data within each slice are invariant with respect to these permutations. This means that the action of $A$ on $X$ is such that the $B$-quotient is $A$-invariant and \thref{Kstablecentral} applies.

If $X$ is the blow-up of $\proj{3}$ along two lines, say $Y_q$ and $Y_r$, then let $q$ and $r$ be the north and south poles of the sphere $\proj{1}$. Shift their slices of the hyperspace two units to the left each, then choose two other non-distinguished points $p_1$ and $p_2$ and shift their slices two units to the right each. Then the slices corresponding to $p_1$, $p_2$ and the distinguished point $p_0$ align, and we can choose these to be the vertices of the equilateral triangle acted on by $S_3$. Again, $S_3$ acts on $X$ and preserves the coloured hyperfan up to balanced integral shifts, so \thref{Kstablecentral} applies.

If $X$ is the blow-up of $\proj{3}$ along three lines, $Y_q, Y_r$ and $Y_s$, the method is the same as when $X = \proj{3}$, only making sure to choose $q, r$ and $s$ to be the vertices of the triangle. Again, \thref{Kstablecentral} applies.

As an aside, it is worth mentioning that the blow-up of $\proj{3}$ along \emph{one} line (2.33) has non-reductive automorphism group and is thus not $K$-polystable. It is easy to see that the method above does not work in this case, since the slice of hyperspace corresponding to the blown-up line is fundamentally distinct from all other slices for this variety, so no action on $\proj{1}$ which does not fix the corresponding point could ever preserve the coloured hyperfan.

\subsection{Action Interchanging Two Points}

The second case of \thref{Kstablecentral} applies to the blow-up of $\proj{1} \times \proj{2}$ along a curve of bidegree $(1,1)$ (3.17).

Consider the $\bb{Z}_2$-symmetry of $\proj{1}$ given by $[\alpha : \beta] \mapsto [\beta:\alpha]$. This interchanges $0$ and $\infty$, fixes $1$ and $-1$, and puts every other point in an orbit of order 2. Recall that the blow-up of $\proj{1} \times \proj{2}$ has subregular colours lying over 0 and $\infty$, and has a distinguished point $-1$. Since the slices of the coloured hyperfan over $0$ and $\infty$ are identical to each other, the interchange of these points by $\bb{Z}_2$ leaves the hyperfan invariant, hence the $\bb{Z}_2$ action on $X$ respects the $B$-quotient map and \thref{Kstablecentral} applies.

Oddly, this method seems not to apply to $\proj{1} \times \proj{2}$ itself, even though this variety is known to be $K$-polystable, since in this case, one of the two slices with subregular colours contains a $G$-divisor while the other does not, and the $\bb{Z}_2$-symmetry therefore does not extend. This may require some further exploration.

\subsection{Three or More Subregular Colours}

We have seen that any smooth Fano $\sl{2}$-threefold whose stabiliser subgroup $H$ is one of $\{\tilde{D}_m, \tilde{T}, \tilde{C}, \tilde{I}\}$ has 3 subregular colours, all lying over distinct points in $\proj{1}$. Then in particular, the third case of \thref{Kstablecentral} applies to $V_{22}$ (1.10), $V_5$ (1.15), $Q$ (1.16), the blow-up of $Q$ along a twisted quartic (2.21), the blow-up of $\proj{3}$ along a twisted cubic (2.27), $W$ (2.32) and the blow-up of $W$ along a curve of bidegree $(2,2)$ (3.13).

It remains to verify that each of these varieties satisfies the hypotheses of \thref{nonint}, which we now do one by one.

\subsubsection{$V_{22}$ (1.10)}

It is known that for $X = V_{22}$, the Mukai-Umemura threefold, we can take $-K_X$ to be a hyperplane section. The $B$-invariant hyperplane section of $X$ for our action is $D_{-4}$. This variety has subregular colours lying over $0, \infty, -4 \in \proj{1}$. With this representative of $-K_X$, we have \[H_\lambda = -\frac{\lambda}{2}[0] + \frac{\lambda}{3}[\infty] + \min\left\{\frac{\lambda+1}{5},0\right\}[-4].\] Choosing $\lambda = -5$ gives $\deg{H_\lambda} = \frac{1}{30}$ and non-integral coefficients at all three points.

\subsubsection{$V_5$ (1.15)}

This time $-K_X = 2D_{-4}$. We have \[H_\lambda = -\frac{\lambda}{2}[0]+\frac{\lambda}{3}[\infty] + \min\left\{\frac{\lambda+2}{4}, 0\right\}[-4].\] Again, $\lambda = -5$ works, giving $\deg{H_\lambda} = \frac{1}{12}$ and non-integral coefficients at all three points.

\subsubsection{$Q$ (1.16)}

$Q$ is a hypersurface of degree 2 in $\proj{4}$ so its anticanonical divisor is given by its intersection with a divisor of degree 3 in $\proj{4}$. The three subregular colours $D_0$, $D_\infty$ and $D_{-4}$ of $Q$ are sections of prime divisors in $\proj{4}$ of degrees 3, 2 and 1, respectively, so we may take $-K_Q = 3D_{-4}$. Then \[H_\lambda = \frac{-\lambda}{2}[0] + \frac{\lambda}{3}[\infty] + \min\left\{\frac{\lambda+3}{3},0\right\}[-4].\] Choosing $\lambda = -5$ gives non-integral coefficients at each point and $\deg{H_\lambda} = \frac{1}{6}$.

\subsubsection{Blow-up of $Q$ (2.21)}

With the same $-K_Q$ as before, blowing up in the twisted quartic, which is contained in $D_0$ and $D_\infty$ but not in $D_{-4}$, gives $-K_X = 3\tilde{D}_{-4}-E$. Adding $\Div(f_0e_{2\alpha}^{-2})$, where $f_0 \in \bb{C}(\proj{1})$ has divisor $[\infty]-[0]$ gives $-K_X = \tilde{D}_{-4} + \tilde{D}_{\infty}$. Then \[H_\lambda = \frac{-\lambda+2}{2}[0] + \min\left\{\frac{\lambda+1}{3},0\right\}[\infty] + \min\left\{\frac{\lambda+1}{3},0\right\}[-4].\] Taking $\lambda = -3$ gives non-integral coefficients and $\deg{H_\lambda} = \frac{1}{6}$.

\subsubsection{Blow-up of $\proj{3}$ along a twisted cubic (2.27)}

$\proj{3}$ with the cubic $\sl{2}$-action has subregular colours $D_0$, $D_\infty$ and $D_{-4}$ of respective degrees 3, 2 and 1. Hence we may choose $-K_{\proj{3}} = 4D_{-4}$. Then blowing up the twisted cubic gives $-K_X = 4\tilde{D}_{-4} - E$, where $E$ is the exceptional divisor and lies over $\infty$. We add $\Div(f_0e_{2\alpha}^{-2})$ to $-K_X$, where $\Div(f_0) = [0]-[\infty]$, giving $-K_X = 2\tilde{D}_{-4}+\tilde{D}_\infty$. We then have \[H_\lambda = \frac{\lambda}{2}[0] + \min\left\{\frac{1-2\lambda}{3},-\lambda\right\}[\infty] + \min\left\{\frac{\lambda+2}{2},0\right\}[-4].\] Choosing $\lambda = -3$ gives non-integral coefficients at all points and $\deg{H_\lambda} = \frac{1}{3}$.

\subsubsection{$W$ (2.32)}

The divisor $W$ on $\proj{2} \times \proj{2}$ of bidegree $(1,1)$ has anticanonical divisor class given by the intersection with $W$ of a class of bidegree $(2,2)$ on $\proj{2} \times \proj{2}$. The subregular colours $D_0$ and $D_\infty$ have bidegrees $(1,0)$ and $(0,1)$, so we take $-K_W = 2D_0 + 2D_\infty$. This gives \[H_\lambda = \min\left\{\frac{\lambda+2}{2},0\right\}[0] + \min\left\{\frac{\lambda+2}{2},0\right\}[\infty] - \frac{\lambda}{2}[-1].\] Taking $\lambda = -3$ gives non-integral coefficients at all points and $\deg{H_\lambda} = \frac{1}{2}$.

\subsubsection{$\proj{1} \times \proj{2}$ (2.34)}

The anticanonical divisor of $\proj{1} \times \proj{2}$ has bidegree $(2,3)$ which we can obtain as $2D_\infty + 3D_{-1}$. This gives \[H_\lambda = -\frac{\lambda}{2}[0] + \min\left\{\frac{2-\lambda}{2},-\lambda\right\}[\infty] + \min\left\{2\lambda+3,\lambda\right\}[-1].\] Taking $\lambda = -3$ gives non-integral coefficients at the points $0, \infty$ corresponding to subregular colours and $\deg{H_\lambda} = 1$.

\subsubsection{Blow-up of $W$ (3.13)}

We blow up $W$ in the curve of bidegree $(2,2)$ obtaining $-K_X = 2\tilde{D}_0 + 2\tilde{D}_\infty - E$. Adding $\Div(e_{2\alpha}^{-1})$ gives $-K_X = \tilde{D}_0 + \tilde{D}_\infty + \tilde{D}_{-1}$. Then \[H_\lambda = \min\left\{\frac{\lambda+1}{2},0\right\}[0] + \min\left\{\frac{\lambda+1}{2},0\right\}[\infty] + \min\left\{\frac{1-\lambda}{2},-\lambda\right\}[-1].\] Taking $\lambda = -2$ gives non-integral coefficients at all points and $\deg{H_\lambda} = \frac{1}{2}$.

\subsubsection{Blow-up of $\proj{1} \times \proj{2}$ (3.17)}

Blowing up $\proj{1} \times \proj{2}$ along the curve of bidegree $(1,1)$ gives $-K_X = 2\tilde{D}_\infty + 3\tilde{D}_{-1} - E$. We can add $\Div(f_0e_{2\alpha}^{-1})$, where $\Div(f_0) = [-1]-[\infty]$ to get $-K_X = \tilde{D}_\infty + 2\tilde{D}_{-1} + \tilde{D}_0$. Then \[H_\lambda = \min\left\{\frac{1-\lambda}{2},-\lambda\right\}[0] +  \min\left\{\frac{1-\lambda}{2},-\lambda\right\}[\infty] + \min\left\{2\lambda+2,\lambda\right\}[-1].\] Choosing $\lambda = -2$ gives non-integral coefficients at the points $0, \infty$ corresponding to subregular colours and $\deg{H_\lambda} = 1$.

\section{Central Divisors of $\sl{2}$-Threefolds}

Here we prove \thref{threefoldsKstable} by calculating the $\beta$-invariant for the $G$-stable central prime divisors over each of our list of $\sl{2}$-threefolds. By \thref{Kstablecentral}, we need to show that (1): $\beta_X(F) \geq 0$ for central divisors $F$ over each variety and (2): if $\beta_X(F) = 0$ then $F$ corresponds to a product configuration. It will turn out that (2) is never the case. Throughout we will mostly use the notation and results of \Cref{combo} for curves, divisors etc. lying in each variety, with any changes to this notation clearly signalled.

\subsection{Existence and Uniqueness of Central Divisors}

In this subsection we will prove the following:

\begin{theorem}\thlabel{uniquecentral}
	Let $X$ be a smooth Fano $\sl{2}$-threefold. There exists a unique central $G$-divisor over $X$. If $X$ is of type I, then it contains this central divisor. If $X$ is of type II, the central $G$-divisor over $X$ lies on the type I variety over $X$ whose existence is proved in \thref{typeIbirational}.
\end{theorem}

\begin{definition}
	Let $\cal{H}$ be the hyperspace of a $G$-model $X$ of complexity one. The \emph{dimension} of $\cal{H}$ is the common dimension of each half-space $\cal{H}_p$ for $p \in \proj{1}$. If $\cal{C}_p$ is a coloured cone in $\cal{H}_p$, its dimension is the dimension of a minimal affine subspace of $\cal{H}_p$ containing $\cal{C}_p$. If $\cal{C}$ is a coloured hypercone of type II in $\cal{H}$, set $\dim{\cal{C}} = \max_{p \in \proj{1}}{\dim{(\cal{C} \inter \cal{H}_p)}}$.
\end{definition}

\begin{lemma}
	Let $X$ be a $G$-model of complexity one and rank $r$, and let $\cal{H}$ be the hyperspace of $X$. We have $\dim{\cal{H}} = 1 + r$.
\end{lemma}

\begin{proof}
	For each $p \in \proj{1}$, the slice $\cal{H}_p$ of hyperspace corresponding to $p$ is isomorphic to $\Lambda_\bb{Q}^* \times \bb{Q}_{\geq 0}$, where $\Lambda$ is the weight lattice of $X$, which has dimension $1 + \dim{\Lambda_\bb{Q}^*} = 1 + \rk{\Lambda} = 1+r$, by the definition of rank.
\end{proof}

\begin{proposition}
	Let $X$ be a $G$-model and let $Y \sub X$ be a $G$-germ. The dimension of $\cal{C}_Y$ in $\cal{H}$ is equal to the codimension of $Y$ in $X$.
\end{proposition}

\begin{proof}
	This follows from the fact that the coloured hypercone corresponding to $X$ itself is $\{0\}$ and that inclusion of coloured hypercones as faces in one another corresponds to the reverse inclusion of the associated $G$-germs.
\end{proof}

\begin{corollary}
	Let $X$ be a complete $G$-model of dimension $d$, rank $r$ and complexity 1. Then minimal $G$-germs in $X$ must have dimension $d-r-1$.
\end{corollary}

\begin{proof}
	Since $\cal{V}$ is full dimensional in $\cal{H}$, it follows from completeness that the coloured (hyper)cone corresponding to a minimal $G$-germ $Y$ in $X$ must have the same dimension as $\cal{H}$, i.e. $1 + r$. Then $\codim_X{Y} = \dim{\cal{C}_Y} = 1 + r$, so $\dim{Y} = d - r - 1$. 
\end{proof}

Now let $G = \sl{2}(\bbk)$ and let $X$ be a complete three dimensional $G$-model of complexity one, i.e. a normal projective threefold with a $G$-action having finite stabilisers. Then $X$ contains an open orbit isomorphic to $G/H$ for $H$ a finite subgroup of $G$.

Let $B$ be the Borel subgroup of $G$ given by the upper triangular matrices. Then $\frak{X}(B)$ is of rank 1, generated by the character $\alpha$ which picks out the upper-left entry. Hence in particular $X$ is a rank 1 variety. The hyperspace $\cal{H}$ of $X$ then has dimension 2, so minimal $G$-germs of $X$ must have codimension 2, i.e. they are curves, and in particular $X$ can contain no $G$-fixed points. The centre $\cal{Z}$ of $\cal{H}$ is a line. Most of the following results arise from this fact. In particular, note that for any finite $H \sub G$, the valuation cone $\cal{V}(G/H)$ intersects $\cal{Z}$ in a ray. 

\begin{proposition}
	$X$ contains at most one central $G$-divisor.	
\end{proposition}

\begin{proof}
	Any central $G$-divisor must be mapped to the intersection $\cal{V} \inter \cal{Z}$ of the valuation cone and the central hyperplane. As noted, this intersection is a ray. If there were two distinct central $G$-divisors, their coloured (hyper)cones would then both be the same ray, a contradiction.
\end{proof}

\begin{proposition}
	$X$ contains a central $G$-divisor if and only if $X$ is a model of type I.
\end{proposition}

\begin{proof}
	Suppose $X$ is of type I, i.e. every $G$-germ of $X$ corresponds to a coloured cone in some slice of the hyperspace. By completeness, these coloured cones must cover $\cal{V}$. Since $\cal{V}$ intersects the central line $\cal{Z}$ in a ray $\rho$, there must be a $G$-germ of $X$ whose coloured cone is $\rho$, i.e. a central $G$-divisor.
	
	Now suppose $X$ has a central $G$-divisor $D$ and that $Y \sub X$ is a $G$-germ of type II with coloured hypercone $\cal{C}_Y$. We know that $\cal{C}_Y$ must intersect $\cal{Z}$ in $\cal{V}$, i.e. it contains the ray $\rho$ corresponding to $D$. Then $\nu_D$ lies in the relative interior of $\cal{C}_Y$, hence in the support $\cal{S}_Y$ of $Y$. But $\nu_D \in \cal{S}_D$, so the supports of $Y$ and $D$ are not disjoint, contradicting separation of $X$. Hence all $G$-germs of $X$ are of type I, as required.
\end{proof}

\begin{proposition}
	If $X$ is of type II, there exists a central prime divisor over $X$.
\end{proposition}

\begin{proof}
	We know from \thref{typeIbirational} that there exists a projective birational morphism $\nu \colon \check{X} \to X$ where $\check{X}$ is of type I. Then by the above proposition there is a central prime divisor $D \sub \check{X}$.
\end{proof}

\begin{proposition}
	Let $X$ be of type II. Then $X$ has a unique $G$-germ of type II, a curve, which must be a closed $G$-orbit.
\end{proposition}

\begin{proof}
	We know that $X$ contains at least one $G$-germ of type II. Suppose $Y, Y^\prime \sub X$ are both $G$-germs of type II. Their corresponding coloured hypercones $\cal{C}_Y$, $\cal{C}_{Y^\prime}$ must both intersect the central ray $\rho \sub \cal{V}$, hence their relative interiors must intersect in $\cal{V}$. It follows that $Y = Y^\prime$, and $X$ has exactly one $G$-germ of type II. 
	
	Then all other $G$-germs of $X$ are of type I and so each defines a coloured cone in some $\cal{H}_p$. In particular, these coloured cones cannot contain $\cal{C}_Y$, so $Y$ does not contain any other $G$-germ of $X$. Hence $Y$ is a minimal $G$-germ, so in particular a curve and a closed $G$-orbit. 
\end{proof}

\begin{proposition}
	Let $X$ be of type II and suppose that every $G$-divisor of $X$ maps to the boundary of the valuation cone. Then the unique minimal $G$-germ $Y \sub X$ of type II is contained in every $G$-divisor, and $X$ has no minimal $G$-germs of type I. In particular, $Y$ is the unique closed $G$-orbit of $X$.
\end{proposition}

\begin{proof}
	Let $F \sub X$ be a $G$-divisor mapping to the non-central boundary of $\cal{V}_p = \cal{V} \inter \cal{H}_p$. The ray $\cal{C}_F$ defined by $F$ must then be a face of some coloured cone in $\cal{H}_p$ or coloured hypercone of type II in $\cal{H}$, i.e. $F$ must contain some minimal $G$-germ.
	
	We know that $X$ contains a unique minimal $G$-germ $Y$ of type II, and in particular the coloured hypercone $\cal{C}_Y$ must have $\cal{C}_F$ as a face, i.e. $F$ contains $Y$. Now suppose $F$ contains another $G$-germ $Z$. This must be of type I since $Y$ is the only $G$-germ of type II, so $\cal{C}_Z$ must be a coloured cone in $\cal{H}_p$ with $\cal{C}_F$ as a face. However, $\cal{C}_Z$ must intersect $\cal{V}$ in its relative interior, and since $\cal{C}_F$ is the boundary of $\cal{V}$, it follows that $\cal{C}_Z$ and $\cal{C}_Y$ intersect in their relative interiors, a contradiction.
	
	Hence for each $\cal{H}_p$ containing a $G$-divisor, the only minimal $G$-germ whose coloured (hyper)cone intersects $\cal{H}_p$ is $Y$. If $\cal{H}_p$ does not contain a $G$-divisor, it can also only support one coloured (hyper)cone since it only contains one $B$-divisor, the colour $D_p$. 
	
	Therefore $Y$ is indeed the unique minimal $G$-germ of $X$ and hence also the unique closed orbit.
\end{proof}

In all cases which we will consider here, we are given a smooth $\sl{2}$-threefold $X$, and either $X$ is of type I and has a central divisor, which is unique, or $X$ is of type II and we can obtain the unique central divisor over $X$ by blowing up a finite sequence of $G$-invariant curves. 

Combining \thref{uniquecentral} with \thref{Kstablecentral}, we see that to prove \thref{threefoldsKstable}, we need only find the central divisor over each variety and show that its $\beta$-invariant is positive. In the following subsections we will perform this calculation for each of the examples, thus proving their $K$-polystability.

\subsection{$\proj{3}$ and Blow-Ups Along Two or Three Lines}

\subsubsection{$\proj{3}$ (1.17)}
The anticanonical divisor of $\proj{3} = \bb{P}(M_2(\bb{C}))$ is the class of a divisor of degree 4, which we may take to be $2\Delta$, where $\Delta$ is the $G$-invariant divisor of singular matrices. This is also the unique central divisor on $\proj{3}$.

To calculate $\beta_X(\Delta)$ when $X = \proj{3}$, note that $(-K_X)^3 = 64$ and $A_X(\Delta) = 1$ since $\Delta$ lies on $\proj{3}$. It remains to calculate $\vol(\delta)$, where $\delta = -K_X - x\Delta = (2-x)\Delta$. Note that since $\delta$ is $G$-invariant we have $\lambda_\delta = 0$.

We have $\cal{P}(\delta) = \{\lambda \in \Lambda \mid \langle \lambda,\ell_{\Delta}\rangle \geq -(2-x)\}$. Since $\Lambda = \bb{Z}\alpha \cong \bb{Z}$ and $\ell_\Delta = -1$, we get $\cal{P}(\delta) = \{\lambda \in \bb{Z} \mid \lambda \leq 2-x\}$. 

Consider \[A(\delta,\lambda) = \sum_{p \in \proj{1}}\left({\min_{p_D = p}{\frac{\langle \lambda,\ell_D\rangle + m_D}{h_D}}}\right).\] We have $\ell_D = m_D = 0$ for all colours $D$ other than the distinguished colour, which has $m_D = 0$, $\ell_D = 2$ and hence contributes a value of $2\lambda$ to $A(\delta,\lambda)$.

We therefore have $\cal{P}_+(\delta) = \{\lambda \leq 2-x \mid \lambda \geq 0\} = [0,2-x]$. Hence \[\vol(\delta) = 6\int_{0}^{2-x}{2\lambda\cdot 2\lambda \ \mathrm{d} \lambda} = 8(2-x)^3.\] We therefore have \[\beta(\Delta) = 64 - \int_{0}^{2}{8(2-x)^3 \ \mathrm{d}x} = 64-32 = 32 > 0.\] Hence by \thref{Kstablecentral}, $\proj{3}$ is $K$-polystable.

\subsubsection{Blow-up of $\proj{3}$ Along Two Lines (3.25)}

Starting with $-K_{\proj{3}} = 2\Delta$ as before, the anticanonical divisor after blowing up two lines $Y_q$ and $Y_r$ is $2\tilde{\Delta}+E_q+E_r$. 

This time we have $(-K_X)^3 = 44$ and again $A_X(\Delta) = 1$. We must now compute $\beta(\tilde{\Delta})$ by calculating $\vol(\delta)$ where $\delta = (2-x)\tilde{\Delta}+E_q+E_r$. We still have $\lambda_\delta = 0$.

Likewise, $\cal{P}(\delta) = (-\infty,2-x]$ as before. This time $A(\delta,\lambda)$ receives the same contribution of $2\lambda$ at the distinguished point, and there is no contribution other than from here and from $q$ and $r$. At $q$, we have the exceptional divisor $E_q$ with $\ell = -1$, $m = 1$ and $h = 1$, and the colour with $\ell = m = 0$. Thus there is a contribution to $A(\delta,\lambda)$ of $-\lambda + 1$ when this is less than or equal to 0, and a contribution of 0 otherwise. The same holds for $r$. Hence we have \[A(\delta,\lambda) = \begin{cases} 2 & 1 \leq \lambda \leq 2-x \\ 2\lambda & \lambda < 1.\end{cases}\] Therefore $\cal{P}_+(\delta) = [0,2-x]$, and \[\vol(\delta) = \begin{cases} 6\int_{0}^{1}{4\lambda^2 \ \mathrm{d}\lambda} + 6\int_{1}^{2-x}{4\lambda \ \mathrm{d}\lambda} & 0 \leq x \leq 1 \\ 6\int_{0}^{2-x}{4\lambda^2 \ \mathrm{d}\lambda} & 1 < x \leq 2.\end{cases}\] We thus get \[\beta_X(\tilde{\Delta}) = 44 - \int_{0}^{2}{\vol(\delta) \ \mathrm{d} x} = 44 - 26 = 18 > 0.\] Hence the blow-up of $\proj{3}$ at two lines is $K$-polystable.

\subsubsection{Blow up of $\proj{3}$ Along Three Lines (4.6)}

The anticanonical divisor of the blow-up of $\proj{3}$ along three lines $Y_q, Y_r$ and $Y_s$ is $2\tilde{\Delta} + E_q + E_r + E_s$. We have $(-K_X)^3 = 34$, $A_X(\tilde{\Delta}) = 1$, $\delta = (2-x)\tilde{\Delta}+E_q+E_r+E_s$, $\lambda_\delta = 0$ and $\cal{P}(\delta) = (-\infty,2-x]$. 

The calculation of $A(\delta,\lambda)$ goes much the same as in the previous case, only $A(\delta,\lambda)$ gains an extra contribution of $-\lambda + 1$ from $E_s$ when $\lambda \geq 1$, giving \[A(\delta,\lambda) = \begin{cases} 3-\lambda & 1 \leq \lambda \leq 2-x \\ 2\lambda & \lambda < 1.\end{cases}\] Therefore $\cal{P}_+(\delta) = [0,2-x]$, and \[\vol(\delta) = \begin{cases} 6\int_{0}^{1}{4\lambda^2 \ \mathrm{d}\lambda} + 6\int_{1}^{2-x}{6\lambda-2\lambda^2 \ \mathrm{d}\lambda} & 0 \leq x \leq 1 \\ 6\int_{0}^{2-x}{4\lambda^2 \ \mathrm{d}\lambda} & 1 < x \leq 2.\end{cases}\] We thus get \[\beta_X(\tilde{\Delta}) = 34 - \int_{0}^{2}{\vol(\delta) \ \mathrm{d} x} = 34 - 23 = 11 > 0.\] Hence the blow-up of $\proj{3}$ at three lines is $K$-polystable.

\subsection{Blow-up of $\proj{1} \times \proj{2}$}

\subsubsection{Central Divisor}

Let $X = \cal{Z}(x_0y_0z_2+x_1y_1z_0-x_0y_1z_1-x_1y_0z_1) \sub \proj{1} \times \proj{1} \times \proj{2}$. This variety is the blow up of $\proj{1} \times \proj{2}$ along the $G=\sl{2}$-stable curve $C = \cal{Z}(x_1z_0-x_0z_1,x_0y_1-x_1y_0)$. In $X$ there are $G$-invariant divisors $\Delta = \cal{Z}(x_0y_1-x_1y_0,x_0^2z_2+x_1^2z_0-2x_0x_1z_1)$, $E = \cal{Z}(x_1z_0-x_0z_1,x_0y_1-x_1y_0)$ and $F = \cal{Z}(y_1z_0-y_0z_1,y_0z_2-y_1z_1)$. The curve $Z = F \inter \Delta \inter E$ is $G$-invariant and defined by $\cal{Z}(x_1z_0-x_0z_1,x_1z_1-x_0z_2,x_0y_1-x_1y_0)$. 

Taking $x_1 = y_1 = z_2 = 1$ gives a maximal $B$-chart $U$ of $Z$, given by $U = \cal{Z}(x_0y_0+z_0-x_0z_1-y_0z_1) \sub \aff{4}$, so eliminating $z_0$ gives $U = \Spec{\bbk[x_0,y_0,z_1]} \cong \aff{3}$. We have $\Delta \inter U = \cal{Z}(x_0-y_0)$, $F \inter U = \cal{Z}(y_0-z_1)$, $E \inter U = \cal{Z}(z_1-x_0)$ and $Z \inter U = \cal{Z}(z_1-x_0,x_0-y_0)$. We blow up $U$ in this curve to obtain the variety $\tilde{X} = \cal{Z}(u_1(z_1-x_0)-u_0(x_0-y_0)) \sub \aff{3} \times \proj{1}$. 

The $B$-invariant rational function \[f = \frac{x_1^2(z_0z_2-z_1^2)}{(x_0z_2-x_1z_1)^2}\] on $X$ becomes, on $\tilde{X}$, $f = \frac{z_1-y_0}{x_0-z_1}$. From this one can see that the exceptional divisor $D = \cal{Z}(z_1-y_0,x_0-z_1)$ of the blow-up $\sigma \colon \tilde{X} \to X$ is central.

\subsubsection{$\beta_X(F)$ (3.17)}

We now want to calculate \[\beta(D) = A_X(D)(-K_X)^3 - \int_{0}^{\infty}{\vol_X{(-K_X-xD)} \ \mathrm{d}x}.\] We have $A_X(D) = 2$ since $D$ is the exceptional divisor on a blow-up of $X$, and $(-K_X)^3 = 36$, so \[\beta(F) = 72 - \int_{0}^{\infty}{\vol_X{(-K_X-xD)} \ \mathrm{d}x} = 72 - \int_{0}^{\infty}{\vol_{\tilde{X}}{(\sigma^*(-K_X)-xD)} \ \mathrm{d}x}.\]

To calculate $\sigma^*(-K_X)$, first let $\Delta = \cal{Z}(x_0^2z_2+x_1^2z_0-2x_0x_1z_1)$ and $F = \cal{Z}(z_0z_2-z_1^2)$ in $\proj{1} \times \proj{2}$. The divisors $\Delta$, $F$ above are the strict transforms of these under the blow-up $\mu \colon X \to \proj{1} \times \proj{2}$. Since $-K_{\proj{1} \times \proj{2}}$ is the class of a divisor of bidegree $(2,3)$, we can represent it by $\Delta + F = (2,1) + (0,2)$. 

Then $-K_X = \mu^*(\Delta + F) - E = (\Delta + E) + (F + E) - E = \Delta + F + E$. Hence $\sigma^*(-K_X) = \tilde{\Delta} + \tilde{F} + \tilde{E} + 3D$, and so we need to calculate the volume of $\delta = \tilde{\Delta} + \tilde{F} + \tilde{E} + (3-x)D$. This is given by \[\vol(\delta) = 6\int\limits_{\lambda_\delta + \cal{P}_+(\delta)}{2\lambda A(\delta,\lambda-\lambda_\delta)\, \mathrm{d}\lambda}.\]

Since $\delta$ is $G$-invariant, we have $\lambda_\delta = 0$. We also have \[\cal{P}(\delta) = \{\lambda \in \Lambda \otimes \bb{R} \mid \langle \lambda, \ell_D\rangle \geq x-3\} = \{\lambda \mid \lambda \leq 3-x\}\] and \[\cal{P}_+(\delta) = \{\lambda \in \cal{P}(\delta) \mid A(\delta,\lambda) \geq 0\} = \{\lambda \leq 3-x \mid A(\delta,\lambda) \geq 0\}\] where \[A(\delta,\lambda) = \sum_{p \in \proj{1}}{\min_{p_D=p}{\frac{\langle \lambda,\ell_D\rangle + m_D}{h_D}}}.\]

To calculate $A(\delta,\lambda)$, first note that for $p \neq 0,-1,\infty$, the only $B$-divisors with $p_D = p$ are the colours $D_p$ with $\ell_{D_p} = m_{D_p} = 0$, $h_{D_p} = 1$, so there is no contribution in these cases.

For $p = 0$, the two divisors with $p_D = p$ are $E$, with $\ell = -1$, $m = 1$ and $h = 1$, and $D_0$ with $\ell = -1$, $m = 0$ and $h=2$. Hence there is a contribution to $A(\delta,\lambda)$ of \[\min{\left\{1-\lambda,-\frac{\lambda}{2}\right\}} = \begin{cases} 1-\lambda & \lambda \geq 2 \\ -\frac{\lambda}{2} & \lambda < 2 \end{cases}.\] For $p = \infty$, the contribution is the same. The two divisors with $p_D = -1$ are $\Delta$ with $\ell = 1$, $m = 1$ and $h = 1$, and $D_{-1}$ with $\ell = -2$, $m = 0$ and $h = 1$, so the contribution to $A(\delta,\lambda)$ is \[\min{\left\{1+\lambda, 2\lambda\right\}} = \begin{cases} 1+\lambda & \lambda \geq 1 \\ 2\lambda & \lambda < 1 \end{cases}.\] Hence we have \[A(\delta,\lambda) = \begin{cases} 3-\lambda & \lambda \geq 2 \\ 1 & 1 \leq \lambda \leq 2 \\ \lambda & \lambda < 1\end{cases}.\] It follows that $\cal{P}_+(\delta) = \{\lambda \leq 3-x\mid 0 \leq \lambda \leq 3\}$, and since $x \geq 0$, $3-x \leq 3$ and $\cal{P}_+(\delta)$ is empty if $x > 3$. Hence $\cal{P}_+(\delta) = [0:3-x]$ where $0 \leq x \leq 3$. 

Therefore \[\begin{split} \vol{\delta} & = 6\begin{cases}\int_{0}^{3-x}{2\lambda^2 \ \mathrm{d}\lambda} & 2 \leq x \leq 3 \\ \int_{0}^{1}{2\lambda^2 \ \mathrm{d}\lambda} + \int_{1}^{3-x}{2\lambda \ \mathrm{d}\lambda} & 1 \leq x \leq 2 \\ \int_{0}^{1}{2\lambda^2 \ \mathrm{d}\lambda} + \int_{1}^{2}{2\lambda \ \mathrm{d}\lambda} + \int_{2}^{3-x}{2\lambda(3-\lambda) \ \mathrm{d}\lambda} & 0 \leq x \leq 1 \end{cases} \\ & = \begin{cases} 4(3-x)^3 & 2 \leq x \leq 3 \\ 6x^2-36x+52 & 1 \leq x \leq 2 \\ 4x^3 - 18x^2 + 36 & 0 \leq x \leq 1\end{cases}\end{split}\] giving \[\begin{split} \beta(F) & = 72 - \int_{0}^{3}{ \vol{\delta} \ \mathrm{d}x} \\ & = 72 - \int_{0}^{1}{4x^3-18x^2+36 \ \mathrm{d}x} - \int_{1}^{2}{6x^2-36x+52 \ \mathrm{d}x} - \int_{2}^{3}{4(3-x)^3 \ \mathrm{d}x} \\ & = 72 - 31 - 12 - 1 = 28.\end{split}\] Hence $X$ is $K$-polystable.

\subsection{The Divisor $W$ on $\proj{2} \times \proj{2}$ and Its Blow-Up}

\subsubsection{Central Divisor}

Let $W = \cal{Z}(x_0y_2-2x_1y_1+x_2y_0) \sub \proj{2} \times \proj{2}$ . We know that $W$ has $G$-invariant divisors $E_\infty = \cal{Z}(x_0x_2-x_1^2) \inter W$ and $E_0 = \cal{Z}(y_0y_2-y_1^2) \inter W$ whose intersection is a $G$-stable curve $Z$. We obtain the smooth Fano (3.13) by blowing up $W$ along $Z$. The curve $Z$ has a minimal $B$-chart $U = W\setminus{(\cal{Z}(x_2)\union \cal{Z}(y_2))} = \cal{Z}(x_0-2x_1y_1+y_0) \sub \aff{4}$. We eliminate $x_0$ to obtain $U = \Spec{\bbk[x_1,y_0,y_1]} \cong \aff{3}$. Introducing new co-ordinates $x = x_1-y_1$, $y=y_0$, $z=y_1$, the curve $Z \inter U$ is defined by $x = y-z^2 = 0$. The divisors $E_\infty \inter U$ and $E_0 \inter U$ are defined by $z^2 - y - x^2 = 0$ and $y-z^2 = 0$, respectively.

Hence blowing up $U$ along $Z \inter U$ we obtain $X = \cal{Z}(vx-u(y-z^2)) \sub \aff{3} \times \proj{1}$ with exceptional divisor $E = \cal{Z}(x,y-z^2)$ and strict transforms (abusing notation) $\tilde{E}_\infty = \cal{Z}(ux-v,z^2-x^2-y)$ and $\tilde{E}_0 = \cal{Z}(y-z^2,v)$. The intersection of these three $G$-invariant divisors (in fact any two of them) gives a $G$-invariant curve $Y = \cal{Z}(y-z^2,x,v)$, the unique minimal $G$-germ of $X$. 

Take a chart $u = 1$ to obtain $X = \cal{Z}(vx-y+z^2)$, then eliminate $y$ so that $X = \Spec{\bbk[x,v,z]} \cong \aff{3}$. Then we have  $\tilde{E}_\infty = \cal{Z}(x-v)$, $\tilde{E}_0 = \cal{Z}(v)$, $E = \cal{Z}(x)$ and $Y = \cal{Z}(x,v)$. 

Blowing up this chart along $Y$, we obtain $\tilde{X} = \cal{Z}(wx-sv) \sub \aff{3} \times \proj{1}$. We now have $\tilde{E}_0 = \cal{Z}(v,w)$, $\tilde{E}_\infty = \cal{Z}(x-v,w-s)$, $\tilde{E} = \cal{Z}(x,s)$ and an exceptional divisor $F = \cal{Z}(x,v)$. The $B$-quotient map $W \dashrightarrow \proj{1}$ was originally given by $P \mapsto [y_2^2(x_0x_2-x_1^2):x_2^2(y_0y_2-y_1^2)]$, which on $\tilde{X}$ reduces to $P \mapsto [v-x:v]$. From this one can see that the exceptional divisor $F$ is central. 

Since $F$ is $G$-invariant we must have $\ell_F < 0$ (for it to lie in the valuation cone), and since it is a hyperplane we must then have $\ell_F = -1$. Since the coloured hyperfan of any model of type I must consist of strictly convex coloured cones, and the $G$-invariant valuations map injectively into the hyperspace, $F$ is the unique central $G$-invariant prime divisor over $W$. 

\subsubsection{$\beta_W(F)$ (2.32)}

To calculate $\beta_W(F)$, we first must calculate $-K_W$ and its pullback to the model containing $F$. Since $W$ is a hypersurface in $\proj{2} \times \proj{2}$, the adjunction formula gives $-K_W = (-K_{\proj{2} \times \proj{2}}-W)\vert_W$. The anticanonical class of $\proj{2} \times \proj{2}$ is $(3,3)$ where we identify the divisor class group with $\bb{Z} \oplus \bb{Z}$, and since $W$ has bidegree $(1,1)$ we get $-K_W = (2,2)\vert_W$. Represent the divisor class $(2,2)$ on $\proj{2} \times \proj{2}$ by $\cal{Z}(x_0x_2-x_1^2) + \cal{Z}(y_0y_2-y_1^2)$, so that the restriction to $W$ of this class is represented by $E_\infty + E_0 = -K_W$. 

After the two blow-ups, this class pulls back to $E_0 + E_\infty + 2E + 4F$, and we must calculate $\beta_W(F) = A_W(F)(-K_W)^3 - \int_{0}^{\infty}{\vol(\delta) \ \mathrm{d}x}$ where $\delta = E_0 + E_\infty + 2E + (4-x)F$. We have $(-K_W)^3 = 48$ and $A_W(F) = 3$ since $F$ is the exceptional divisor of the second of two nested blow-ups of $W$.

We have $\lambda_\delta = 0$ and $\cal{P}(\delta) = (-\infty,4-x]$. To calculate $A(\delta,\lambda)$, first note that there is no contribution at points other than $0$, $\infty$ and $-1$. At $p = 0, \infty$, we have divisors $E_p$ with $\ell = 0$, $h=1$ and $m = 1$, and $D_p$ with $\ell = 1$, $h = 2$ and $m = 0$, so the contribution in each case is $\min\{1,\frac{\lambda}{2}\}$. At $p = -1$ we have $E$ with $\ell = -1$, $h = 1$ and $m = 2$, and $D_{-1}$ with $\ell = -1$, $h = 2$ and $m = 0$, so the contribution is $\min\{2-\lambda,-\frac{\lambda}{2}\}$. Overall, we have \[A(\delta,\lambda) = \begin{cases} 4-\lambda & \lambda \geq 0 \\ \frac{\lambda}{2} & \lambda < 2.\end{cases}\] Therefore we have $\cal{P}_+(\delta) = [0,4-x]$, so $0 \leq x \leq 4$, and: \begin{align*}\vol(\delta) & = 6\int_{0}^{2-x}{2\lambda A(\delta,\lambda) \ \mathrm{d}\lambda} \\ & = \begin{cases}6\int_{0}^{2-x}{\lambda^2 \ \mathrm{d}\lambda} & 2 \leq x \leq 4 \\ 6\int_{0}^{2}{\lambda^2 \ \mathrm{d}\lambda} + 6\int_{2}^{2-x}{8\lambda - 2\lambda^2 \ \mathrm{d}\lambda} & 0 \leq x \leq 2\end{cases} \\ & = \begin{cases} 2(4-x)^3 & 2 \leq x \leq 4 \\ 4x^3-24x^2+80 & 0 \leq x \leq 2.\end{cases}\end{align*} Hence \[\beta_W(F) = 3\cdot 48 - \int_{0}^{4}{\vol(\delta)\ \mathrm{d}x} = 144-120 = 24 > 0\] so $W$ is $K$-polystable.

\subsubsection{$\beta_X(F)$ (3.13)}

We now want to calculate \[\beta_X(F) = A_X(F)(-K_X)^3 - \int_{0}^{\infty}{\vol_X{(-K_X-xF)} \ \mathrm{d}x}.\] We have $A_X(F) = 2$ since $F$ is a prime divisor on a blow-up of $X$, and $(-K_X)^3 = 30$. We have $-K_X = \mu^*(-K_W) - E$ where $\mu$ is the blow-up of $W$ in $E_0 \inter E_\infty$, which gives $-K_X = E_0 + E_\infty + E$. Under the next blow-up to the model containing $F$, this pulls back to $E_0 + E_\infty + E + 3F$, so we set $\delta = E_0 + E_\infty + E + (3-x)F$.

We have $\lambda_\delta = 0$ and $\cal{P}(\delta) = (-\infty,3-x]$. To calculate $A(\delta,\lambda)$, first note that for $p \neq 0,-1,\infty$, there is no contribution. For $p = 0, \infty$, the two divisors with $p_D = p$ are $E_p$, with $\ell = 0$, $m = 1$ and $h = 1$, and $D_p$ with $\ell = 1$, $m = 0$ and $h=2$. Hence in each case there is a contribution to $A(\delta,\lambda)$ of \[\min{\left\{1,\frac{\lambda}{2}\right\}} = \begin{cases} 1 & \lambda \geq 2 \\ \frac{\lambda}{2} & \lambda < 2 \end{cases}.\] For $p = -1$, the two divisors with $p_D = p$ are $E$ with $\ell = -1$, $m = 1$ and $h = 1$, and $D_p$ with $\ell = -1$, $m = 0$ and $h = 2$, so the contribution to $A(\delta,\lambda)$ is \[\min{\left\{1-\lambda, -\frac{\lambda}{2}\right\}} = \begin{cases} 1-\lambda & \lambda \geq 2 \\ -\frac{\lambda}{2} & \lambda < 2 \end{cases}.\] Hence we have \[A(\delta,\lambda) = \begin{cases} 3-\lambda & \lambda \geq 2 \\ \frac{\lambda}{2} & \lambda < 2\end{cases}.\] It follows that $\cal{P}_+(\delta) = [0,3-x]$, so $0 \leq x \leq 3$

Therefore \begin{align*}\vol{\delta} & = \begin{cases} 6\int_{0}^{3-x}{\lambda^2 \ \mathrm{d}\lambda} & 1 \leq x \leq 3 \\ 6\int_{0}^{2}{\lambda^2 \ \mathrm{d}\lambda} + 6\int_{2}^{3-x}{6\lambda-2\lambda^2 \ \mathrm{d}\lambda} & 0 \leq x \leq 1\end{cases} \\ & = \begin{cases} 2(3-x)^3 & 1 \leq x \leq 3 \\ 4x^3-18x^2+30 & 0 \leq x \leq 1.\end{cases}\end{align*} Hence \[\beta(F) = 60 - \int_{0}^{3}{\vol(\delta) \ \mathrm{d}x} = 60 - 33 = 27.\] So $X$ is $K$-polystable.

\subsection{Blow up of $\proj{3}$ along the Twisted Cubic}

\subsubsection{Central Divisor}

Let $G=\sl{2}$ act on $\proj{3} = \bb{P}(S^3\bbk^2)$. The twisted cubic curve \[C = \cal{Z}(x_0x_2-x_1^2,x_0x_3-x_1x_2,x_1x_3-x_2^2)\] is $G$-invariant. The prime divisor \[F = \cal{Z}(3x_1^2x_2^2-4x_1^3x_3-x_0^2x_3^2-4x_0x_2^3+6x_0x_1x_2x_3).\] is $G$-invariant and contains $C$ - indeed $F$ is the secant variety to $C$ and $C$ is the singular locus of $F$, contained with multiplicity 2.

The smooth Fano (2.27) is obtained by blowing up $\proj{3}$ along $C$. We first take the open chart given by $x_3 = 1$, which is the minimal $B$-chart of $C$. In this chart, \[C = \cal{Z}(x_0x_2-x_1^2,x_0-x_1x_2,x_1-x_2^2) = \cal{Z}(x_0-x_1x_2,x_1-x_2^2),\] and \[F = \cal{Z}(3x_1^2x_2^2-4x_1^3-x_0^2-4x_0x_2^3+6x_0x_1x_2)\] Now consider the change of co-ordinates \[(x_0, x_1, x_2) \mapsto (x_0+3x_1x_2+x_2^3, x_1+x_2^2, x_2).\] It is easily checked to be an isomorphism, and it sends $F$ to $\cal{Z}(x_0^2-4x_1^3)$ and $C$ to $\cal{Z}(x_0,x_1)$. Hence we see that $F$ is isomorphic to the product of a line ($C$) and a cuspidal cubic plane curve. Performing another transformation $x_1 \mapsto x_1/\sqrt[3]{4}$ gives $F = \cal{Z}(x_0^2-x_1^3)$ and leaves $C$ invariant.

Now we blow up $C$, giving $X = \cal{Z}(y_1x_0-y_0x_1) \sub \aff{3} \times \proj{1}$ with exceptional divisor $E = \cal{Z}(x_0,x_1)$. We have \[\tilde{F} = \cal{Z}(x_0^2-x_1^3,y_1x_0-y_0x_1)\setminus{\cal{Z}(x_0,x_1)}.\] It is easy to check that this gives \[\tilde{F} = \cal{Z}(x_0^2-x_1^3,y_1x_0-y_0x_1,y_0x_0-x_1^2y_1,y_0^2-y_1^2x_1).\] The intersection $\tilde{F}\inter E$ is then given by $\cal{Z}(x_0,x_1,y_0)$. Since we don't yet have a central divisor, we will blow up this curve. 

First, take the chart $y_1 = 1$. Then $X$ becomes $\cal{Z}(x_0-x_1y_0) \cong \Spec{\bbk[x_1,x_2,y_0]}$, $F$ becomes $\cal{Z}(y_0^2-x_1)$, and $E$ becomes $\cal{Z}(x_1)$. Hence we obtain $\tilde{X} = \cal{Z}(z_0x_1-z_1y_0) \sub \aff{3} \times \proj{1}$, with exceptional divisor $D = \cal{Z}(x_1,y_0)$. The strict transforms of $\tilde{F}$ and $E$ are $\cal{Z}(y_0^2-x_1,z_0y_0-z_1)$ and $\cal{Z}(x_1,z_1)$, respectively. It is straightforward to check that $\tilde{E}, \tilde{F}$ and $D$ mutually intersect in the curve $\cal{Z}(x_1,y_0,z_1)$. In particular this shows that $D$ is not central, so we must blow up again.

Take the chart $z_0 = 1$, giving $\tilde{X} = \cal{Z}(x_1-z_1y_0) \cong \Spec{\bbk[x_2,y_0,z_1]}$, $\tilde{F} = \cal{Z}(y_0-z_1)$, $\tilde{E} = \cal{Z}(z_1)$ and $D = \cal{Z}(y_0)$. Blowing up the intersection $\cal{Z}(y_0,z_1)$ of these divisors gives $\tilde{X}^\prime = \cal{Z}(u_1y_0-u_0z_1)$ with exceptional divisor $H = \cal{Z}(y_0,z_1)$. Now the strict transforms $\tilde{E}$, $\tilde{F}$ and $\tilde{D}$ all intersect $H$ in different curves and are disjoint from each other: hence $H$ is a central divisor.

\subsubsection{$\beta(H)$ (2.27)}

We now want to calculate $\beta(H)$. We have $(-K_X)^3 = 38$, and $A_X(H) = 3$ since $H$ is a prime divisor on a variety obtained by two blow-ups of $X$. Hence \[\beta(H) = 114 - \int_{0}^{\infty}{\vol_X{(-K_X-xH)}\ \mathrm{d}x} = 114 - \int_{0}^{\infty}{\vol_{\tilde{X}^\prime}{(\sigma^*(-K_X)-xH)}\ \mathrm{d}x}\] where $\sigma \colon \tilde{X}^\prime \to X$ is the birational morphism given by composing the two blow-ups described above.

To calculate $\sigma^*(-K_X)$, first note that the anticanonical class of $\proj{3}$ is the class of a prime divisor of degree $4$, so we can set $-K_{\proj{3}} = F$. Then, blowing up $C$, which is contained in $F$ with multiplicity 2, gives $-K_X = (\tilde{F}+2E) - E = \tilde{F} + E$. The pullback of this class under the blowing up of $\tilde{F} \inter E$ is then $(\tilde{F} + D) + (\tilde{E} + D) = \tilde{F} + \tilde{E} + 2D$. Finally, the second blow-up gives \[\sigma^*(-K_X) = (\tilde{F}+H) + (\tilde{E} + H) + 2(\tilde{D} + H) = \tilde{F} + \tilde{E} + 2\tilde{D} + 4H.\] 

Our next step is the calculate the volume of the divisor $\delta = \tilde{F} + \tilde{E} + 2\tilde{D} + (4-x)H.$ We have $\lambda_\delta = 0$ and $\cal{P}(\delta) = (-\infty,4-x]$  since we must have $\ell_H = -1$. Now we calculate $A(\delta,\lambda)$. The points $p \neq -4, 0, \infty$ contribute nothing as $p_D = p$ in this case only for colours $D_p$ with $m_D = \ell_D = 0$. 

At $p = -4$, we have two divisors: $\tilde{F}$, with $m = 1, \ell = 0$ and $h = 1$, and $D_{-4}$, with $m = 0$, $\ell = 1$ and $h = 2$. Hence there is a contribution of \[\min\left\{1,\frac{\lambda}{2}\right\} = \begin{cases} 1 & \lambda \geq 2 \\ \frac{\lambda}{2} & \lambda < 2\end{cases}.\]

At $p = 0$, we have two divisors: $\tilde{D}$, with $m = 2$, $\ell = 0$ and $h = 1$, and $D_0$, with $m = 0$, $\ell = 1$ and $h = 2$. Hence there is a contribution of \[\min\left\{2,\frac{\lambda}{2}\right\} = \begin{cases} 2 & \lambda \geq 4 \\ \frac{\lambda}{2} & \lambda < 4\end{cases}.\]

At $p = \infty$, we have two divisors: $\tilde{E}$, with $m = 1$, $\ell = -1$ and $h = 1$, and $D_\infty$, with $m = 0$, $\ell = -2$ and $h = 3$. Hence there is a contribution of \[\min\left\{1-\lambda, -\frac{2\lambda}{3}\right\} = \begin{cases} 1 - \lambda & \lambda \geq 3 \\ -\frac{2\lambda}{3} & \lambda < 3\end{cases}.\]

All in all, we have \[A(\delta,\lambda) = \begin{cases} 4-\lambda & \lambda \geq 4 \\ 2-\frac{\lambda}{2} & 3 \leq \lambda < 4 \\ 1-\frac{\lambda}{6} & 2 \leq \lambda < 3 \\ \frac{\lambda}{3} & \lambda < 2\end{cases}.\] We can then read off $\cal{P}_+(\delta) = [0,4-x]$, so in particular $x \leq 4$. Hence \[\vol(\delta) = 6\int_{0}^{4-x}{2\lambda A(\delta,\lambda) \ \mathrm{d}\lambda}.\] That is, \[\begin{split}\vol(\delta) & = \begin{cases} \int_{0}^{4-x}{4\lambda^2 \ \mathrm{d}\lambda} & 2 \leq x \leq 4 \\ \int_{0}^{2}{4\lambda^2 \ \mathrm{d}\lambda} + \int_{2}^{4-x}{12\lambda - 2\lambda^2 \ \mathrm{d}\lambda} & 1 \leq x < 2 \\ \int_{0}^{2}{4\lambda^2 \ \mathrm{d}\lambda} + \int_{2}^{3}{12\lambda - 2\lambda^2 \ \mathrm{d}\lambda} + \int_{3}^{4-x}{24\lambda - 6\lambda^2 \ \mathrm{d}\lambda} & 0 \leq x < 1\end{cases} \\ & = \begin{cases} -\frac{4}{3}(x-4)^3 & 2 \leq x \leq 4 \\ \frac{32}{3} + \frac{2}{3}(x^3-3x^2-24x+52) & 1 \leq x < 2 \\ 28 + 2(x^3-6x^2+5) & 0 \leq x < 1.\end{cases}\end{split}\] Hence we have \[ \beta(H) = 114 - \int_{0}^{4}{\vol(\delta)} = 114 - 59 = 55.\] So $X$ is $K$-polystable.

\subsection{The Quadric Threefold and Its Blow-Up}

\subsubsection{Central Divisor}
Let $Q = \cal{Z}(3x_2^2-4x_1x_3+x_0x_4) \sub \proj{4}$. The twisted quartic curve \[C = \cal{Z}(x_0x_2-x_1^2, x_0x_3-x_1x_2, x_1x_4-x_2x_3, x_2x_4-x_3^2)\] is $G$-invariant. The prime divisor \[F = \cal{Z}(4x_2^3+x_1^2x_4+x_0x_3^2-6x_1x_2x_3) \inter Q\] is $G$-invariant and contains $C$. 

The smooth Fano (2.21) is obtained by blowing up $Q$ along $C$. The minimal $B$-chart of $C$ is given by taking $x_4 = 1$. In this chart, the equation of $Q$ allows us to eliminate $x_0$, so that $Q \cong \Spec{\bbk[x_1,x_2,x_3]}$. Then $C$ is given by \[C = \cal{Z}(x_1-x_2x_3,x_2-x_3^2)\] and $F$ by \[F = \cal{Z}(4x_2^3 + x_1^2 + 4x_1x_3^3 - 3x_2^2x_3^2 - 6x_1x_2x_3)\] One can read off immediately that there is an isomorphism from this $B$-chart in $Q$ to the $B$-chart in $\proj{3}$ we took in the previous example and that this isomorphism preserves the $G$-invariant subvarieties $F$ and $C$ (again, $F$ is the secant variety of $C$). Hence finding the central divisor is identical in this case to the previous one. Hence keeping the same notation as above, we must blow $Q$ up three times, obtaining exceptional divisors $E$, $D$ and $H$, with the latter being central.

\subsubsection{$\beta_Q(H)$ (1.16)}

To calculate $\beta_Q(H)$, we must first calculate $-K_Q$ and its pullback to the model containing $H$. Since $Q$ is a hypersurface in $\proj{4}$, the adjunction formula gives $-K_Q = (-K_{\proj{4}}-Q)\vert_Q$. The anticanonical class of $\proj{4}$ is the class of a divisor of degree 5. If we represent this by $Q + \cal{Z}(4x_2^3+x_1^2x_4+x_0x_3^2-6x_1x_2x_3)$, we see that $F$ is an anticanonical divisor of $Q$. 

After the three blow-ups described above, this pulls back to $F + 2E + 3D + 6H$, so set $\delta = F + 2E + 3D + (6-x)H$. We have $\lambda_\delta = 0$ and $\cal{P}(\delta) = (\infty,6-x]$. To calculate $A(\delta,\lambda)$, first note that there is no contribution at points $p \neq 0, \infty, -4$. At $p=0$ we have two divisors: $D$ with $\ell = -1$, $m = 3$ and $h = 1$, and $D_0$ with $\ell = -1$, $m = 0$ and $h = 2$. Hence at this point there is a contribution of \[\min\left\{3-\lambda,-\frac{\lambda}{2}\right\} = \begin{cases} 3-\lambda & \lambda \geq 6 \\ -\frac{\lambda}{2} & \lambda < 6.\end{cases}\] At $p = -4$ we have $F$ with $\ell = 0$, $m = 1$, $h = 1$, and $D_{-4}$ with $\ell = 1$, $m = 0$, $h = 3$. Hence the contribution is \[\min\left\{1,\frac{\lambda}{3}\right\} = \begin{cases}1 & \lambda \geq 3 \\ \frac{\lambda}{3} & \lambda < 3.\end{cases}\] Finally, at $p = \infty$ we have $E$ with $\ell = 0$, $m = 2$, $h = 1$, and $D_\infty$ with $\ell = 1$, $m = 0$, $h = 3$. Hence the contribution is \[\min\left\{2,\frac{\lambda}{3}\right\} = \begin{cases} 2 & \lambda \geq 6 \\ \frac{\lambda}{3} & \lambda < 6.\end{cases}\] All in all, we have \[A(\delta,\lambda) = \begin{cases} 6-\lambda & \lambda \geq 6 \\ 1-\frac{\lambda}{6} & 3 \leq \lambda \leq 6 \\ \frac{\lambda}{6} & \lambda \leq 3.\end{cases}\] Hence $A(\delta,\lambda) \geq 0$ for $0 \leq \lambda \leq 6$, so we have $\cal{P}_+(\delta) = [0,6-x]$ and $0 \leq x \leq 6$.

Now we have \begin{align*}\vol(\delta) & = 6\int_{0}^{6-x}{2\lambda A(\delta,\lambda) \ \mathrm{d}\lambda} \\ & = \begin{cases} 6\int_{0}^{6-x}{\frac{\lambda^2}{3} \ \mathrm{d}\lambda} & 3 \leq x \leq 6 \\ 6\int_{0}^{3}{\frac{\lambda^2}{3} \ \mathrm{d}\lambda} + 6\int_{3}^{6-x}{2\lambda-\frac{\lambda^2}{3} \ \mathrm{d}\lambda} & 0 \leq x \leq 3\end{cases}\\ & = \begin{cases} \frac{2}{3}(6-x)^3 & 3 \leq x \leq 6 \\ \frac{2x^3}{3}-6x^2+54 & 0 \leq x \leq 3.\end{cases}\end{align*} We have $(-K_Q)^3 = 54$ and $A_Q(H) = 4$, since we reached $H$ as the final exceptional divisor after 3 nested blow-ups of $Q$. Therefore \[\beta_Q(H) = 216-\int_{0}^{6}{\vol(\delta) \ \mathrm{d}x} = 216-135 = 81 > 0\] and $Q$ is $K$-polystable.
\subsubsection{$\beta_X(H)$ (2.21)}

We now calculate $\beta_X(H)$, where $X$ is the blow-up of $Q$ in the twisted quartic $C$, i.e. the smooth Fano (2.21). We have $(-K_X)^3 = 28$, and $A_X(H) = 3$. 

Since $-K_Q = F$ and the curve $C$ has multiplicity 2 in $F$, we have $-K_X = (F+2E) - E = F + E$. Under thw two subsequent blow-ups to the model containing $H$, this pulls back to $F + E + 2D + 4H$, so we set $\delta = F + E + 2D + (4-x)H$. We have $\lambda_\delta = 0$ and $\cal{P}(\delta) = (\infty,4-x]$. Moving on to calculating $A(\delta,\lambda)$: as before, points $p \neq -4, 0, \infty$ do not contribute. 

At $p = -4$, we have two divisors: $F$, with $m = 1, \ell = 0$, $h = 1$, and $D_{-4}$, with $m = 0$, $\ell = 1$, $h = 3$. Hence there is a contribution of \[\min\left\{1,\frac{\lambda}{3}\right\} = \begin{cases} 1 & \lambda \geq 3 \\ \frac{\lambda}{3} & \lambda < 3\end{cases}.\]

At $p = \infty$, the situation is identical to that at $p = -4$, so we get the same contribution again.

At $p = 0$, we have two divisors: $D$, with $m = 2$, $\ell = -1$ and $h = 1$, and $D_0$, with $m = 0$, $\ell = -1$ and $h = 2$. Hence there is a contribution of \[\min\left\{2-\lambda,-\frac{\lambda}{2}\right\} = \begin{cases} 2-\lambda & \lambda \geq 4 \\ -\frac{\lambda}{2} & \lambda < 4\end{cases}.\]

Hence all things considered, we have \[A(\delta,\lambda) = \begin{cases} 4-\lambda & \lambda \geq 4 \\ 2-\frac{\lambda}{2} & 3 \leq \lambda < 4 \\ \frac{\lambda}{6} & \lambda < 3 \end{cases}.\] Therefore $A(\delta,\lambda) \geq 0$ for $0 \leq \lambda \leq 4$. Hence $\cal{P}_+(\delta) = [0,4-x]$ with $0 \leq x \leq 4$. We thus have \begin{align*}\vol(\delta) & = 6\int_{0}^{4-x}{2\lambda A(\delta,\lambda) \ \mathrm{d}\lambda} \\ & = \begin{cases} \int_{0}^{4-x}{2\lambda^2 \ \mathrm{d}\lambda} & 1 \leq x \leq 4 \\ \int_{0}^{3}{2\lambda^2 \ \mathrm{d}\lambda} + \int_{3}^{4-x}{24\lambda - 6\lambda^2 \ \mathrm{d}\lambda} & 0 \leq x \leq 1\end{cases} \\ & = \begin{cases} \frac{2}{3}(4-x)^3 & 1 \leq x \leq 4 \\ 18 + 2(x^3-6x^2+5) & 0 \leq x \leq 1.\end{cases}\end{align*} Hence \[\begin{split} \beta(H) & = 84 - \int_{0}^{4}{\vol(\delta) \ \mathrm{d}x} \\ & = 84 - \int_{0}^{1}{18 + 2(x^3-6x^2+5) \ \mathrm{d}x} - \int_{1}^{4}{\frac{2}{3}(4-x)^3 \ \mathrm{d}x} \\ & = 84 -\frac{49}{2} - \frac{27}{2} = 46 > 0,\end{split}\] so $X$ is $K$-polystable.

\bibliographystyle{alpha}
\bibliography{biblio} 
\end{document}